\documentclass[10pt,letterpaper,twoside]{article}

\usepackage[letterpaper,left=1in,right=1in,top=1.5in,bottom=1.5in]{geometry}

\title{Derived plethystic algebra}\author{Benjamin Antieau, Edith
H\"ubner, Dmitry Kubrak, and Joost Nuiten}\date{\today}

\usepackage[pdfstartview=FitH,
            pdfauthor={},
            pdftitle={},
            colorlinks,
            linkcolor=reference,
            citecolor=citation,
            urlcolor=e-mail,
            backref]{hyperref}

\usepackage{amsmath}
\usepackage{amscd}
\usepackage{amsbsy}
\usepackage{amssymb}
\usepackage{verbatim}
\usepackage{eufrak}
\usepackage{eucal}
\usepackage{microtype}
\usepackage{hyperref}
\usepackage{mathrsfs}
\usepackage{amsthm}
\usepackage{stmaryrd}
\usepackage{wasysym}
\usepackage{xy}
\input xy
\xyoption{all}
\usepackage{tikz}
\usetikzlibrary{matrix,arrows}
\usepackage{tikz-cd}
\usepackage{dsfont}
\usepackage[T1]{fontenc}
\usepackage{spectralsequences}
\usepackage[shortlabels]{enumitem}

\usepackage{fancyhdr}
\usepackage{makeidx}
\makeindex

\newcommand{\stackspace}{2.5}
\newcommand{\stack}[2][1cm]{\;\tikz[baseline, yshift=.65ex]%
    {\foreach \k [evaluate=\k as \r using (.5*#2+.5-\k)*\stackspace] in {1,...,#2}{%
    \ifodd\k{\draw[->](0,\r pt)--(#1,\r pt);}%
    \else{\draw[<-](0,\r pt)--(#1,\r pt);}\fi
    }}\;}

\usepackage[bbgreekl]{mathbbol}
\usepackage{amsfonts}
\DeclareSymbolFontAlphabet{\mathbb}{AMSb} 
\DeclareSymbolFontAlphabet{\mathbbl}{bbold}

\usepackage{color}
\definecolor{todo}{rgb}{1,0,0}
\definecolor{conditional}{rgb}{0,1,0}
\definecolor{e-mail}{rgb}{0,.40,.80}
\definecolor{reference}{rgb}{.20,.60,.22}
\definecolor{mrnumber}{rgb}{.80,.40,0}
\definecolor{citation}{rgb}{0,.40,.80}

\let\oldmarginpar\marginpar
\renewcommand\marginpar[1]{\-\oldmarginpar[\raggedleft\footnotesize #1]%
{\raggedright\footnotesize #1}}

\newcommand{\DK}[1]{\textcolor{brown}{#1}}

\newcommand{\fs}{\star}

\newcommand{\afrak}{\mathfrak{a}}

\newcommand{\mfrak}{\mathfrak{m}}

\newcommand{\Ascr}{\mathcal{A}}
\newcommand{\Bscr}{\mathcal{B}}
\newcommand{\Cscr}{\mathcal{C}}
\newcommand{\Dscr}{\mathcal{D}}
\newcommand{\Escr}{\mathcal{E}}
\newcommand{\Fscr}{\mathcal{F}}

\newcommand{\Iscr}{\mathcal{I}}
\newcommand{\Kscr}{\mathcal{K}}
\newcommand{\Lscr}{\mathcal{L}}
\newcommand{\Mscr}{\mathcal{M}}
\newcommand{\Nscr}{\mathcal{N}}
\newcommand{\Oscr}{\mathcal{O}}
\newcommand{\Pscr}{\mathcal{P}}
\newcommand{\Qscr}{\mathcal{Q}}

\newcommand{\Sscr}{\mathcal{S}}

\newcommand{\B}{\mathrm{B}}

\renewcommand{\c}{\mathrm{c}}
\renewcommand{\d}{\mathrm{d}}

\newcommand{\F}{\mathrm{F}}

\renewcommand{\H}{\mathrm{H}}
\newcommand{\h}{\mathrm{h}}

\newcommand{\K}{\mathrm{K}}
\renewcommand{\L}{\mathrm{L}}
\newcommand{\M}{\mathrm{M}}

\newcommand{\R}{\mathrm{R}}
\renewcommand{\r}{\mathrm{r}}

\newcommand{\s}{\mathrm{s}}

\newcommand{\bB}{\mathbf{B}}

\newcommand{\bE}{\mathbf{E}}
\newcommand{\bF}{\mathbf{F}}
\newcommand{\bG}{\mathbf{G}}

\newcommand{\bL}{\mathbf{L}}
\newcommand{\bM}{\mathbf{M}}
\newcommand{\bN}{\mathbf{N}}

\newcommand{\bQ}{\mathbf{Q}}

\newcommand{\bS}{\mathbf{S}}

\newcommand{\bZ}{\mathbf{Z}}

\renewcommand{\mathds}{\mathbbl}

\newcommand{\NN}{\mathds{N}}

\newcommand{\Cat}{\Cscr\mathrm{at}}

\newcommand{\iCat}{\Cat_\infty}
\newcommand{\Fin}{\Fscr\mathrm{in}}

\newcommand{\Spc}{\Sscr}

\newcommand{\Ab}{\Ascr\mathrm{b}}

\newcommand{\PrL}{\Pscr\mathrm{r}^{\mathrm{L}}}
\newcommand{\PrR}{\Pscr\mathrm{r}^{\mathrm{R}}}

\newcommand{\op}{\mathrm{op}}

\newcommand{\fib}{\mathrm{fib}}

\newcommand{\lex}{\mathrm{lex}}
\newcommand{\cn}{\mathrm{cn}}

\newcommand{\BMod}{\mathrm{BMod}}
\newcommand{\BM}{\mathbf{BM}}
\newcommand{\Mod}{\mathrm{Mod}}
\newcommand{\Mon}{\mathrm{Mon}}
\newcommand{\LMod}{\mathrm{LMod}}
\newcommand{\FMod}{\mathrm{FMod}}

\newcommand{\RMod}{\mathrm{RMod}}

\newcommand{\Perf}{\mathrm{Perf}}
\newcommand{\Ind}{\mathrm{Ind}}
\newcommand{\Pro}{\mathrm{Pro}}
\newcommand{\Alg}{\mathrm{Alg}}
\newcommand{\cAlg}{\mathrm{cAlg}}
\newcommand{\CAlg}{\mathrm{CAlg}}

\newcommand{\DAlg}{\mathrm{DAlg}}

\newcommand{\ev}{\mathrm{ev}}

\newcommand{\LSym}{\mathrm{LSym}}
\newcommand{\Sym}{\mathrm{Sym}}

\newcommand{\Ani}{\mathrm{Ani}}
\newcommand{\ani}{\mathrm{ani}}

\newcommand{\mon}{\mathrm{mon}}

\newcommand{\Poly}{\mathrm{Poly}}
\newcommand{\Proj}{\mathrm{Proj}}
\newcommand{\Flat}{\mathrm{Flat}}
\newcommand{\Fr}{\mathrm{Fr}}
\newcommand{\Assoc}{\mathrm{Assoc}}
\renewcommand{\part}{\mathrm{part}}
\newcommand{\rex}{\mathrm{rex}}

\newcommand{\Pleth}{\mathrm{Pleth}}
\renewcommand{\flat}{\mathrm{flat}}
\newcommand{\RE}{\mathrm{RE}}

\newcommand{\fin}{\mathrm{fin}}
\newcommand{\Bool}{\mathrm{Bool}}

\renewcommand{\Bar}{\mathrm{Bar}}
\newcommand{\Bin}{\mathrm{Bin}}

\newcommand{\can}{\mathrm{can}}

\newcommand{\perf}{\mathrm{perf}}

\newcommand{\heart}{\heartsuit}

\newcommand{\oblv}{\mathrm{oblv}}

\newcommand{\id}{\mathrm{id}}

\newcommand{\coker}{\mathrm{coker}}

\renewcommand{\geq}{\geqslant}
\renewcommand{\leq}{\leqslant}

\DeclareMathOperator{\LBin}{LBin}

\newcommand{\Mul}{\mathrm{Mul}}
\newcommand{\Map}{\mathrm{Map}}

\newcommand{\Hom}{\mathrm{Hom}}
\newcommand{\Fun}{\mathrm{Fun}}
\newcommand{\End}{\mathrm{End}}

\newcommand{\Ga}{\bG_{a}}

\DeclareMathOperator*{\colim}{colim}

\DeclareMathOperator*{\Tot}{Tot}
\newcommand{\LEq}{\mathrm{LEq}}

\DeclareMathOperator{\Spec}{Spec}

\newcommand{\we}{\simeq}
\newcommand{\iso}{\cong}

\theoremstyle{plain}
\newtheorem{theorem}{Theorem}[section]
\newtheorem*{theorem*}{Theorem}
\newtheorem{lemma}[theorem]{Lemma}

\newtheorem{proposition}[theorem]{Proposition}

\newtheorem{corollary}[theorem]{Corollary}
\newtheorem*{corollary*}{Corollary}

\theoremstyle{plain}

\theoremstyle{definition}

\newtheoremstyle{named}{}{}{\itshape}{}{\bfseries}{.}{.5em}{#1 \thmnote{#3}}
\theoremstyle{named}

\theoremstyle{definition}

\newtheorem{definition}[theorem]{Definition}
\newtheorem{warning}[theorem]{Warning}

\newtheorem{notation}[theorem]{Notation}

\newtheorem{example}[theorem]{Example}
\newtheorem*{example*}{Example}

\newtheorem*{question*}{Question}
\newtheorem{construction}[theorem]{Construction}

\newtheorem{remark}[theorem]{Remark}

\begin{document}

\maketitle

\begin{abstract}
    \noindent
    We prove comparison results about
    $\infty$-categories of derived commutative rings with extra algebraic structure.
    For example, we show that several notions of derived $\delta$-rings agree and obtain similar
    results for derived $\lambda$-rings, derived binomial rings, and derived Boolean rings. Our comparison results apply more generally to modules over a specific kind of monad on derived rings that we call \textit{derived plethory}. 
\end{abstract}

\tableofcontents

\section{Introduction}

At present, there are two basic approaches to the theory of derived commutative rings.
The first is to study simplicial cosimplicial commutative rings and invert a class of weak
equivalences. The second is to extend the animated commutative ring monad from a monad on
connective $\bZ$-modules to a monad on all of $\Mod_\bZ$ as has been done by Raksit~\cite{raksit} based on results of Brantner--Mathew~\cite{brantner-mathew}.
The category of derived commutative rings is then defined as the category of modules for the derived symmetric powers monad acting on $\Mod_\bZ$.
Comparing these approaches is somewhat subtle, but has been done by
Brantner--Campos--Nuiten~\cite{brantner-campos-nuiten} in general and in the
coconnective case in forthcoming work of Mathew--Mondal~\cite{mathew-mondal}.

This comparison completed, it is natural to study derived commutative
rings with extra algebraic structure. A natural example of such is given by derived $\delta$-rings, which can be
defined in either of the following ways:
\begin{enumerate}
    \item[(i)] by using simplicial cosimplicial $\delta$-rings, leading to
        an $\infty$-category $\s\c\CAlg^\delta[W^{-1}]$;
    \item[(ii)] by deriving the free $\delta$-ring monad, leading to $\DAlg^\delta$, the
        $\infty$-category of derived $\delta$-rings~\cite{holeman-derived};
    \item[(iii)] by encoding the notion of a derived commutative ring with a
        lift of Frobenius as a pullback of $\infty$-categories, leading to $\DAlg^\varphi$, the
        $\infty$-category of derived $\varphi$-algebras.
\end{enumerate}
A typical result of this paper is then the following.

\begin{theorem}[$\delta$-comparison]\label{thm:delta intro}
    There are equivalences $\s\c\CAlg^\delta[W^{-1}]\we\DAlg^\delta\we\DAlg^\varphi$ which are compatible with the natural forgetful functors to $\DAlg_{\bZ}$.
\end{theorem}

The three ways of defining derived $\delta$-rings admit natural counterparts for $\lambda$-rings, binomial rings and Boolean rings. 
For example, algebras over the derived $\lambda$-ring monad and the derived binomial ring monad have been studied in~\cite{antieau-spherical} and~\cite{kubrak-shuklin-zakharov}, respectively.
The analogue of (iii) for $\lambda$-rings has been introduced in~\cite{hubner_animated}
in the animated setting. 
An approach to derived binomial rings and Boolean algebras in the style of (iii) appears in~\cite{antieau-spherical}: 
here these algebras are defined in terms of fixed points for the natural action of the Adams operations on derived $\lambda$-rings and of the $p$-Frobenius on derived $\bF_p$-algebras, respectively.

Our other main comparison results are as follows. We refer to the body of the paper for precise notation.

\begin{theorem}[$\lambda$-comparison]\label{thm:lambda intro}
    There are equivalences $\s\c\CAlg^\lambda[W^{-1}]\we\DAlg^\lambda\we\DAlg^\psi$.
\end{theorem}

\begin{theorem}[Binomial comparison]\label{thm:Bin intro}
    There are equivalences $\s\c\CAlg^\Bin[W^{-1}]\we\DAlg^\Bin\we\DAlg^{\psi=1}$.
\end{theorem}

\begin{theorem}[Boolean comparison]\label{thm:Bool intro}
    There are equivalences
   $\s\c\CAlg_{\bF_p}^{\varphi=1}[W^{-1}]\we\DAlg_{\bF_p}^\Bool\we\DAlg_{\bF_p}^{\varphi=1}$.
\end{theorem}

Theorem \ref{thm:Bin intro}, which was our motivation for starting this project, provides a verification that the three models for integral homotopy types studied in~\cite{antieau-spherical,horel,kubrak-shuklin-zakharov} agree.\\

Our proofs of Theorems~\ref{thm:delta intro} to~\ref{thm:Bool intro} rely on identifying the categories obtained from (i), (ii) and (iii) with categories of modules for a special type of monad on derived commutative rings. We call these \emph{derived flat plethories}.\\

Let us fix a discrete commutative ring $k$.
In classical commutative algebra, a \textit{plethory} is defined to be a monad $T\colon \CAlg_k^{\heart}\to\CAlg_k^{\heart}$ on discrete commutative $k$-algebras which admits a right adjoint
$R^T$,~\cite{bergman-hausknecht,borger-wieland,tall-wraith}. 
One can view $R^T\cong \Spec(T(k[x]))$ as an affine ring scheme, and we will say that a plethory $T$ is \textit{flat} if the discrete $k$-algebra $T(k[x])$ is flat. An example of a flat plethory is given by the free $\delta$-ring monad, in which case $\Spec(T(k[x]))$ is the ring scheme of $p$-typical Witt vectors. Plethories admit a natural counterpart in the setting
of derived commutative $k$-algebras:
we say that a colimit-preserving monad $T\colon \DAlg_k\to \DAlg_k$ is a \textit{derived flat
plethory} if its right adjoint $R^T$ defines a flat affine ring scheme over $\Spec(k)$ and $T$ satisfies a further technical condition
which determines its behavior on coconnective objects.
We refer the reader to Definition~\ref{def:derived_plethory} for a precise formulation of the latter.
The engine propelling our work is then the following theorem.
\begin{theorem}[Derived flat plethories are unipotent flat plethories, Theorem~\ref{thm:deriving-plethories}]\label{thm:deriving-plethories-intro}
	Let $k$ be a commutative ring. Let 
    $\Pleth_k^{\heart,\flat}$ denote the category of flat plethories on $\CAlg_k^\heart$ and
    let $\Pleth^{\flat}_k$ denote the $\infty$-category of derived flat plethories on $\DAlg_k$.
    Then there exists a fully faithful functor
    $$\Pleth_k^{\flat}\rightarrow\Pleth_k^{\heart,\flat},\quad T\mapsto \pi_0T\big|_{\CAlg_k^{\heart}}.$$
   If $k$ is regular, then its essential image consists of those flat plethories for which $R^T$ is a flat unipotent ring scheme.
\end{theorem}

Theorem~\ref{thm:deriving-plethories-intro} relies on two building blocks: The first is a form of Morita theory for monads, the second is a generalization of the approach to extending functors of Brantner--Mathew~\cite{brantner-mathew}
to arbitrary base rings $k$. Together, roughly speaking, they allow us to identify sifted-colimit preserving monads on derived $k$-algebras with algebras in $\Sym_k$-bimodules in endofunctors on $\Mod_{k}^\heart$.
A central result is then the following.
\begin{theorem}\label{thm:extending-monad-intro}
Let $k$ be a commutative ring. Then there exist equivalences of $\infty$-categories
$$\Alg(\End_{\Sigma}'(\DAlg_k)) \simeq \Alg(\End_{\Sigma}'(\Mod_k))_{\LSym_k/} \simeq
    \Alg(\End_\omega^{\RE}(\Mod_k^{\flat}))_{\Sym_k/}.$$
Here, $\End_{\Sigma}'$ denotes the category of sifted-colimit preserving endofunctors which preserve coskeletal totalizations of flat $k$-modules and $\End_\omega^{\RE}$ denotes the category of filtered-colimit preserving endofunctors which satisfy the condition of being right-extendable.
\end{theorem}

In particular, these results provide a way to extend classical monads on discrete $k$-modules to $\DAlg_k$ by verifying a condition on their underlying endofunctor.

To prove our comparison results, we start with the observation that the derived $\delta$-ring monad
and similarly the other monads obtained from (ii) naturally define derived
plethories\footnote{In~\cite{holeman-derived,antieau-spherical,kubrak-shuklin-zakharov} the derived
$\delta$-ring, $\lambda$-ring, and binomial ring monads were constructed following the filtered monad approach from~\cite{raksit}.}.
In view of Theorem~\ref{thm:deriving-plethories-intro}, the comparison with the categories defined following (i) and (iii) then comes down to showing that the latter are monadic over derived rings and proving that the associated monads indeed define derived plethories. For (i), this follows from a generalization of the results from~\cite{brantner-campos-nuiten} which we provide in Section~\ref{sec:cosimplicial}. For (iii), we rely on a characterization of derived plethories in terms of coconnective faithful flatness -- a notion that has recently been introduced by Mathew and Mondal~\cite{mathew-mondal}.\\

Let us sketch the ideas which underlie the identification of the monads in Theorem~\ref{thm:delta
intro}. In the case of derived $\varphi$-algebras, monadicity over derived rings is immediate by definition. We then proceed by showing that the
derived $\varphi$-algebra monad $T^\varphi$ restricts to the monad describing animated $\delta$-rings on connective derived rings. The results from~\cite[App.~A]{bhatt-lurie-prismatization} then identify its value on $\bZ[x]$ with the free discrete
$\delta$-ring $\delta\{x\}$.
To verify the coconnective faithful flatness property, we switch perspectives from Frobenius lifts
to $p$-typical Witt vectors. This allows us to identify $T^\varphi$ with the free monad on the
pointed endofunctor on derived rings which is left adjoint
to a derived version of $W_2$. It is then a consequence of the results from Section~\ref{sec:cofree} that the underlying endofunctor of the derived $\varphi$-algebra monad is endowed with a natural exhaustive filtration by derived endofunctors $F_i^\varphi$ on $\DAlg_{\bZ}$. The ccff-condition for the individual $F_i^\varphi$ then reduces to coconnective faithful flatness of the maps $\LSym_{\bZ}(\bZ[-n]) \to \bZ$ as shown in~\cite{mathew-mondal}.\\

\paragraph{Outline.} Section~\ref{sec:morita} contains some monad-theoretic background and a
discussion of functors between categories of modules for monads.
We explain how to compute the free monad on a well-behaved pointed endofunctor in
Section~\ref{sec:cofree}. We need some basic results about categorical actions which we review in
Section~\ref{sec:categorical_actions}. 
Section~\ref{sec:deriving} proves our main result about derived plethories.
In Section~\ref{sec:cosimplicial}, we explain the approach to modeling $\infty$-categories of
derived commutative rings with extra structure using simplicial cosimplicial objects.
Finally, we provide all of the comparison results in
Section~\ref{sec:comparison}.

\paragraph{Notation.}If $k$ is a commutative ring, a derived commutative ring or an
$\bE_\infty$-ring spectrum, we let $\Mod_k$ denote the $\infty$-category of $k$-modules. If
$\pi_ik=0$ for $i<0$, the full
subcategory of those $k$-modules $M$ such that $\pi_iM=0$ for $i<0$ is denoted by $\Mod_k^\cn$; the objects in it are connective $k$-modules.
Under the same assumption, the full subcategory consisting
of those $M$ with $\pi_iM=0$ for $i\neq 0$ is denoted $\Mod_k^\heart$; we call these $k$-modules discrete.
We note that $\Mod_k^\heart \simeq \Mod_{\pi_0 k}(\Ab)$ is equivalent to the $1$-category of $\pi_0k$-modules in abelian groups.
For $k$ a commutative ring, we write $\DAlg_k$, $\Ani_k$ and $\CAlg_k^\heart$ for the categories of derived, animated and discrete commutative $k$-algebras.
We further denote by $\Poly_k$ the full subcategory of $\CAlg_k^\heart$ spanned by the finitely generated free polynomial $k$-algebras. The categories $\Ani_k$ and $\CAlg_k^\heart$ are 
equivalent to the full subcategories of $\DAlg_k$ spanned by those derived commutative $k$-algebras whose underlying $k$-module is connective, resp.~discrete.
We write $\Sym_k$ and $\LSym_k$ for the monads of $\CAlg_k^\heart$ and $\DAlg_k$ over
$\Mod_k^\heart$ and $\Mod_k$, respectively.

\paragraph{Acknowledgments.}
We thank Adam Holeman, Achim Krause, Akhil Mathew, Tasos Moulinos, Thomas Nikolaus, Arpon Raksit, and Maxime Ramzi for
helpful conversations. We also thank Elden Elmanto and Jiaqi Fu for annoying us into finishing this paper.

BA was supported by NSF grant DMS-2152235, by the Simons Foundation and by NSF grant DMS-2424139 while at SLMath during the Fall
2026 semester, by Simons Fellowship 00005925, by the Simons
Collaboration on Perfection, by Northwestern University, and by UC Berkeley.
EH was supported by the Deutsche Forschungsgemeinschaft (DFG, German Research Foundation) under Germany's Excellence Strategy EXC 2044 –390685587, Mathematics Münster: Dynamics–Geometry–Structure. JN was supported by ANR grant ANR-24-CE40-5367 (LieDG). DK thanks IH\'ES and LMO for the hospitality and pleasant work conditions during the time the paper has been written.

\section{Morita theory for monads}\label{sec:morita}

We collect some general facts about operads, monads, and their module categories.

\subsection{Refresher on $\infty$-operads}

We will lightly use the theory of $\infty$-operads from~\cite[Chap.~2]{ha}. We recall some of the main
definitions in this section. We have the
$\bE_\infty^\otimes$-operad which is the category $\bE_\infty^\otimes=\Fin^\part$ of finite sets and partially defined
functions between them. Every object of this category is isomorphic to $\langle
n\rangle=\{1,\ldots,n\}$ for some $n\geq 0$. A morphism
$f\colon\langle m\rangle\rightarrow\langle n\rangle$ defines a
graph $\Gamma_f\subseteq\langle m\rangle\times\langle n\rangle$. It is
active if the induced function $\Gamma_f\rightarrow\langle m\rangle$ is a bijection.
It is inert if the induced function $\Gamma_f\rightarrow\langle n\rangle$ is a bijection.
For $1\leq i\leq m$, let $\rho^i\colon\langle m\rangle\rightarrow\langle 1\rangle$ be the unique
inert morphism with $\rho^i(i)=1$.

\begin{definition}[$\infty$-operads]
    Given a functor $a\colon\Oscr^\otimes\rightarrow\bE_\infty^\otimes$, we let $\Oscr^\otimes_{\langle
    m\rangle}$ be the fiber over $\langle m\rangle$ and fix $\Oscr=\Oscr^\otimes_{\langle 1\rangle}$.
    Similarly, if $f\colon\langle m\rangle\rightarrow\langle n\rangle$ is a morphism in
    $\bE_\infty^\otimes$ and if $M\in\Oscr^\otimes_{\langle m\rangle}$ and $N\in\Oscr^\otimes_{\langle
    n\rangle}$, then we let $\Map_{\Oscr^\otimes}^f(M,N)$ be the fiber over $f$ of
    $\Map_{\Oscr^\otimes}(M,N)\rightarrow\Map_{\bE_\infty^\otimes}(\langle m\rangle,\langle n\rangle)$.
    Such a functor is an $\infty$-operad if
    \begin{enumerate}
        \item[(i)] for every inert morphism $f\colon\langle m\rangle\rightarrow\langle n\rangle$ and
            each $M\in\Oscr^\otimes_{\langle m\rangle}$ there is an $a$-coCartesian morphism $f'\colon
            M\rightarrow N$ lifting $f$;
        \item[(ii)] for every $f\colon\langle m\rangle\rightarrow\langle n\rangle$ and every pair
            $M\in\Oscr^\otimes_{\langle m\rangle}$ and $N\in\Oscr^\otimes_{\langle n\rangle}$ the
            choice of $a$-coCartesian morphisms $N\rightarrow N_i$ over $\rho^i\colon\langle
            n\rangle\rightarrow\langle 1\rangle$ induces equivalences
            $$\Map_{\Oscr^\otimes}^f(M,N)\rightarrow\prod_{i=1}^n\Map_{\Oscr^\otimes}^{\rho^i\circ
            f}(M,N_i);$$
        \item[(iii)] for each $n\geq 0$, the induced pushforward functors
            $\rho^i_!\colon\Oscr^\otimes_{\langle n\rangle}\rightarrow\Oscr$ induce equivalences
            $\Oscr^\otimes_{\langle n\rangle}\we\Oscr^n$.
    \end{enumerate}
    Condition (iii) is slightly different than Lurie's, but see~\cite[Rem.~2.1.1.14]{ha}.
    By (iii), the objects of $\Oscr^\otimes$ can be written as $n$-tuples $(M_1,\ldots,M_n)$, where
    each $M_i\in\Oscr$.
\end{definition}

\begin{definition}[Multilinear maps]
    Let $(M_1,\ldots,M_n)\in\Oscr^\otimes$ and let $N\in\Oscr$. We let
    $\Mul_{\Oscr^\otimes}((M_1,\ldots,M_n),N)$ be the full sub-anima of
    $\Map_{\Oscr^{\otimes}}((M_1,\ldots,M_n),N)$ on the maps whose image in
    $\bE_\infty^\otimes$ is (equivalent to) the unique active morphism $\langle
    n\rangle\rightarrow\langle 1\rangle$.
\end{definition}

\begin{definition}[Inert morphisms]
    Given an $\infty$-operad $a\colon\Oscr^\otimes\rightarrow\bE_\infty^\otimes$ and a morphism
    $f\in\Oscr^\otimes$, we say that $f$ is inert if $a(f)$ is inert and $f$ is $a$-coCartesian.
\end{definition}

\begin{definition}[Morphisms of $\infty$-operads]
    A morphism $F\colon\Oscr^\otimes\rightarrow\Oscr'^{\otimes}$ of $\infty$-operads is the data of a commutative
    triangle
    $$\xymatrix{
        \Oscr^\otimes\ar[rr]^F\ar[dr]_a&&\Oscr'^{\otimes}\ar[ld]^{a'}\\
        &\bE_\infty^\otimes
    }$$
    such that $F$ preserves inert morphisms. We will typically abuse notation and leave the
    $2$-simplex witnessing commutativity implicit.
    We let $\Alg_\Oscr(\Oscr')$ denote the $\infty$-category of morphisms
    $\Oscr^\otimes\rightarrow\Oscr'^{\otimes}$ of $\infty$-operads. It is the full subcategory of
    $\Fun_{\bE_\infty^\otimes}(\Oscr^\otimes,\Oscr'^{\otimes})$ on the functors which preserve
    inert morphisms.
\end{definition}

\begin{definition}[Fibrations of $\infty$-operads]
    A morphism $F\colon\Oscr^\otimes\rightarrow\Oscr'^{\otimes}$ of $\infty$-operads is a fibration
    of $\infty$-operads if it is a categorical fibration.
\end{definition}

\begin{example}
    If $a\colon\Oscr^\otimes\rightarrow\bE_\infty^\otimes$ is an $\infty$-operad, then $a$ itself
    is a morphism of $\infty$-operads and is moreover a fibration of $\infty$-operads;
    see~\cite[Rem.~2.1.1.13]{ha}.
\end{example}

\begin{definition}[Algebras]
    Suppose that $f\colon\Oscr'^{\otimes}\rightarrow\Oscr^\otimes$ is a morphism of $\infty$-operads and
    $a\colon\Cscr^\otimes\rightarrow\Oscr^\otimes$ is a fibration of $\infty$-operads. We let
    $\Alg_{\Oscr'/\Oscr}(\Cscr)$ be the $\infty$-category of $\infty$-operad maps
    $g\colon\Oscr'^{\otimes}\rightarrow\Cscr^\otimes$ equipped with an equivalence $a\circ g\we f$.
    In other words, it is the fiber of $\Alg_{\Oscr'}(\Cscr)\rightarrow\Alg_{\Oscr'}(\Oscr)$ over
    $f$. If $f$ is the identity on $\Oscr^\otimes$, then we write $\Alg_{/\Oscr}(\Cscr)$ for
    $\Alg_{\Oscr/\Oscr}(\Cscr)$.
\end{definition}

\begin{definition}[$\Oscr$-monoidal $\infty$-categories]
    Suppose that $a\colon\Cscr^\otimes\rightarrow\Oscr^\otimes$ is a map of operads which, as a map
    of $\infty$-categories, is a coCartesian fibration. Then, we say
    that $\Cscr^\otimes$ (or $\Cscr$) is an $\Oscr$-monoidal $\infty$-category. In this case, it is
    also a fibration of $\infty$-operads.
\end{definition}

\begin{example}[Algebras and bimodules]
    Let $\Cscr^\otimes\rightarrow\Assoc^\otimes$ be an $\Assoc$-monoidal $\infty$-category (a monoidal
    $\infty$-category, by definition).
    Lurie has shown in~\cite[Sec.~4.1]{ha} that $\Alg_{/\Assoc}(\Cscr)$ captures the correct
    higher categorical notion of associative algebra in $\Cscr$ and in~\cite[Sec.~4.3]{ha} that
    algebras for the operad $\BM^\otimes$ (using the functor $\BM^\otimes\rightarrow\Assoc^\otimes$
    of~\cite[Rem.~4.3.1.8]{ha})
    capture the correct notion of two associative algebras $A,B$ in $\Cscr$ together with an $(A,B)$-bimodule $M$.
\end{example}

\begin{example}[Full sub-operads]
    Let $a\colon\Cscr^\otimes\rightarrow\bE_\infty^\otimes$ be an $\infty$-operad. Fix a full
    subcategory $\Dscr\subseteq\Cscr$ and let $\Dscr^\otimes\subseteq\Cscr^\otimes$ be the full
    subcategory on the $n$-tuples of objects of $\Dscr$. Then, the composition
    $\Dscr^\otimes\rightarrow\bE_\infty^\otimes$ exhibits $\Dscr^\otimes$ as an $\infty$-operad and
    $\Dscr^\otimes\rightarrow\Cscr^\otimes$ is an $\infty$-operad map. We will say that
    $\Dscr^\otimes$ is the full sub-operad generated by $\Dscr\subseteq\Cscr$.
    Conversely, if $\Escr^\otimes\subseteq\Cscr^\otimes$ is a full subcategory and the composition
    $\Escr^\otimes\subseteq\Cscr^\otimes\xrightarrow{a}\bE_\infty^\otimes$ exhibits $\Escr^\otimes$ as an $\infty$-operad, then it is the
    full sub-operad generated by $\Escr\subseteq\Cscr$.
\end{example}

\begin{lemma}[Recognition criterion]\label{lem:fully_faithful_operad_map}
    Suppose that $F\colon\Dscr^\otimes\rightarrow\Cscr^\otimes$ is a map of $\infty$-operads.
    Then, $F$ is fully faithful if and only if for all $n\geq 0$ and all $M_1,\ldots,M_n,N\in\Dscr$ the induced map
    $$\Mul_{\Dscr^\otimes}((M_1,\ldots,M_n),N)\rightarrow\Mul_{\Cscr^\otimes}((F(M_1),\ldots,F(M_n),F(N))$$
    is an equivalence.
\end{lemma}

\begin{proof}
    Condition (ii) of the definition of an $\infty$-operad immediately reduces the problem to
    showing that
    $$\Map_{\Dscr^\otimes}((M_1,\ldots,M_n),N)\rightarrow\Map_{\Cscr^\otimes}((F(M_1),\ldots,F(M_n)),F(N))$$
    is an equivalence for all $n\geq 0$ and all $M_1,\ldots,M_n,N\in\Dscr$. Suppose that 
    $f\colon\langle n\rangle\rightarrow\langle 1\rangle$ is a morphism. It is enough to show that
    $$\Map_{\Dscr^\otimes}^f((M_1,\ldots,M_n),N)\rightarrow\Map_{\Cscr^\otimes}^f((F(M_1),\ldots,F(M_n)),F(N))$$
    is an equivalence. But, $f=g\circ h$ where $h\colon\langle n\rangle\rightarrow\langle m\rangle$
    is inert and $g\colon\langle m\rangle\rightarrow\langle 1\rangle$ is active.
    By the coCartesian property over inert morphisms and~\cite[Prop.~2.4.4.3]{htt}, we have
    $$\Map_{\Dscr^\otimes}^f((M_1,\ldots,M_n),N)\we\Map_{\Dscr^\otimes}^g(h_!(M_1,\ldots,M_n),N)\we\Mul_{\Dscr^\otimes}(h_!(M_1,\ldots,M_n),N)$$
    and similarly for $\Cscr^\otimes$. Thus, $F$ is fully faithful under the hypothesis of the
    lemma.
\end{proof}

\begin{lemma}\label{lem:suboperads}
    Suppose that $a\colon\Cscr^\otimes\rightarrow\Oscr^\otimes$ is a fibration of $\infty$-operads.
    Let $\Dscr\subseteq\Cscr$ be a full subcategory and let $\Dscr^\otimes\subseteq\Cscr^\otimes$
    be the full sub-operad generated by $\Dscr$. Then, the composition
    $b\colon\Dscr^\otimes\subseteq\Cscr^\otimes\xrightarrow{a}\Oscr^\otimes$ is a fibration of $\infty$-operads.
\end{lemma}

\begin{proof}
    Both maps $\Cscr^\otimes\rightarrow\Oscr^\otimes$ and $\Dscr^\otimes\rightarrow\Cscr^\otimes$
    are maps of $\infty$-operads and hence so is the composite. Since $\Oscr^\otimes$ is an
    $\infty$-category and $a$ is a categorical fibration, it is an inner fibration
    by~\cite[Cor.~2.4.6.5]{htt}. Any fully faithful inclusion of $\infty$-categories
    $\Dscr^\otimes\subseteq\Cscr^\otimes$ is an inner
    fibration since whether or not an $n$-simplex of $\Cscr^\otimes$ belongs to $\Dscr^\otimes$
    depends only on its vertices. Again,
    by~\cite[Cor.~2.4.6.5]{htt}, it also follows that a fully faithful inclusion of
    $\infty$-categories is a categorical fibration. Thus, $\Dscr^\otimes\subseteq\Cscr^\otimes$ is
    a categorical fibration and, as these compose, so is $\Dscr^\otimes\rightarrow\Oscr^\otimes$.
\end{proof}

\begin{lemma}\label{lem:suboperad_algebras}
    Suppose that $a\colon\Cscr^\otimes\rightarrow\Oscr^\otimes$ is a fibration of $\infty$-operads,
    that $\Dscr^\otimes\subseteq\Cscr^\otimes$ is the full sub-operad spanned by a full subcategory
    $\Dscr\subseteq\Cscr$, and
    that $f\colon\Oscr'^\otimes\rightarrow\Oscr^\otimes$ is a map of $\infty$-operads.
    Then, the induced functor $\Alg_{\Oscr'/\Oscr}(\Dscr)\rightarrow\Alg_{\Oscr'/\Oscr}(\Cscr)$ is
    fully faithful and the essential image consists of those $\Oscr'$-algebras
    $A\colon\Oscr'^\otimes\rightarrow\Cscr^\otimes$ over $\Oscr^\otimes$ such that the induced
    functor $A_{\langle 1\rangle}\colon\Oscr'\rightarrow\Cscr$ has essential image in
    $\Dscr\subseteq\Cscr$.
\end{lemma}

\begin{proof}
    Fully faithfulness is straightforward. Thus, suppose that
    $A\colon\Oscr'^\otimes\rightarrow\Cscr^\otimes$ satisfies the condition that $A_{\langle
    1\rangle}$ takes values in $\Dscr\subseteq\Cscr$. Since $A$ preserves inert morphisms, it
    follows that it takes values in $\Dscr^\otimes$. Thus, it factors as the composition of a functor
    $A'\colon\Oscr'^\otimes\rightarrow\Dscr^\otimes$ followed by the inclusion
    $\Dscr^\otimes\subseteq\Cscr^\otimes$. To conclude, we must show that $A'$ is in fact a map of
    $\infty$-operads (the compatibility with projection to $\Oscr^\otimes$ being automatic).
    The only thing to check is that $A'$ preserves inert morphisms. However, a morphism of
    $\Dscr^\otimes$ is inert if it is inert as a morphism in $\Cscr^\otimes$, for example by the
    criterion for coCartesianness of~\cite[Prop.~2.4.1.8]{htt}.
\end{proof}

\begin{example}
    Let $\Cscr^\otimes\rightarrow\Assoc^\otimes$ be a monoidal $\infty$-category. Let $X\in\Cscr$
    be an object and let $\{X\}^\otimes\subseteq\Cscr^\otimes$ be the sub-operad spanned by $X$.
    Unless $X$ is the unit object, $\{X\}^\otimes\subseteq\Cscr^\otimes$ is not monoidal as
    $\{X\}^\otimes$ does not contain the unit object. Nevertheless,
    $\Alg_{/\Assoc}(\{X\})\subseteq\Alg_{/\Assoc}(\Cscr)$ is the full subcategory of associative
    algebra structures on $X$.
\end{example}


\subsection{Monadic bricolage}

We will recall how a monad can be conjugated by an adjunction, using the following standard lemma.

\begin{lemma}\label{lem:functor_category_adjoints}
    Let $F\colon\Cscr\rightleftarrows\Dscr\colon U$ be an adjunction between
    $\infty$-categories with left adjoint $F$ and right adjoint $U$. Let $\Escr$ be another $\infty$-category.
    \begin{enumerate}
         \item[{\em (a)}] There is an adjunction
            $F_\ast \colon\Fun(\Escr,\Cscr)\rightleftarrows\Fun(\Escr,\Dscr)\colon U_\ast$ induced by postcomposition with $F$ and $U$.
        \item[{\em (b)}] There is an adjunction
            $U^\ast \colon\Fun(\Cscr,\Escr)\rightleftarrows\Fun(\Dscr,\Escr)\colon F^\ast$ induced by precomposition with $F$ and $U$.
    \end{enumerate}
    Moreover, if $U$ is monadic, so is $U_\ast$.
\end{lemma}

\begin{proof}
	To see part (b), note that precomposition with the unit transformation $\id_\Cscr\rightarrow UF$ 
    induces a unit transformation $\id_{\Fun(\Cscr,\Escr)}\rightarrow(UF)^\ast$ in
    the sense of~\cite[Def.~5.2.2.7]{htt}. Thus, $U^\ast$ is left adjoint to $F^\ast$
    by~\cite[Prop.~5.2.2.8]{htt}. The proof of (a) is similar. 
    If $U$ is monadic, then $U_*$ is monadic since equivalences and
     realizations of $U_*$-split simplicial diagrams are detected pointwise.
\end{proof}

	Combining parts (a) and (b) of Lemma~\ref{lem:functor_category_adjoints}, 
	we obtain an adjunction between endomorphism $\infty$-categories
	$$F_*U^*\colon \End(\Cscr)\rightleftarrows\End(\Dscr)\colon U_*F^*$$
	in which the left adjoint can be described on objects by $T\mapsto F\circ T\circ U$ and the right
    adjoint is given on objects by $S\mapsto U\circ S\circ F$. 
    We need the following compatibility of this adjunction with
    the monoidal structure on endofunctors given by composition.
    See also \cite[Lem.~3.10]{brantner-campos-nuiten}.

\begin{proposition}\label{prop:intertwining}
	Let $F_i\colon\Cscr_i\rightleftarrows\Dscr_i\colon U_i$ be adjunctions for $i=1, 2$, 
	and consider the right adjoints
        $$
        U_{2*}F_2^*\colon \End(\Dscr_2) \rightarrow \End(\Cscr_2),\qquad
        U_{2*}F_1^*\colon \Fun(\Dscr_1, \Dscr_2) \rightarrow \Fun(\Cscr_1, \Cscr_2), \qquad
        U_{1*}F_1^*\colon \End(\Dscr_1)\rightarrow \End(\Cscr_1).
        $$
	Then the following assertions hold.
    \begin{enumerate}
         \item[{\em (a)}] For $i=1, 2$, the right adjoint $U_{i*}F_i^*$ has a natural lax monoidal structure, given on objects $X, Y\in \End(\Dscr_i)$ by the counit transformation
         $$
         (U_i\circ X\circ F_i)\circ (U_i\circ Y\circ F_i)\xrightarrow{U_iX(\epsilon_i)} U_i\circ (X\circ Y)\circ F_i.
         $$
        \item[{\em (b)}] The triple of right adjoints $(U_{2*}F_2^*, U_{2*}F_1^*, U_{1*}F_1^*)$ has the natural structure of a lax $\bB\bM^{\otimes}$-monoidal functor of bimodule $\infty$-categories, given on objects $A\in \End(\Dscr_1)$, $B\in \End(\Dscr_2)$ and $X\in \Fun(\Dscr_1, \Dscr_2)$ by the counit transformations
        $$
        (U_2\circ B\circ F_2)\circ (U_2\circ X\circ F_1)\xrightarrow{U_2B(\epsilon_2)} U_2\circ (B\circ X)\circ F_1 \qquad\quad (U_2\circ X\circ F_1)\circ (U_1\circ A\circ F_1)\xrightarrow{U_2X(\epsilon_1)} U_2\circ (X\circ A)\circ F_1.
		$$          
    \end{enumerate}
\end{proposition}

\begin{remark}
	By adjunction, the left adjoint $F_{i*}U_i^*\colon \End(\Cscr_i)\rightarrow \End(\Dscr_i)$
	inherits a natural oplax monoidal structure and the left adjoint
	$F_{2*}U_1^*\colon \Fun(\Cscr_1, \Cscr_2)\rightarrow \Fun(\Dscr_1, \Dscr_2)$
	inherits an oplax $\bB\bM^{\otimes}$-monoidal structure.
\end{remark}

	Before turning to the proof of Proposition~\ref{prop:intertwining}, 
	we record the following consequences.
	
\begin{corollary}\label{cor:e1}
	\begin{enumerate}
		\item[{\em (a)}] Let $F\colon \Cscr\rightleftarrows \Dscr\colon U$ be
		an adjunction between $\infty$-categories,
		and let $T\in \Alg_{\bE_1}(\End(\Cscr))$ be the associated monad.
		Then $T$ is equivalent as an $\bE_1$-algebra to the image of $\id_{\Dscr}$
		under the lax monoidal functor $U_*F^*\colon \End(\Dscr)\rightarrow \End(\Cscr)$.
		
		\item[{\em (b)}] Let $F\colon \Cscr\rightleftarrows \Dscr\colon U$ and 
		$F'\colon \Dscr \rightleftarrows \Escr\colon U'$ be adjunctions of $\infty$-categories,
		and let $T'$ and $T''$ be the monads associated to $(F', U')$ and 
		the composite adjunction $(F'F, UU')$, respectively.
    	Then, $T''\we U\circ T'\circ F$ as $\bE_1$-algebras in $\End(\Cscr)$.
    \end{enumerate}
\end{corollary}
\begin{proof}
	Applying Proposition~\ref{prop:intertwining} to $F\colon \Cscr\rightleftarrows \Dscr\colon U$
	and the identity on $\Dscr$, it follows that the pair of functors
	$U_*F^*\colon \End(\Dscr)\to \End(\Cscr)$ and $U_*\colon \End(\Dscr)\to \Fun(\Dscr, \Cscr)$
	forms a lax $\bL\bM^{\otimes}$-monoidal functor. 
	The canonical action of $\id_{\Dscr}$ on itself therefore induces an action
	of $U\circ F\in \Alg_{\bE_1}(\End(\Cscr))$ on $U\in \Fun(\Dscr, \Cscr)$,
	with underlying action map
	$$
	U(\epsilon)\colon U\circ F\circ U \to U.
	$$
	By \cite[Lem.~4.7.3.1]{ha}, this exhibits $U\circ F$ as the endomorphism algebra of $U$.
	This is precisely the definition of the monad $T$ associated to
	the adjunction \cite[Def.~4.7.3.2]{htt}, proving part (a).
	Part (b) follows because these $\bE_1$-algebras are
	the images of $\id_{\Escr}$
	under equivalent lax monoidal functors.
\end{proof}


\begin{corollary}\label{cor:T-bimodule}
	Let $F_i\colon \Cscr_i\rightleftarrows \Dscr_i\colon U_i$ be adjunctions for $i=1, 2$,
	and let $T_i$ be the associated monads on $\Cscr_i$.
	Then there is a natural functor
	$$
	\psi\colon \Fun(\Dscr_1, \Dscr_2)\rightarrow {}_{T_2}\BMod_{T_1}(\Fun(\Cscr_1, \Cscr_2))
	$$
	sending $X\colon \Dscr_1\rightarrow \Dscr_2$ to a $(T_2, T_1)$-bimodule 
	in $\Fun(\Cscr_1, \Cscr_2)$ with underlying functor $U_2\circ X\circ F_1$.
\end{corollary}

\begin{proof}
	This follows from Proposition~\ref{prop:intertwining} and 
	the fact that $T_1$ and $T_2$ are the images of the monoidal units.
\end{proof}

\begin{remark}\label{rem:T-bimodule_explicit}
	The bimodule structure on $\psi(X)$ can be made more explicit using Proposition~\ref{prop:intertwining}: 
	it is given by the natural maps
	\begin{align*}
	T_2\circ \psi(X) = (U_2\circ \id_{\Dscr_2}\circ F_2)\circ (U_2\circ X\circ F_1)
	&\xrightarrow{U_2(\epsilon_2)} U_2\circ (\id_{\Dscr_2}\circ X)\circ F_1=\psi(X),\\
	\psi(X)\circ T_1 = (U_2\circ X\circ F_1)\circ (U_1\circ \id_{\Dscr_1}\circ F_1)
	&\xrightarrow{U_2 X(\epsilon_1)} U_2\circ (X\circ \id_{\Dscr_1})\circ F_1=\psi(X).
	\end{align*}
\end{remark}
	
We now turn to the proof of Proposition~\ref{prop:intertwining},
which requires some preliminary observations.

\begin{lemma}\label{lem:pullback_localization}
	Consider a pullback square of $\infty$-categories
	$$\begin{tikzcd}
	\Cscr'\arrow[d, "F'"{swap}]\arrow[r, "G'"] & \Cscr\arrow[d, "F"]\\
	\Dscr'\arrow[r, "G"] & \Dscr.
	\end{tikzcd}$$
	If $F$ has a fully faithful right adjoint, then $F'$ has a fully faithful right adjoint as well.
\end{lemma}

\begin{proof}
	Let $R$ be the right adjoint to $F$. 
	Since $R$ is fully faithful, the adjunction $F\colon \Cscr \rightleftarrows \Dscr \colon R$ 
	can be viewed as an adjunction relative to $\Dscr$ \cite[Def.~7.3.2.2]{ha}, 
	where we consider $\Cscr$ as an $\infty$-category over $\Cscr$ via $F$. 
	The result now follows from \cite[Prop.~7.3.2.5]{ha}.
\end{proof}

\begin{lemma}\label{lem:bicart_adjunction}
	Let $F\colon \Cscr\leftrightarrows \Dscr\colon U$ and $F'\colon \Cscr'\rightleftarrows \Dscr'\colon U'$ 
	be two adjunctions, classifying biCartesian fibrations
	$p\colon \Mscr\to \Delta^1$ and $q\colon \Mscr'\to \Delta^1$.	
	Let $\Fun_{/\Delta^1}(\Mscr, \Mscr')\subseteq \Fun(\Mscr, \Mscr')$ 
	be the full subcategory of functors relative to $\Delta^1$ 
	and consider the functors
	$$\begin{tikzcd}
	\Fun(\Cscr, \Cscr') & \Fun_{/\Delta^1}(\Mscr, \Mscr')\arrow[r, "\ev_1"]\arrow[l, "\ev_0"{swap}] & \Fun(\Dscr, \Dscr')
	\end{tikzcd}$$	
	sending a functor $G\colon \Mscr\to \Mscr'$ relative to $\Delta^1$ to its restrictions 
	$$
	\ev_0(G)\colon \Cscr= p^{-1}(0)\rightarrow q^{-1}(0)= \Cscr',
	 \qquad\qquad \ev_1(G)\colon \Dscr= p^{-1}(1)\rightarrow q^{-1}(1)= \Dscr'.
	$$
	\begin{enumerate}
        \item[{\em (a)}] The functor $\ev_1$ has a fully faithful right adjoint
		 $\ev_1^\R$ and $\ev_0\circ \ev_1^\R$ is naturally equivalent to $U'_*F^*$.
	
		\item[{\em (b)}] The functor $\ev_0$ has a fully faithful left adjoint
        $\ev_0^\L$ and $\ev_1\circ\ev_0^\L$ is naturally equivalent to $F'_*U^*$.
	\end{enumerate}
\end{lemma}

\begin{proof}
	We only prove part (a), part (b) is dual.
	Let us start by describing $\Fun_{/\Delta^1}(\Mscr, \Mscr')$ more explicitly. 
	To this end, recall that there is an adjunction $\Cat_{\infty/\Delta^1}\rightleftarrows \Cat_\infty$
	with left adjoint sending $\Ascr\mapsto\Ascr\times_{\Delta^1}\{1\}$.
	The right adjoint sends $\Bscr$ to the cone $\Bscr^{\lhd}\rightarrow \Delta^1$,
	with the cone point mapping to $0$ and $\Bscr$ mapping to $1$; see for instance \cite[Tag 0234]{kerodon}.
	The $\infty$-categories $\Mscr$ and $\Mscr'$ then fit into squares in $\Cat_{\infty/\Delta^1}$
	$$\begin{tikzcd}
	\Dscr\times\{0\}\arrow[r] \arrow[d, "U"{swap}] & \Dscr\times \Delta^1\arrow[d]\\
	\Cscr\times \{0\}\arrow[r] & \Mscr
	\end{tikzcd}
	\qquad\qquad\qquad \begin{tikzcd}
	\Mscr'\arrow[r] \arrow[d]& \Dscr'^{\lhd}\arrow[d, "U^{\lhd}"]\\
	\Cscr'\times \Delta^1\arrow[r] & \Cscr'^{\lhd}.
	\end{tikzcd}$$
	The left square is a pushout by \cite[Prop.~3.2.2.7]{htt} and 
	the right square is a pullback: it is the unstraightening of the square
	in $\Fun(\Delta^{1, \op}, \Cat_{\infty})$ exhibiting $U'\colon \Dscr'\to\Cscr'$ as
	the pullback of $\id_{\Cscr'}\colon \Cscr'\rightarrow \Cscr'$ and $\Dscr'\to \ast$ over $\Cscr'\to \ast$.
	We thus obtain equivalences
	\begin{align*}
	\Fun_{/\Delta^1}(\Mscr, \Mscr') &\simeq \Fun(\Cscr, \Cscr')\times_{\Fun(\Dscr\times \{0\}, \Cscr')} \Fun_{/\Delta^1}(\Dscr\times\Delta^1, \Mscr')\\
	&\simeq \Fun(\Cscr, \Cscr')\times_{\Fun(\Dscr, \Cscr')^{\{0\}}} \Fun(\Dscr, \Cscr')^{\Delta^1}\times_{\Fun(\Dscr, \Cscr')^{\{1\}}} \Fun(\Dscr, \Dscr').
	\end{align*}
	Here the first equivalence follows from the left pushout square,	and
	the second equivalence follows from the right pullback square, the universal
	property of the cone and the equivalence 
	$$
	\Fun_{/\Delta^1}(\Dscr\times \Delta^1, \Cscr\times \Delta^1)\simeq \Fun(\Dscr, \Cscr)^{\Delta^1}.
	$$
	The functor $\ev_1$ is therefore the base change of
	$$
	\tilde{\ev}_1\colon \Fun(\Cscr, \Cscr')\times_{\Fun(\Dscr, \Cscr')^{\{0\}}} \Fun(\Dscr, \Cscr')^{\Delta^1}
	\rightarrow \Fun(\Dscr, \Cscr')^{\{1\}}.
	$$
	This admits a fully faithful right adjoint $\tilde{\ev}^{\R}_1$, 
	sending $X\colon \Dscr\to \Cscr'$ to the pair $(X\circ F, X\epsilon\colon X\circ F\circ U\rightarrow X)$.
	The counit is given by the natural equivalence
	$\tilde{\ev}_1\tilde{\ev}_1^{\R}\simeq \id$,
	and the unit is given on a pair $(Y, f\colon Y\circ U\to X)$ by
	\begin{equation}\label{eq:unit_transformation}\begin{tikzcd}
	Y\arrow[d, "g=f\circ Y(\eta)"{swap}] &[2pc] Y\circ U\arrow[r, "f"]\arrow[d, "g"{swap}] & X\arrow[d, "\id"]\\
	X\circ F & X\circ F\circ U\arrow[r, "X(\epsilon)"] & X.
	\end{tikzcd}\end{equation}
	The triangle identities follow from the triangle identities for the 
	adjunction $(F, U)$. The result now follows from Lemma~\ref{lem:pullback_localization}.	
\end{proof}

\begin{proof}[Proof of Proposition~\ref{prop:intertwining}]
	We only prove (b), part (a) being similar but easier.
    For $i=1, 2$, let $p_i\colon \Mscr_i\rightarrow\Delta^1$
    denote the biCartesian fibration classified by the adjunction
    $F_i\colon\Cscr_i\rightleftarrows\Dscr_i\colon U_i$.
    The triple of functors
    $$
    \ev_1\colon \End_{/\Delta^1}(\Mscr_2)\rightarrow \End(\Cscr_2)
    \qquad
    \ev_1\colon \Fun_{/\Delta^1}(\Mscr_1, \Mscr_2)\rightarrow \Fun(\Cscr_1, \Cscr_2)
    \qquad
    \ev_1\colon \End_{/\Delta^1}(\Mscr_1)\rightarrow \End(\Cscr_1)
    $$
    has the natural structure of a (strong) $\bB\bM^{\otimes}$-monoidal functor, and likewise for $\ev_0$.  
    The right adjoint $\ev_1^\R$ from Lemma~\ref{lem:bicart_adjunction} then
    inherits a lax $\bB\bM^{\otimes}$-monoidal structure by \cite[Cor.~7.3.2.7]{ha}. 
    It follows that the triple of composite right adjoints
    $$
    U_{j*}F_i^*\simeq \ev_0\circ \ev_1^{\R}\colon \Fun(\Dscr_i, \Dscr_j)
    \rightarrow \Fun(\Cscr_i, \Cscr_j)\qquad\qquad (1\leq i\leq j\leq 2)
    $$
    inherits a natural lax $\bB\bM^{\otimes}$-monoidal structure, as desired.
    To identify this lax $\bB\bM^{\otimes}$-monoidal structure, 
    fix $B\in \End(\Dscr_2)$ and $X\in \Fun(\Dscr_1, \Dscr_2)$.
    The lax structure map is then given by
    $$
    (U_2\circ B\circ F_2)\circ (U_2\circ X\circ F_1)\simeq \ev_0\big(\ev_1^{\R}(B)\circ \ev_1^{\R}(X)\big) \xrightarrow{\ev_0(\mu)} \ev_0(\ev_1^{\R}(BX))\simeq U_2\circ (B\circ X)\circ F_1
    $$  
    where $\mu$ is the lax structure map of the adjunction $(\ev_1, \ev_1^{\R})$. 
    Since $\ev_1$ is strong $\bB\bM^{\otimes}$-monoidal and $\ev_1^{\R}$ is fully faithful,
    the map $\mu$ is nothing but the unit of the adjunction $(\ev_1, \ev_1^{\R})$.
    Unraveling the description of this unit from \eqref{eq:unit_transformation},
    one sees that $\ev_0(\mu)$ is given by the composite
    $$
    U_2\circ B\circ F_2\circ U_2\circ X\circ F_1\xrightarrow{U_2BF_2U_2XF_1(\eta_1)} U_2\circ B\circ F_2\circ U_2\circ X\circ F_1\circ U_1\circ F_1\xrightarrow{U_2B(\epsilon_2)\circ U2BF_2U_2X(\epsilon_1)} U_2\circ B\circ X\circ F_1.
    $$
    This is equivalent to $U_2B(\epsilon_2)$ by the triangle identity for $(F_1, U_1)$.
\end{proof}

\subsection{Functors to monadic adjunctions}

We need the following description of the $\infty$-category of functors to an $\infty$-category of
left modules for a monad.

\begin{lemma}\label{lem:functors_to_mod}
	Let $\Cscr$ and $\Dscr$ be $\infty$-categories and consider the $\infty$-category 
	$\Fun(\Dscr,\Cscr)$ with its left-tensoring over $\End(\Cscr)$. 
	For any monad $T$ on $\Cscr$, there is a natural equivalence
    $$\Fun(\Dscr,\LMod_T(\Cscr))\rightarrow\LMod_T(\Fun(\Dscr,\Cscr)).$$
\end{lemma}

\begin{proof}
    Let $F\colon \Cscr \rightleftarrows \LMod_T(\Cscr)\colon U$ denote 
    the free-forgetful adjunction. Applying Corollary~\ref{cor:T-bimodule} to
    this adjunction and the identity on $\Dscr$, we obtain a commuting triangle
    $$\begin{tikzcd}
    \Fun(\Dscr, \LMod_T(\Cscr))\arrow[rd, "U_*"{swap}]\arrow[rr, "\psi"] & & \LMod_T(\Fun(\Dscr, \Cscr))\arrow[ld, "V"]\\
    & \Fun(\Dscr, \Cscr)
    \end{tikzcd}$$
    where $V$ forgets the left $T$-module structure. By Lemma~\ref{lem:functor_category_adjoints}, 
    $U_*$ is monadic with left adjoint $F_*$. 
    Likewise, $V$ is monadic, with left adjoint $F'$ taking the free left $T$-module.
    For any functor $X\colon \Dscr\to \Cscr$, we then obtain a natural map 
    $\alpha_X\colon UFX\to VF'(X)$. This is an equivalence, since the free left 
    $T$-module is given by the composition $VF'(X)\simeq TX$.
    By \cite[Cor.~4.7.3.16]{ha}, the top map $\psi$ is therefore an equivalence.
\end{proof}

In~\cite{ha}, Lurie constructs a functor
$\LMod\colon\Alg(\End(\Cscr))^\op\rightarrow(\Cat_\infty)_{/\Cscr}$ which sends a monad $T$ to
$\LMod_T(\Cscr)\rightarrow\Cscr$.
The next lemma shows that it is fully faithful as was noted in~\cite[Rem.~4.7.3.8]{ha}.
See also~\cite[Thm.~3.24]{henry-meadows}.

\begin{proposition}\label{prop:monads_and_functors}
    Let $F_i\colon \Cscr\rightleftarrows\Dscr_i\colon G_i$ be adjunctions for $i=1,2$. Let $T_i$
    denote the associated monads and assume that $G_2$ is monadic. Then, there is a natural
    equivalence
    $$\Map_{\Alg(\End(\Cscr))}(T_2,T_1)
    \we\Map_{(\Cat_\infty)_{/\Cscr}}(\Dscr_1\xrightarrow{G_1}\Cscr,\Dscr_2\xrightarrow{G_2}\Cscr).$$
\end{proposition}

\begin{proof}
    Consider the following diagram
    $$\begin{tikzcd}[column sep=2pc, row sep=2pc]
    \makebox[2em]{$\{G_1\}$}\arrow[rd] & & \makebox[2em]{$\LMod(\Fun(\Dscr_1,\Cscr))$}\arrow[rd]\arrow[ld] & & \makebox[2em]{$\{T_2\}$}\arrow[ld]\\
    & \makebox[2em]{$\Fun(\Dscr_1,\Cscr)$} & &  \makebox[2em]{$\Alg(\End(\Cscr)),$}
    \end{tikzcd}$$
    where $\LMod(\Fun(\Dscr_1,\Cscr))\rightarrow\Alg(\End(\Cscr))$ is the $\LMod$-fibration
    associated to the left-tensoring of $\Fun(\Dscr,\Cscr)$ over $\End(\Cscr)$ and
    $\LMod(\Fun(\Dscr_1,\Cscr)\rightarrow\Fun(\Dscr_1,\Cscr)$ is the forgetful functor.
    Lurie proves that the map from the pullback of the left span
    $\{G_1\}\times_{\Fun(\Dscr_1,\Cscr)}\LMod(\Fun(\Dscr_1,\Cscr))$ to
    $\Alg(\End(\Cscr))$ is a right fibration represented by $T_1$. In other words, this pullback
    is $\Alg(\End(\Cscr))_{/T_1}$. As a consequence, the total pullback is
    $\Map_{\Alg(\End(\Cscr))}(T_2,T_1)$. On the other hand, the pullback of the right span
    $\LMod(\Fun(\Dscr_1,\Cscr))\times_{\Alg(\End(\Cscr))}\{T_2\}$ is equivalent to
    $\LMod_{T_2}(\Fun(\Dscr_1,\Cscr))$, by definition of the $\LMod$ fibration.
    We showed in Lemma~\ref{lem:functors_to_mod} that this pullback is then equivalent to
    $\Fun(\Dscr_1,\LMod_{T_2}(\Cscr))\we\Fun(\Dscr_1,\Dscr_2)$. Thus, the total pullback
    $\Map_{\Alg(\End(\Cscr))}(T_2,T_1)$ fits into a diagram of $\infty$-categories
    $$\begin{tikzcd}
    \Map_{\Alg(\End(\Cscr))}(T_2,T_1)\arrow[r]\arrow[d] &  \Map(\Dscr_1,\Dscr_2)\arrow[d]\arrow[r] & \Fun(\Dscr_1,\Dscr_2)\arrow[d]\\
    \{G_1\}\arrow[r]& \Map(\Dscr_1,\Cscr)\arrow[r] & \Fun(\Dscr_1,\Cscr),
    \end{tikzcd}$$
    where the right two vertical arrows are given by composition with $G_2$
    and the outer square is a pullback. 
    Since $G_2$ is conservative, the right square is a pullback as well. 
    The left square is therefore a pullback, which identifies 
    $\Map_{\Alg(\End(\Cscr))}(T_2,T_1)$ with the desired mapping anima in $(\Cat_\infty)_{/\Cscr}$.
\end{proof}

Let $(\Cat_\infty)_{/\Cscr}^\R$ denote the full subcategory of $(\Cat_\infty)_{/\Cscr}$ on those
objects $G\colon\Dscr\rightarrow\Cscr$ where $G$ admits a left adjoint. The functor
$\LMod\colon\Alg(\End(\Cscr))^\op\rightarrow(\Cat_\infty)_{/\Cscr}$ factors through
$(\Cat_\infty)_{/\Cscr}^\R$ by~\cite[Cor.~4.2.4.8]{ha}. In particular, note that the condition of
admitting a left adjoint is automatic in the setting of condition (b) of Lurie's Barr--Beck theorem~\cite[Thm.~4.7.3.5]{ha}.

\begin{corollary}\label{cor:monads_adjunctions}
    The functor $\LMod\colon\Alg(\End(\Cscr))^\op\rightarrow(\Cat_\infty)_{/\Cscr}^\R$ 
    is fully faithful and admits a left adjoint given by taking the monad associated to an
    adjunction.
\end{corollary}

%
%

\subsection{Functors from monadic adjunctions}

The following is an elaboration of an idea of Holeman~\cite[Sec.~2.2]{holeman-derived}.

\begin{proposition}\label{prop:maps_out}
    Let $T$ be a monad on an $\infty$-category $\Cscr$ and let $\Dscr$ be an $\infty$-category with
    geometric realizations. The $\infty$-category $\Fun(\Cscr,\Dscr)$ is right-tensored over
    $\End(\Cscr)$ and there is a natural functor
    $$\Fun(\LMod_T(\Cscr),\Dscr)\rightarrow\RMod_T(\Fun(\Cscr,\Dscr))$$ which admits a fully faithful left adjoint with essential image
    the functors which commute with $U$-split geometric realizations.
\end{proposition}

\begin{proof}
    Let $F\colon \Cscr \rightleftarrows \LMod_T(\Cscr)\colon U$ denote 
    the free-forgetful adjunction. Applying Corollary~\ref{cor:T-bimodule} to
    this adjunction and the identity on $\Dscr$, we obtain a commuting triangle
    $$\begin{tikzcd}
    \Fun(\LMod_T(\Cscr), \Dscr)\arrow[rd, "F^*"{swap}]\arrow[rr, "\psi"] & & \RMod_T(\Fun(\Cscr, \Dscr))\arrow[ld, "V"]\\
    & \Fun(\Cscr, \Dscr)
    \end{tikzcd}$$
    where $V$ forgets the right $T$-module structure.
    The functor $V$ is monadic, with left adjoint $F'$ taking the free right $T$-module.
    On the other hand, $F^*$ has a right adjoint $U^*$ and preserves
    geometric realizations, since $\Dscr$ has geometric realizations.
    For any functor $X\colon \Cscr\to \Dscr$, we then obtain a natural map 
    $\alpha_X\colon X\circ U\circ F\to VF'(X)$. This is an equivalence, since the free right $T$-module
    is given by the composition $VF'(X)\simeq X\circ T$.
    By \cite[Cor.~4.7.3.16]{ha}, the functor $\psi$ therefore has a fully faithful left adjoint $\phi$.
    
    To identify the essential image of $\phi$, let us write
    $\Kscr\subseteq \Fun(\LMod_T(\Cscr),\Dscr)$
    for the full subcategory of functors which preserve $U$-split geometric realizations.
    Note that $\Kscr$ is closed under all colimits that exist
    in $\Fun(\LMod_T(\Cscr),\Dscr)$, and in particular under geometric realizations. 
    In addition, it contains each functor $A\circ U$ with $A\colon \Cscr\to \Dscr$;
    this preserves $U$-split geometric realizations as $A$, like any functor, 
    preserves split geometric realizations, which are absolute colimits. 
    Since $\phi$ sends each free right $T$-module $A\circ T$ to $A\circ U$, it follows 
    that its essential image is contained in $\Kscr$.
    It then remains to show that $\psi$ restricts to an equivalence on $\Kscr$. 
    To see this, observe that $(F^*, U^*)$ restricts to an adjunction
    $U^*\colon \Fun(\Cscr, \Dscr)\rightleftarrows \Kscr\colon F^*$,
    with the additional property that $F^*$ is conservative:
    if a natural transformation between functors preserving $U$-split geometric realizations
    is an equivalence on free $T$-modules, then it is an equivalence
    on all $T$-modules using their canonical resolutions.
    Another application of \cite[Cor.~4.7.3.16]{ha} then shows that $\psi$ restricts to an equivalence on $\Kscr$.
\end{proof}

\begin{remark}
    The functor from Proposition~\ref{prop:maps_out} is not an equivalence in general.
    Consider the functor $\K\colon\CAlg_{\bQ}^\heart\rightarrow\Mod_\bS$ given by algebraic $\K$-theory. The
    underlying right $\Sym_\bQ^\heart$-module is the constant functor 
    $\Mod_\bQ^\heart\rightarrow\Mod_\bS$ on $\K(\bQ)$. However, $\K$ is not constant.
\end{remark}

\begin{remark}
    If $T$ is a localization, one can prove Proposition~\ref{prop:maps_out} directly using the
    universal property of localizations.
    Note that in this case, every functor from the localization to $\Dscr$ commutes with $U$-split
    geometric realizations. Indeed, $U$-split geometric realizations are in fact split and hence
    are absolute.
\end{remark}

\begin{corollary}\label{cor:bimodules}
            Let $T_i$ be a monad on $\Cscr_i$ for $i=1,2$, and let $F_i\colon\Cscr_i\rightleftarrows\LMod_{T_i}(\Cscr_i)\colon U_i$ 
            denote the free-forgetful adjunctions. Consider the functor
            $$
            \psi\colon\Fun(\LMod_{T_1}(\Cscr_1),\LMod_{T_2}(\Cscr_2))\rightarrow {}_{T_2}\BMod_{T_1}(\Fun(\Cscr_1,\Cscr_2))
            $$
            sending $X$ to $U_2\circ X\circ F_1$. 
            If $\LMod_{T_2}(\Cscr_2)$ admits geometric
            realizations, then $\psi$ admits
            a fully faithful left adjoint, with essential image the full subcategory of functors
            $X\colon\LMod_{T_1}(\Cscr_1)\rightarrow\LMod_{T_2}(\Cscr_2)$ which preserve $U_1$-split geometric
            realizations.
\end{corollary}


\begin{proof}
    Note that $\RMod_{T_1}(\LMod_{T_2}(\Fun(\Cscr_1,\Cscr_2)))$ is equivalent to the
    $\infty$-category of $(T_2,T_1)$-bimodules in $\Fun(\Cscr_1,\Cscr_2)$
    by~\cite[Thm.~3.4.2.7]{ha}, so the result follows by Lemma~\ref{lem:functors_to_mod} and Proposition~\ref{prop:maps_out}.
    \end{proof}


	When $\Cscr_1=\Cscr_2=\Cscr$, Corollary~\ref{cor:bimodules} has 
	the following monoidal refinement.
	
\begin{corollary}\label{cor:bimodules_monoidal}
	Let $\Cscr$ be an $\infty$-category with geometric realizations and
	let $\End'(\Cscr)\subseteq \End(\Cscr)$ be the full subcategory of endofunctors
	preserving geometric realizations.
	If $T\in \Alg(\End'(\Cscr))$ is a monad preserving geometric realizations,
	then the functor of Corollary~\ref{cor:bimodules} restricts
	to a monoidal equivalence
	$$\End'(\LMod_T(\Cscr))\rightarrow{}_{T}\BMod_T(\End'(\Cscr)).$$
\end{corollary}

\begin{proof}
	As usual, let $F\colon \Cscr\rightleftarrows \LMod_T(\Cscr)\colon U$
	denote the free-forgetful adjunction. It follows from \cite[Prop.~4.3.3.9]{ha} that 
	$\LMod_T(\Cscr)$ admits geometric realizations and that $U$ and $F$ preserve these. 
	This implies that the functor from Corollary~\ref{cor:bimodules} indeed
	restricts to a functor $\psi\colon \End'(\LMod_T(\Cscr))\rightarrow{}_{T}\BMod_T(\End'(\Cscr))$.
	By Corollary~\ref{cor:bimodules}, $\psi$ is fully faithful. 
	Its essential image is closed under geometric realizations, since $\psi$ preserves these,
	and contains all free $T$-bimodules $T\circ X\circ T\simeq \psi(F\circ X\circ U)$.
	Consequently, $\psi$ is essentially surjective as well.
	
	For the monoidal structure, first note that the monoidal structures on $\End'(\Cscr)$
	and $\End'(\LMod_T(\Cscr))$ given by composition are compatible with geometric realizations.
	The relative tensor product then endows $\BMod_T(\End'(\Cscr))$ 
	with a natural monoidal structure by \cite[Prop.~4.4.3.12]{ha}.
	By Proposition~\ref{prop:intertwining}, we have a lax monoidal functor 
	$$
	U_*F^*\colon \End'(\LMod_T(\Cscr))\rightarrow \End'(\Cscr).
	$$
	The functor $\psi$ is the induced functor between bimodules over $\id_{\LMod_T(\Cscr)}$
	and its image in $\End'(\Cscr)$, as in Corollary \ref{cor:T-bimodule}.
	In particular, this endows $\psi$ with a lax monoidal structure.
	It remains to show that for any $X, Y\colon \LMod_T(\Cscr)\to \LMod_T(\Cscr)$, 
	the natural transformation
	$$
	\mu_{X, Y}\colon (U\circ X\circ F)\circ_T(U\circ Y\circ F)\rightarrow U\circ (X\circ Y)\circ F
	$$
	is an equivalence. Let us first consider the case where $X\simeq A\circ U$, 
	with $A\colon \Cscr\to \LMod_T(\Cscr)$.	Then the map 
	$$
	U\circ A\xrightarrow{UA(\eta)} U\circ A\circ U\circ F=\psi(A\circ U)
	$$
	exhibits $\psi(A\circ U)$ as a free right $T$-module; 
	this follows from the proof of Proposition~\ref{prop:maps_out} 
	or can be seen directly from Remark~\ref{rem:T-bimodule_explicit}.
	Using this, we obtain a commuting square
	$$\begin{tikzcd}
	(U\circ A)\circ (U\circ Y\circ F)\arrow[r, "\sim"]\arrow[d, "{UA(\eta)}"{swap}] & (U\circ A\circ U\circ F)\circ_T (U\circ Y\circ F)\arrow[d, "\mu_{AU, Y}"]\\
	(U\circ A\circ U\circ F)\circ (U\circ Y\circ F)\arrow[r, "UAU(\epsilon)"] & U\circ (A\circ U\circ Y)\circ F
	\end{tikzcd}$$
	where the bottom map is the lax monoidal structure on $U_*F^*$ from Proposition~\ref{prop:intertwining}.
	The bottom-left composite is an equivalence by the triangle identity, 
	so $\mu_{AU, Y}$ is an equivalence.
	Now, the first part of the proof shows that
	$U_*F^*\colon \End'(\LMod_T(\Cscr))\to \End'(\Cscr)$ is monadic.
	Each $X\in \End'(\LMod_T(\Cscr))$ is therefore the geometric realization of a 
	simplicial diagram of functors in the image	of the left adjoint $F_*U^*$;
	these are all of the form $A\circ U$.
	Since $\mu$ is a natural transformation between functors preserving geometric realizations, 
	it follows that $\mu_{X, Y}$ is an equivalence for all $X$ and $Y$.	
\end{proof}

\section{Free monads and cofree comonads}\label{sec:cofree}

Let us consider the following general situation:
\begin{notation}\label{not:modules for pointed endofunctors}
Let $\Cscr$ be an $\infty$-category and let $\End(\Cscr)$ be its category of endofunctors.
    \begin{enumerate}[(1)]
\item Given a map $\eta\colon \id_{\Cscr}\rightarrow F$ in $\End(\Cscr)$, let $\Mod_F(\Cscr)$ denote the pullback
\[\begin{tikzcd}
\Mod_F(\Cscr)\arrow[r] \arrow[d, "\mathrm{oblv}"{swap}] & \Fun(\Delta^2, \Cscr)\arrow[d]\\
\Cscr\arrow[r, "{(\eta, \id)}"] & \Fun(\Delta^{\{0, 1\}}, \Cscr)\times_{\Fun(\Delta^{\{0\}}, \Cscr)} \Fun(\Delta^{\{0, 2\}}, \Cscr).
\end{tikzcd}\]
\item Dually, given a map $\epsilon \colon G\rightarrow \id_{\Cscr}$ in $\End(\Cscr)$, let $\mathrm{Comod}_G(\Cscr)$ denote the pullback
\[\begin{tikzcd}
\mathrm{Comod}_G(\Cscr)\arrow[r] \arrow[d, "\mathrm{oblv}"{swap}] & \Fun(\Delta^2, \Cscr)\arrow[d]\\
\Cscr\arrow[r, "{(\epsilon, \id)}"] & \Fun(\Delta^{\{1, 2\}}, \Cscr)\times_{\Fun(\Delta^{\{0\}}, \Cscr)} \Fun(\Delta^{\{0, 2\}}, \Cscr).
\end{tikzcd}\]
\end{enumerate}
In other words, an object of $\Mod_F(\Cscr)$ is an object $X\in \Cscr$ equipped with a retraction $F(X)\to X$ of $\eta_X$ and an object of $\mathrm{Comod}_G(\Cscr)$ is an object $Y\in \Cscr$ equipped with a section $Y\rightarrow G(Y)$ of $\epsilon_Y$.
\end{notation}
\begin{remark}\label{rmk:duality sections retracts}
Let $\eta\colon \id_{\Cscr}\to F$ be a natural transformation between left adjoint functors. This induces a natural transformation between the corresponding right adjoints of the form $\epsilon\colon G\to \id_{\Cscr}$; note that $\eta$ and $\epsilon$ do not denote the unit and counit of the adjunctions. In this case, there is a natural equivalence $\Mod_F(\Cscr)\simeq \mathrm{Comod}_G(\Cscr)$ compatible with the forgetful functors, sending a retraction $\rho\colon F(X)\to X$ of $\eta$ to the adjoint map $\rho^\dagger\colon X\to G(X)$, which is a section of $\epsilon$.

To see this, it suffices to observe that there is a natural equivalence between the following two diagrams of $\infty$-categories, where the pullbacks are taken along $\eta$ and $\epsilon$:
\[\begin{tikzcd}[column sep=-0.1pc]
\Fun(\Delta^2, \Cscr)\times^\eta_{\Fun(\Delta^{\{0, 1\}}, \Cscr)} \Cscr\arrow[rd, "{(\ev_0, \ev_2)}"{swap}]\arrow[rr] & & \Fun(\Delta^{\{0, 2\}}, \Cscr)\arrow[ld,"{(\ev_0, \ev_2)}"] &[20pt] \Fun(\Delta^2, \Cscr)\times^\epsilon_{\Fun(\Delta^{\{1, 2\}}, \Cscr)} \Cscr\arrow[rd, "{(\ev_0, \ev_2)}"{swap}]\arrow[rr] & & \Fun(\Delta^{\{0, 2\}}, \Cscr)\arrow[ld,"{(\ev_0, \ev_2)}"]\\
& \Cscr\times\Cscr & & & \Cscr\times\Cscr.
\end{tikzcd}\]
Indeed, we can view both diagrams as maps of bifibrations over $\Cscr\times \Cscr$. Under straightening, these correspond to the natural transformations $\eta^*\colon \Map_{\Cscr}(F-, -)\to \Map_{\Cscr}(-, -)$ and $\epsilon_*\colon \Map_{\Cscr}(-, G-)\to \Map_{\Cscr}(-, -)$. These natural transformations are equivalent by adjunction.
\end{remark}

\begin{proposition}\label{prop:free monad}
Let $\Cscr$ be an $\infty$-category with colimits and let $\eta\colon \id\to F$ be a pointed endofunctor preserving sequential colimits. Then $\mathrm{oblv}\colon \Mod_F(\Cscr)\to \Cscr$ is a monadic right adjoint. The associated monad $T_F$ on $\Cscr$ is the free monad on $(F,\eta)$.
\end{proposition}

To prove Proposition \ref{prop:free monad}, we will construct the left adjoint to $\Mod_F(\Cscr)\to \Cscr$ by an inductive process. This requires a few preliminary definitions.
\begin{construction}\label{con:fmod}
Let $\Cscr$ be an $\infty$-category. Restriction along $s\colon \NN_{\leq} \times\Delta^1\to \NN_{\leq}; s(n, \epsilon)=n+\epsilon$ gives rise to a functor
\[\begin{tikzcd}
s^*\colon \Fun(\NN_{\leq}, \Cscr)\arrow[r] & \Fun(\Delta^1, \Fun(\NN_{\leq}, \Cscr))
\end{tikzcd}\]
that we will identify with a natural transformation $\sigma\colon \id\to \mathrm{shift}$ of endofunctors of $\Fun(\NN_{\leq}, \Cscr)$. Explicitly, this is given at an increasingly filtered object by the natural map $\sigma\colon F_{\fs}X\to \mathrm{shift}(F_{\fs}X)=F_{\fs+1}X$. 

Let us now consider the diagram
\[\begin{tikzcd}
\Cscr\arrow[d]\arrow[r, "{(\eta, \id)}"] & \Fun(\Lambda^2_0, \Cscr)\arrow[d] & \Fun(\Delta^2, \Cscr)\arrow[d]\arrow[l]\\
\Fun(\NN_{\leq}, \Cscr)\arrow[r, "{(\eta, \sigma)}"] & \Fun(\Lambda^2_0, \Fun(\NN_{\leq}, \Cscr)) & \Fun(\Delta^2, \Fun(\NN_{\leq}, \Cscr))\arrow[l]
\end{tikzcd}\]
where the vertical functors take the constant $\NN_{\leq}$-diagram on an object in $\Cscr$. Each of these vertical functors has a left adjoint, taking the colimit over $\NN_{\leq}$. Since $F$ preserves these colimits, these vertical left adjoints commute with the horizontal functors. 
Taking pullbacks of the rows therefore yields an adjoint pair that we will denote by
\[\begin{tikzcd}
\colim \colon \FMod_F(\Cscr)\arrow[r, yshift=1ex] & \Mod_F(\Cscr) \colon \mathrm{cst}\arrow[l, yshift=-1ex]
\end{tikzcd}\]
where $\FMod_F(\Cscr)$ is the $\infty$-category of objects with an increasing filtration $F_{\fs}X$, equipped with a factorization of the structure maps of its filtration as
\[\begin{tikzcd}
\sigma\colon F_{\fs}X \arrow[r, "\eta"] & F(F_{\fs}X)\arrow[r, "\rho"] &  F_{\fs+1}X.
\end{tikzcd}\]
Taking the colimit over $\NN_{\leq}$ and using that $F$ preserves sequential colimits, this indeed yields an object in $\Mod_F(\Cscr)$, that is, a retract diagram $\colim(F_{\fs}X)\xrightarrow{\eta} F(\colim F_{\fs}X)\xrightarrow{\rho} \colim(F_{\fs} X)$. 
\end{construction}
We are going to show that the functor $\FMod_F(\Cscr)\to \Cscr; F_{\fs}X\longmapsto F_0X$ admits a fully faithful left adjoint. The composition of this left adjoint with the functor taking colimits over $\NN_{\leq}$ will then provide the left adjoint from Proposition \ref{prop:free monad}. Let us start by giving a different description of the $\infty$-category $\FMod_F(\Cscr)$.
\begin{definition}
Let $\Iscr$ be the category informally depicted as follows:
\[\begin{tikzcd}
{(0, 0)} \arrow[r, "\eta"] & {(0, 1)}\arrow[r, "\rho"]\arrow[rr, "\beta", bend left=30] & {(1, 0)}\arrow[r, "\eta"] & {(1, 1)}\arrow[r, "\rho"]\arrow[rr, "\beta", bend left=30] & {(2, 0)}\arrow[r, "\eta"] & (2, 1)\arrow[r, "\rho"] \arrow[rr, "\beta", bend left=30, end anchor={[yshift=1ex]}] & {(3, 0)}\arrow[r, "\eta"] & \dots
\end{tikzcd}\]
More precisely, $\Iscr$ is the category defined as follows.
\begin{itemize}
\item The objects of $\Iscr$ are tuples $(n, 0)$ and $(n, 1)$ for $n\in\mathbb{N}$;

\item The only non-empty sets of maps between two objects are given by the following sets with at most two elements, for $m\leq n$:
\begin{align*}
\Map_{\Iscr}((m, 0), (n, 0)) &= \{(\rho\eta)^{n-m}\}  & \Map_{\Iscr}((m, 0), (n, 1)) &= \{\eta(\rho\eta)^{n-m}\} \\
 \Map_{\Iscr}((m, 1), (n+1, 0)) &= \{(\rho\eta)^{n-m-1}\rho\} & \Map_{\Iscr}((m, 1), (n, 1)) &= \{(\eta\rho)^{n-m}, \beta^{n-m}\}.
\end{align*}
\item Composition of maps is determined by word concatenation, together with the fact that $\beta \eta = \eta \rho\eta$ and $\rho\beta=\rho\eta\rho$.
\end{itemize}
Let us write $\Iscr_{\leq n}$ for the full subcategory spanned by $(k, 0)$ and $(k, 1)$ with $k\leq n$ and let $\Iscr_{<n}$ be the full subcategory on $(n, 0)$ and all $(k, 0), (k, 1)$ with $k<n$.
\end{definition}
\begin{construction}\label{con:u}
Let $u\colon \NN_{\leq}\times \Delta^2\to \Iscr$ denote the functor sending
\[
\begin{tikzcd}
(n, 0)\arrow[r]\arrow[d] & (n, 1)\arrow[r]\arrow[d] & (n, 2)\arrow[d]\\
(n+1, 0)\arrow[r] & (n+1, 1)\arrow[r] & (n+1, 2)\end{tikzcd} \quad\longmapsto\quad \begin{tikzcd} (n, 0)\arrow[r, "\eta"] \arrow[d, "\rho\eta"{swap}] & (n, 1)\arrow[r, "\rho"]\arrow[d, "\beta"] & (n+1, 0)\arrow[d, "\rho\eta"]\\
(n+1, 0)\arrow[r, "\eta"] & (n+1, 1)\arrow[r, "\rho"] & (n+2, 0)\end{tikzcd}\]
We will make use of various restrictions of $u$, all of which will simply be denoted
by $u$ as well. For example, $u$ restricts to functors
\[
u\colon \Delta^n\times \Delta^{\{0, 1\}} \to \Iscr_{\leq n},  \qquad\qquad u\colon (\Delta^n\times \Delta^{\{0, 1\}})^\circ\to \Iscr_{<n}
\]
where we write $(\Delta^n\times\Delta^{\{0, 1\}})^\circ\subseteq \Delta^n\times\Delta^{\{0, 1\}}$ for the full subcategory on all objects except $(n,1)$.
\end{construction}
\begin{lemma}\label{lem:pushout of cats}
Let $n\geq 1$ and consider the following diagrams.
\[\begin{tikzcd}
\Iscr_{\leq 1}\arrow[r, "f"]\arrow[d] & \Iscr_{\leq n}\arrow[d]\\
\Iscr_{<2}\arrow[r, "f"] & \Iscr_{<n+1}
\end{tikzcd}
\hspace{50pt} \begin{tikzcd}
(\Delta^1\times \Delta^{\{0, 1\}})^\circ\arrow[r, "u"]\arrow[d] & \Iscr_{<1}\arrow[r, "f"]\arrow[d] & \Iscr_{<n}\arrow[d]\\
\Delta^1\times \Delta^{\{0, 1\}}\arrow[r, "u"] & \Iscr_{\leq 1}\arrow[r, "f"] & \Iscr_{\leq n}.
\end{tikzcd}\]
Here the vertical functors are the natural inclusions and the horizontal functors $f$ are the fully faithful inclusions determined by $f(k, \epsilon)=(k+n-1, \epsilon)$. Then all squares appearing in the above diagrams are pushout squares of $\infty$-categories.
\end{lemma}
\begin{proof}
For the left diagram, let us observe that the inclusion $\Iscr_n\to \Iscr_{<n+1}$ is equivalent to the inclusion $\Iscr_{\leq n}\to \Iscr_{\leq n}\star \Delta^0$ into its cone. By \cite[Lem.~2.1.2.3, Prop.~4.1.1.3]{htt}, it then suffices to verify that $f\colon \Iscr_{\leq 1}\to \Iscr_{\leq n}$ is cofinal. To do this, let us fix an object $(k, \epsilon)\in \Iscr_{\leq n}$ and show that the comma category $f(\Iscr_{\leq 1})_{(k, \epsilon)/}$ is contractible. We distinguish three cases:
    \begin{enumerate}[(1)]
\item $(k, \epsilon)$ with $k\geq n-1$: in this case, $f(\Iscr_{\leq 1})_{(k, \epsilon)/}$ admits an initial object given by the identity on $(k, \epsilon)$.
\item $(k, 0)$ with $k\leq n-2$: in this case, each object of $f(\Iscr_{\leq 1})$ receives a unique map from $(k, 0)$, so that $f(\Iscr_{\leq 1})_{(k, 0)/}\simeq \Iscr_{\leq 1}$ admits an initial object, corresponding to the map $(k, 0)\to (n-1, 0)$.

\item $(k, 1)$ with $k\leq n-2$: in this case, $f(\Iscr_{\leq 1})_{(k, 1)/}$ is isomorphic to the following contractible poset, where objects are labelled by the corresponding map from $(k, 1)$:
\[\begin{tikzcd}
& \beta^{n-k-1}\arrow[r, "\beta\circ(-)"] \arrow[rd, "\rho\circ(-)"] & \beta^{n-k}\\
(\rho\eta)^{n-k-2}\rho\arrow[r, "\eta\circ(-)"] & \eta(\rho\eta)^{n-k-2}\rho\arrow[r, "\rho\circ(-)"] & (\rho\eta)^{n-k-1}\rho\arrow[r, "\eta\circ(-)"] & \eta(\rho\eta)^{n-k-1}\rho
\end{tikzcd}\]
\end{enumerate}
For the right diagram, it suffices to show that the outer rectangle is a pushout (which implies that the left square is as well). Let us write $\Pscr$ for the pushout of $\Delta^1\times \Delta^{\{0, 1\}}\leftarrow (\Delta^1\times\Delta^{\{0, 1\}})^\circ\to \Iscr{<n}$. We will explicitly compute this pushout using the model for $\infty$-categories in terms of simplicial categories. In this model structure, we can factor
\[\begin{tikzcd}
(\Delta^1\times \Delta^{\{0, 1\}})^\circ\arrow[r, >->] & \Qscr\arrow[r, "\sim"] & \Delta^1\times \Delta^{\{0, 1\}}
\end{tikzcd}\]
into a cofibration followed by a weak equivalence, where $\Qscr$ is the simplicially enriched category with objects $(i, j)$ for $0\leq i, j\leq 1$ and whose non-empty mapping spaces are of the form
\[
\Map_{\Qscr}((i, j), (i, j))= \Map_{\Qscr}((i, j), (i+1, j))=\Map_{\Qscr}((i, j), (i, j+1))=\ast \qquad\qquad \Map_{\Qscr}((0, 0), (1, 1))=\Delta^1.
\]
Since the model structure on simplicial categories is left proper, we can model $\Pscr= \Qscr\amalg_{(\Delta^1\times \Delta^{\{0, 1\}})^\circ} \Iscr_{<n}$. Unravelling the definitions, one sees that this pushout freely adds a morphism $\eta\colon (n, 0)\to (n, 1)$, a morphism $\beta\colon (n-1, 1)\to (n, 1)$ and a $1$-simplex between the map $\beta \eta$ and $\eta\rho\eta$ from $(n-1, 0)$ to $(n, 1)$.
It follows that $\Pscr$ contains $\Iscr_{<n}$ as a full subcategory and has other non-trivial mapping spaces given by
\begin{align*}
\Map_{\Pscr}((n, 0), (n, 1)) &=\{\eta\} & \Map_{\Pscr}((n-1, 1), (n, 1)) &= \{\beta, \eta\rho\}\\
\Map_{\Pscr}((k, 0), (n, 1)) &= \Delta^1  & \Map_{\Pscr}((k, 1), (n, 1)) &= \Delta^1\amalg \Delta^1
\end{align*}
for $k<n$. The map $\Pscr\to \Iscr_{\leq n}$ is the equivalence contracting all of the above $1$-simplices.
\end{proof}

\begin{lemma}\label{lem:link with sequential retracts}
Let $n\geq 0$ and consider the square of categories
\[\begin{tikzcd}
\Delta^n\times \Delta^{\{0, 2\}}\arrow[r]\arrow[d, "s"{swap}] & \Delta^n\times \Delta^2\arrow[d, "u"]\\
\Delta^{n+1}\times \Delta^{\{0\}}\arrow[r, "u"] & \Iscr_{<n+1}
\end{tikzcd}\]
where the top functor is the evident inclusion, the left functor sends $(k, 0)\mapsto (k, 0)$ and $(k, 2)\mapsto (k+1, 0)$ and the bottom and right functors are both restrictions of the functor $u$ from Construction \ref{con:u}.
This square is a pushout of $\infty$-categories.
\end{lemma}
\begin{proof}
Notice that the square is the level $n$ component of a square of cosimplicial diagrams of ($\infty$-)categories, where the bottom left corner is given by $\Delta^\bullet\star \Delta^0$. Each of the four corners satisfies the co-Segal condition, where we use Lemma \ref{lem:pushout of cats} for the bottom right corner. It therefore remains to treat the case $n=1$.

In this case, we will compute the homotopy pushout in the Joyal model structure. Since the map $\Delta^n\times\Delta^{\{0, 2\}}\to \Delta^n\times\Delta^2$ is injective, the homotopy pushout is simply given by the pushout $P$ of simplicial sets. Unravelling the definitions, one sees that the map of simplicial sets $P\to N(\Iscr_{<2})$ is injective. The non-degenerate simplices of $N(\Iscr_{<2})$ that are not in the image of $P$ are
\begin{align*}
& (0, 1)\xrightarrow{\eta\rho} (1, 1)  & & (0, 1)\xrightarrow{\rho} (1, 0)\xrightarrow{\eta} (1, 1) & & (0, 0)\xrightarrow{\eta} (0, 1)\xrightarrow{\rho} (1, 0)\xrightarrow{\eta} (1, 1)\\
& & & (0, 0)\xrightarrow{\eta} (0, 1)\xrightarrow{\eta\rho} (1, 1) & & (0, 1)\xrightarrow{\rho} (1, 0)\xrightarrow{\eta} (1, 1)\xrightarrow{\rho} (2, 0)\\
& & & (0, 1)\xrightarrow{\eta\rho} (1, 1) \xrightarrow{\rho} (2, 0) & & (0, 0)\xrightarrow{\eta} (0, 1)\xrightarrow{\eta\rho} (1, 1)\xrightarrow{\rho} (2, 0)
\end{align*}
and the $4$-simplex $(0, 0)\xrightarrow{\eta} (0, 1)\xrightarrow{\rho} (1, 0)\xrightarrow{\eta} (1, 1)\xrightarrow{\rho} (2, 0)$. It follows that $P\to N(\Iscr_{<2})$ is inner anodyne: we can first attach the $2$-simplex $\eta\circ \rho$ along an inner horn inclusion, then the $3$-simplices $\eta\circ \rho\circ \eta$ and $\rho\circ\eta\circ \rho$ and finally the $4$-simplex $\rho\circ\eta\circ\rho\circ \eta$.
\end{proof}
\begin{corollary}\label{cor:FMod}
The $\infty$-category $\FMod_F(\Cscr)$ fits into a pullback square of $\infty$-categories as follows.
\[\begin{tikzcd}[column sep=3pc]
\FMod_F(\Cscr)\arrow[r]\arrow[d, "\mathrm{oblv}"{swap}] & \Fun(\Iscr, \Cscr)\arrow[d, "u^*"]\\
\Fun(\mathbb{N}_{\leq}, \Cscr)\arrow[r, "\eta\colon \id\to F"] & \Fun(\mathbb{N}_{\leq}\times \Delta^{\{0, 1\}}, \Cscr)
\end{tikzcd}\]
\end{corollary}
\begin{proof}
Consider the following two squares of $\infty$-categories:
\[\begin{tikzcd}
\NN_{\leq}\times \Delta^{\{0, 2\}}\arrow[r]\arrow[d, "s"{swap}] & \NN_{\leq}\times \Delta^2\arrow[d, "u"]\\
\NN_{\leq}\times \{0\}\arrow[r] & \Iscr
\end{tikzcd}\hspace{30pt} \begin{tikzcd}
\Fun(\Iscr, \Cscr)\arrow[d, "u^*"{swap}]\arrow[r] & \Fun(\NN_{\leq}, \Cscr) \arrow[d, "\sigma"]\\
\Fun(\NN_{\leq}\times \Delta^2, \Cscr)\arrow[r] & \Fun(\NN_{\leq}\times \Delta^{\{0, 2\}}, \Cscr)
\end{tikzcd}\]
Here the right square is obtained by restriction along the functors in the left square. The left square is a pushout square of $\infty$-categories by Lemma \ref{lem:link with sequential retracts} and taking the colimit over all $n$. It follows that the right square is a pullback square. The result now follows by unravelling Construction \ref{con:fmod}.
\end{proof}
We can use Corollary \ref{cor:FMod} to identify $\FMod_F(\Cscr)$ with the limit of a tower
\[\begin{tikzcd}[column sep=1pc]
\FMod_F(\Cscr)\arrow[r] & \dots\arrow[r] & \FMod_F(\Cscr)_{\leq n}\arrow[r] & \FMod_F(\Cscr)_{<n}\arrow[r] & \FMod_F(\Cscr)_{\leq n-1}\arrow[r] & \dots\arrow[r] & \FMod_F(\Cscr)_{<0}\simeq \Cscr,
\end{tikzcd}\]
whose finite stages are given by the following pullbacks of $\infty$-categories:
\[\begin{tikzcd}
\FMod_F(\Cscr)_{\leq n}\arrow[r] \arrow[d] & \Fun(\Delta^n, \Cscr)\arrow[d, "\eta\colon \id\to F" ] &[20pt] \FMod_F(\Cscr)_{< n}\arrow[r]\arrow[d] & \Fun(\Delta^n, \Cscr)\arrow[d, "\eta\colon \id\to F"]\\
\Fun(\Iscr_{\leq n}, \Cscr)\arrow[r, "u^*"] & \Fun(\Delta^n\times\Delta^{\{0, 1\}}, \Cscr) & \Fun(\Iscr_{< n}, \Cscr)\arrow[r, "u^*"] & \Fun((\Delta^n\times\Delta^{\{0, 1\}})^{\circ}, \Cscr).
\end{tikzcd}\]
\begin{lemma}\label{lem:iterated LKE}
Let $\Cscr$ be an $\infty$-category with finite colimits and $\eta\colon \id\to F$ a pointed endofunctor of $\Cscr$.
    \begin{enumerate}[{\em (1)}]
\item For each $n\geq 1$, the functor $\Mod_F(\Cscr)_{\leq n}\to \Mod_F(\Cscr)_{<n}$ is an equivalence.

\item For each $n\geq 1$, consider the diagram
\begin{equation}\label{diag:filtered modules}\begin{tikzcd}
\FMod_{F}(\Cscr)_{<n+1}\arrow[d]\arrow[r] & \Fun(\Iscr_{<n+1}, \Cscr)\arrow[r]\arrow[d] & \Fun(\Iscr_{<2}, \Cscr)\arrow[d]\\
\FMod_F(\Cscr)_{\leq n}\arrow[r] & \Fun(\Iscr_{\leq n}, \Cscr)\arrow[r] & \Fun(\Iscr_{\leq 1}, \Cscr).
\end{tikzcd}\end{equation}
Each square is a pullback square and the vertical functors admit fully faithful left adjoints.
\end{enumerate}
\end{lemma}
\begin{proof}
For the first part, consider the diagram
\[\begin{tikzcd}
(\Delta^{\{n-1, n\}}\times\Delta^1)^\circ\arrow[r]\arrow[d] & (\Delta^n\times \Delta^1)^\circ\arrow[r, "u"]\arrow[d] & \Iscr_{<n}\arrow[d]\\
\Delta^{\{n-1, n\}}\times\Delta^1\arrow[r] & \Delta^n\times \Delta^1\arrow[r, "u"] & \Iscr_{\leq n}.
\end{tikzcd}\]
The total rectangle is a pushout by Lemma \ref{lem:pushout of cats}.
The left square is a pushout as well: 
its bottom row is obtained from the top row by taking cones and the inclusion 
$(\Delta^{\{n-1, n\}}\times\Delta^1)^\circ\hookrightarrow (\Delta^n\times \Delta^1)^\circ$ 
is cofinal (similarly to the proof of Lemma \ref{lem:pushout of cats}). 
It follows that the right square is a pushout. Unwinding the definitions, this implies that 
$\Mod_F(\Cscr)_{\leq n}\to \Mod_F(\Cscr)_{<n}$ is an equivalence.

For the second statement, observe that the right square in \eqref{diag:filtered modules}
is a pullback by Lemma \ref{lem:pushout of cats}.
For the left square, consider the diagram
\[\begin{tikzcd}
\Fun(\Delta^{n+1}, \Cscr)\arrow[d]\arrow[r, "\eta\colon \id\to F"] &[2pc]
\Fun((\Delta^{n+1}\times\Delta^1)^\circ, \Cscr)\arrow[r]\arrow[d] &
\Fun(\Delta^{n+1}\times \{0\}, \Cscr)\arrow[d]\\
\Fun(\Delta^n, \Cscr)\arrow[r, "\eta\colon \id\to F"] &
\Fun(\Delta^{n}\times \Delta^1, \Cscr)\arrow[r] &
\Fun(\Delta^{n}\times \{0\}, \Cscr).
\end{tikzcd}\]
The total square and the right square are pullbacks, 
so that the left square is a pullback as well.
Unravelling the definition of $\FMod_F(\Cscr)_{<n+1}$ and $\FMod_F(\Cscr)_{\leq n}$, 
one then sees that the left square in \eqref{diag:filtered modules} 
is a pullback as well. 
Finally, the right vertical functor admits a fully faithful left adjoint
given by left Kan extension, and this property is stable under base change.
\end{proof}
\begin{proof}[Proof of Proposition \ref{prop:free monad}]
The forgetful functor can be identified with the composite
\[\begin{tikzcd}
\mathrm{oblv}\colon \Mod_F(\Cscr)\arrow[r, "\mathrm{cst}"] & \FMod_F(\Cscr)\arrow[r] & \FMod_F(\Cscr)_{<0} \simeq \Cscr.
\end{tikzcd}\]
The first functor admits a left adjoint by Construction \ref{con:fmod}. 
The second functor admits a fully faithful left adjoint by Lemma \ref{lem:iterated LKE}. 
Indeed, we can decompose $\FMod_F(\Cscr)\to\Cscr$ into a tower where almost all functors are equivalences or admit a fully faithful left adjoint by Lemma \ref{lem:iterated LKE}. 
It remains to examine the lowest stages $\FMod_F(\Cscr)_{<1}\to \FMod_F(\Cscr)_{\leq 0}\to \FMod_F(\Cscr)_{<0}$ of the tower. 
These can be identified with
\[\begin{tikzcd}
\Fun(\Delta^2, \Cscr)\times_{\Fun(\Delta^{\{0, 1\}}, \Cscr)}^\eta \Cscr\arrow[r] &
\Fun(\Delta^{\{0, 1\}}, \Cscr)\times^\eta_{\Fun(\Delta^{\{0, 1\}}, \Cscr)} \Cscr\arrow[r, "\sim"] &
\Cscr
\end{tikzcd}\]
where the pullbacks are taken along $\eta\colon \Cscr\to \Fun(\Delta^{\{0, 1\}}, \Cscr)$. The second functor is an equivalence and the first admits a fully faithful left adjoint, sending $F_0X\xrightarrow{\eta} F(F_0X)$ to $F_0X\xrightarrow{\eta} F(F_0X) = F(F_0(X))$. 

It now remains to verify that the right adjoint $\mathrm{oblv}\colon \Mod_F(\Cscr)\to \Cscr$ is monadic. It is clearly conservative. Given a simplicial diagram of $F$-modules $X_\bullet\to F(X_\bullet)\to X_\bullet$ with a split colimit $X_\bullet\to X_{-1}$ in $\Cscr$, we find that $F(X_\bullet)\to F(X_{-1})$ is a split colimit as well. The natural transformation $F(X_\bullet)\to X_\bullet\to X_{-1}$ then induces a retract diagram $X_{-1}\to F(X_{-1})\to X_{-1}$. This yields an augmented simplicial diagram $\Delta^{\op}_+\to \Mod_F(\Cscr)$ and (the proof of) \cite[Cor.~4.7.2.13]{ha} implies that this diagram is split. In particular, this implies that it is a colimit that is preserved by the forgetful functor. The assertion that the monad associated to
$\mathrm{oblv}\colon \Mod_F(\Cscr)\to \Cscr$ is the free monad on the pointed endofunctor $F$ is an immediate consequence of Corollary~\ref{cor:recognition free monads} below.
\end{proof}
\begin{remark}\label{rem:free monad explicit}
Let $T\colon \Cscr\to \Cscr$ denote the monad associated to the monadic right adjoint $\mathrm{oblv}\colon \Mod_F(\Cscr)\to \Cscr$ from Proposition \ref{prop:free monad}. 
It follows from the proof, and in particular from Lemma \ref{lem:iterated LKE}, 
that $T$ arises as the colimit of a sequence of endofunctors
\[\begin{tikzcd}
\id_{\Cscr}=F^{(0)}\arrow[r, "\sigma"] & F^{(1)}=F\arrow[r, "\sigma"] & F^{(2)}\arrow[r, "\sigma"] & \dots\arrow[r] & T
\end{tikzcd}\]
where each $\sigma$ decomposes as $\sigma\colon F^{(n)}\xrightarrow{\eta} F\circ F^{(n)}\xrightarrow{\rho} F^{(n+1)}$. This sequence is defined inductively, starting with $\sigma=\eta\colon \id_{\Cscr}=F^{(0)}\to F^{(1)}=F$ with $\rho$ being the identity, with $F^{(n+1)}$ for $n\geq 1$ being defined as the colimit of the following $\Iscr_{\leq 1}$-diagram of endofunctors:
\[\begin{tikzcd}
F^{(n-1)}\arrow[r, "\eta"] & F\circ F^{(n-1)}\arrow[rr, bend left=30, "F(\sigma)"] \arrow[r, "\rho"] & F^{(n)}\arrow[r, "\eta"] & F\circ F^{(n)}\arrow[r, dotted, "\rho"] & F^{(n+1)}.
\end{tikzcd}\]
\end{remark}
The filtration from Remark \ref{rem:free monad explicit} can often be understood quite explicitly in a stable setting:
\begin{lemma}\label{lem:cofibers of filtration}
Let $\Cscr$ be an $\infty$-category with colimits and $\eta\colon \id_{\Cscr}\to F$ 
a pointed endofunctor preserving sequential colimits.
Suppose that there exists a colimit preserving functor $U\colon \Cscr\to \Dscr$ 
to a stable $\infty$-category with colimits, together with a cofiber sequence in $\Fun(\Cscr, \Dscr)$
\[\begin{tikzcd}
U\arrow[r, "U\eta"] & UF\arrow[r] & U.
\end{tikzcd}\]
For each $n\geq 1$, there then exists a natural cofiber sequence in
$\Fun(\Cscr, \Dscr)$ of the form
\[\begin{tikzcd}
UF^{(n)}\arrow[r, "U\sigma"] & UF^{(n+1)}\arrow[r] & U.
\end{tikzcd}\]
\end{lemma}
\begin{proof}
We proceed by induction on $n$. The colimit from Remark \ref{rem:free monad explicit} can be computed as the iterated pushout
$$\begin{tikzcd}
F^{(n)} \amalg_{F^{(n-1)}} F\circ F^{(n-1)}\arrow[r, "{(\id, \rho)}"]\arrow[d, "{(\eta, F(\sigma))}"{swap}] & F^{(n)}\arrow[d, "\sigma"]\\
F\circ F^{(n)}\arrow[r, "\rho"] & F^{(n+1)}
\end{tikzcd}$$
where the top left corner denotes the pushout of $\sigma\colon F^{(n-1)}\to F^{(n)}$ and $\eta\colon F^{(n-1)}\to F\circ F^{(n-1)}$. It follows that the cofiber of $U\sigma\colon UF^{(n)}\to UF^{(n+1)}$ is naturally equivalent to the cofiber of the image of the left vertical map under $U$, that is, the total cofiber of the square
\[\begin{tikzcd}
UF^{(n-1)}\arrow[d, "\sigma"]\arrow[r, "\eta"] & UF\circ F^{(n-1)}\arrow[d, "F(\sigma)"]\\
UF^{(n)}\arrow[r, "\eta"] & UF\circ F^{(n)}.
\end{tikzcd}\]
The map between the horizontal cofibers can be identified with $U\sigma\colon UF^{(n-1)}\to UF^{(n)}$, so that the total cofiber is equivalent to $U$ by inductive hypothesis.
\end{proof}
\begin{remark}\label{rem:cofree comonad}
Replacing $\Cscr$ by its opposite, one obtains a similar result for the cofree comonad on a co-pointed endofunctor $G\to \id_{\Cscr}$. 
The underlying endofunctor of the induced comonad is then given by $G_{(\infty)}=\lim_i G_{(i)}$ where the $G_{(i)}$ are defined by a limit formula dual to Remark \ref{rem:free monad explicit}. 
In this case, if $U\colon \Cscr\to \Dscr$ is a limit-preserving functor to a stable $\infty$-category such that $\fib(UG\to U)\simeq U$, then there is a fiber sequence
\[
U\to UG_{(n)}\to UG_{(n-1)}
\]
for each $n\geq 2$.
\end{remark}
\begin{remark}\label{rem:adjoint cofree comonad}
Suppose that $\Cscr$ admits both limits and colimits and that the endofunctor $F$ admits a right adjoint $G$. 
Then the natural transformation between left adjoint functors $\id_{\Cscr}\to F$ is adjoint to a natural transformation between the corresponding right adjoint functors $G\to \id_{\Cscr}$. 
In this case, the free monad $F^{(\infty)}$ and the cofree comonad $G_{(\infty)}$ from Remark \ref{rem:cofree comonad} are adjoint as well.
In fact, the functors $F^{(i)}$ and $G_{(i)}$ are adjoint for all $i$.
\end{remark}

\begin{remark}\label{rmk:pb formula cofree}
	Let $\Cscr$ be an $\infty$-category with limits and let $w\colon G \to \id_{\Cscr}$ be a co-pointed endofunctor on $\Cscr$ which preserves sequential limits.
	In the special situation where $G$ is obtained as a pullback
	$$\begin{tikzcd}[column sep=4em]
		G \ar[d, "w"'] \ar[r,"F"] \ar[dr, phantom, "\lrcorner", near start] & \id_{\Cscr} \ar[d, "t"] \\[3ex]
		\id_{\Cscr} \ar[r, "\varphi"] & T
	\end{tikzcd}$$
	in $\End(\Cscr)$, the limit formula for the $G_{(i)}$ from the dual of Remark~\ref{rem:free monad explicit} can be expressed explicitly as follows.
	For any $i \geq 2$, there exists a pullback square
	$$\begin{tikzcd}[column sep=large]
			G_{(i+1)} \ar[r, "F_i"] \ar[d, "r_i"] & G_{(i)} \ar[d, "{(r_{i-1},t_{G_{(i)}})}"] \\[4ex]
			G_{(i)} \ar[r, "{(F_{i-1},\varphi_{G_{(i)}})}"] & G_{(i-1)}\times_{T G_{(i-1)}} T G_{(i)}
	\end{tikzcd}$$
	in $\End(\Cscr)$, where we set $G_{(1)} = \id_{\Cscr}$ and $G_{(2)} = G$. The $r_i$ and $F_i$ are defined inductively with $r_1 = w$ and $F_1 = F$ by definition.
	The pullback in the bottom right corner is taken over the maps $t_{G_{(i-1)}}\colon G_{(i-1)} \to T\circ G_{(i-1)}$ and $Tr_{i-1}\colon TG_{(i)} \to TG_{(i-1)}$ for any $i\geq 2$.
\end{remark}

\begin{definition}
	Let $\Cscr$ be an $\infty$-category. We note that the category $\Mod_F(\Cscr)$ from Notation~\ref{not:modules for pointed endofunctors} is contravariantly functorial in the
	pointed endofunctor $F$. In particular, there is a functor
	\begin{equation*}
		\Mod_{(-)} \colon \Alg_{\bE_0}(\End(\Cscr))^{\op} \to (\iCat)_{/\Cscr}.
	\end{equation*}
	This sends a pointed endofunctor $\eta \colon \id_\Cscr \to F$ to the forgetful functor $\oblv\colon\Mod_F(\Cscr) \to \Cscr$ and a map $\alpha \colon (F,\eta) \to (F',\eta')$ of
	pointed endofunctors to the functor $\alpha^\ast \colon \Mod_{F'}(\Cscr) \to \Mod_F(\Cscr)$ which is induced by pre-composition with $\alpha$.
\end{definition}

\begin{notation}\label{not:transformation E0 E1}
	Let $T$ be a monad on a category $\Cscr$ with unit transformation $\eta\colon \id_{\Cscr} \to T$ in $\End(\Cscr)$. We write $T_0 = (T,\eta)$ for the underlying $\bE_0$-algebra of $T$.
	Then there is a canonical commutative diagram $\beta_T$
	$$\begin{tikzcd}[column sep=small]
		\LMod_T(\Cscr) \ar[dr, "U_T"'] \ar[rr] && \Mod_{T_0}(\Cscr) \ar[dl, "\oblv"] \\[2ex]
		& \Cscr &
	\end{tikzcd}$$
	in $\iCat$ and these diagrams $\beta_T$ define a natural transformation $\beta \colon \LMod \to {\Mod_{(-)}} \circ U$. It follows that there is a lax commutative diagram
	$$\begin{tikzcd}[column sep=small]
		\Alg(\End(\Cscr))^{\op} \ar[dr, hook,"\LMod"'] \ar[rr, "U"] && \Alg_{\bE_0}(\End(\Cscr))^{\op} \ar[dl, "\Mod_{(-)}"] \\
		& (\iCat)_{/\Cscr} &
	\end{tikzcd}$$
	of $\infty$-categories.
\end{notation}

Recall from Corollary~\ref{cor:monads_adjunctions} that the functor $\LMod$ defines an equivalence $\Alg(\End(\Cscr))^{\op} \simeq (\iCat)_{/\Cscr}^{\R,\mon}$ between
the opposite of the category of monads on $\Cscr$ and the full subcategory of $(\iCat)_{/\Cscr}$ spanned by the monadic right adjoints. It follows that there exists a functor
\begin{equation}\label{eq:partial left adjoint}
	(\iCat)_{/\Cscr}^{\R,\mon} \to \Alg_{\bE_0}(\End(\Cscr))^{\op}
\end{equation}
which assigns to a monadic right adjoint $U\colon \Dscr \to \Cscr$ the underlying $\bE_0$-algebra $T_0$ of the monad $T = U \circ U^L$ on $\Cscr$. The following proposition
shows that the natural transformation $\beta$ from Notation~\ref{not:transformation E0 E1} exhibits the functor from Eq.~(\ref{eq:partial left adjoint}) as a partial left adjoint of $\Mod_{(-)}$.

\begin{proposition}\label{prop:partial left adjoint}
	Let $\Cscr$ be an $\infty$-category. For any pointed endofunctor $\eta \colon \id_{\Cscr} \to F$ and any monadic right adjoint
	$U \colon \Dscr \to \Cscr$ with associated monad $T = U \circ U^L$ on $\Cscr$, the natural map
	\begin{equation*}
		\Map_{\Alg_{\bE_0}(\End(\Cscr))}(F,T_0) \to \Map_{(\iCat)_{/\Cscr}}((\Dscr \xrightarrow{U} \Cscr),(\Mod_F(\Cscr) \xrightarrow{\oblv} \Cscr))
	\end{equation*}
	defined by $\alpha \mapsto \alpha^\ast \circ \beta_T$ is an equivalence.
\end{proposition}

\begin{proof}
	We identify the category $\Mod_F(\Cscr)$ associated to $\eta \colon \id_{\Cscr} \to F$ with the pullback
	$$\begin{tikzcd}
		\Mod_F(\Cscr) \ar[d, "\oblv"'] \ar[r] & \LEq(F,\id_{\Cscr}) \ar[d, "\eta^\ast"] \\[3ex]
		\Cscr \ar[r, "\id"] & \Cscr^{B\bN}
	\end{tikzcd}$$
	in $\iCat$. It follows that for any monadic right adjoint $U\colon\Dscr \xrightarrow{} \Cscr$ there exists a commutative diagram
	$$\begin{tikzcd}
		\Map_{(\iCat)_{/\Cscr}}(\Dscr,\Mod_F(\Cscr)) \ar[r] \ar[d] & \Map_{(\iCat)_{/\Cscr}}(\Dscr,\LEq(F,\id_{\Cscr})) \ar[d, "(\eta^\ast)_\ast"] \ar[r, "\sim"] & \Map_{\End(\Cscr)}(F,U \circ U^L) \ar[d, "\eta^\ast"] \\[3ex]
		\ast \ar[r, "\id \circ U"] & \Map_{(\iCat)_{/\Cscr}}(\Dscr,\Cscr^{B\bN}) \ar[r, "\sim"] & \Map_{\End(\Cscr)}(\id_{\Cscr},U\circ U^L)
	\end{tikzcd}$$
	whose left hand square is a pullback. Here, the inverse of the lower horizontal equivalence on the right sends a map $\alpha \colon \id_{\Cscr} \to U \circ U^L$ in $\End(\Cscr)$ to
	the composite
	\begin{equation*}
		\Dscr \simeq \LMod_T(\Cscr) \xrightarrow{\beta_T} \Mod_{T_0}(\Cscr) \to \LEq(U \circ U^L,\id_{\Cscr}) \xrightarrow{\alpha^\ast} \Cscr^{B\bN}
	\end{equation*}
	in $(\iCat)_{/\Cscr}$. The claim now follows from identifying the outer pullback with $\Map_{\Alg_{\bE_0}(\End(\Cscr))}(F,T_0)$.
\end{proof}

\begin{corollary}\label{cor:recognition free monads}
	Let $\Cscr$ be an $\infty$-category and let $T$ be a monad on $\Cscr$ with underlying $\bE_0$-algebra $T_0$. Let $\alpha \colon (F,\vartheta) \to T_0$ be a map of pointed endofunctors on $\Cscr$. Suppose that the induced functor
	\begin{equation*}
		\alpha^\ast \circ \beta_T \colon \LMod_T(\Cscr) \to \Mod_F(\Cscr)
	\end{equation*}
	is an equivalence in $(\iCat)_{/\Cscr}$. Then $\alpha$ exhibits $T$ as the free $\bE_1$-algebra on $\vartheta\colon\id_{\Cscr} \to F$.
\end{corollary}

\begin{proof}
	Let $T'$ be a monad on $\Cscr$. Then there exists a commutative diagram
	$$\begin{tikzcd}
		\Map_{\Alg(\End(\Cscr))}(T,T') \ar[d, "\alpha^\ast \circ U"] \ar[r, "\sim"] & \Map_{(\iCat)_{/\Cscr}}(\LMod_{T'}(\Cscr),\LMod_T(\Cscr)) \ar[d,"{(\alpha^\ast \circ \beta_T)_\ast}"] \\[3ex]
		\Map_{\Alg_{\bE_0}(\End(\Cscr))}(F,T_0') \ar[r, "\sim"] & \Map_{(\iCat)_{/\Cscr}}(\LMod_{T'}(\Cscr),\Mod_F(\Cscr))
	\end{tikzcd}$$
	by naturality of $\beta$. Here, the upper horizontal map is an equivalence by Corollary~\ref{cor:monads_adjunctions} and the lower horizontal equivalence is the map
	from Proposition~\ref{prop:partial left adjoint}. Since the right vertical map is an equivalence by assumption, the left vertical map is such as well.
\end{proof}

\section{Categorical actions}\label{sec:categorical_actions}

We give some background on actions of monoidal categories on other categories, including an
analysis of the orbits and fixed points of these actions. This material is used in the next section for
the actions by the Frobenius homomorphisms and the Adams operations.

\subsection{Functoriality of module categories}\label{ssec:bar-construction}
Let $\Mscr$ be a monoidal $\infty$-category, that is, an $\mathbf E_1$-monoid in
$\Cat_\infty$. 
A right action of $\Mscr$ on an $\infty$-category $\Cscr$ is a right 
$\Mscr$-module structure in the Cartesian symmetric monoidal $\infty$-category $\Cat_\infty$ of
$\infty$-categories.
Let us simply write $\RMod_\Mscr=\RMod_\Mscr(\Cat_\infty)$ for the
$\infty$-category of right $\Mscr$-modules.
Of course, there is a similar $\infty$-category $\LMod_{\Mscr}$ of left $\Mscr$-modules.
This $\infty$-category can also be understood in terms of right modules via the equivalence
$\LMod_{\Mscr^{\mathrm{rev}}}\simeq \RMod_{\Mscr}$, where 
$\Mscr^{\mathrm{rev}}$ denotes the $\infty$-category $\Mscr$ with the reverse
monoidal structure. Given $\Cscr,\Dscr$ in $\RMod_\Mscr$ or $\LMod_\Mscr$, we denote by $\Fun_{\Mscr}(\Cscr,\Dscr)\in \Cat_\infty$ the category of (strong) $\Mscr$-linear functors.

A monoidal functor $f\colon \Mscr\rightarrow \Nscr$ defines a forgetful functor $f_*\colon \RMod_\Nscr\rightarrow \RMod_\Mscr$ which has both left and right adjoints, given by
$$
f^*(-):= -\otimes_{\Mscr}\Nscr \qquad \text{and} \qquad f^!(-):= \Fun_{{\Mscr}}(\Nscr, -).
$$
Here $\Nscr$ is considered as an $(\Mscr,\Nscr)$-bimodule 
in the formula for $f^*$, and as a $(\Nscr,\Mscr)$-bimodule in the formula for $f^!$. 
Given $\Cscr \in \RMod_\Mscr$, the categories $f^*\Cscr$ and $f^!\Cscr\in \RMod_\Nscr$ can be computed by simplicially resolving $\Nscr$ as the geometric
realization of a simplicial free $\Mscr$-module. Namely, consider the simplicial left
$\Mscr$-module category $\Bar_{\Mscr}^\bullet(\Mscr, \Nscr)\colon \Delta^\op \rightarrow
\RMod_{\Mscr}({\Cat}_\infty)$ given by the bar construction (see \cite[Constr.~4.4.2.7]{ha}). Here
we view $\Mscr$ as an $(\Mscr,\Mscr)$-bimodule category, and $\Nscr$ as an $(\Mscr,\Nscr)$-bimodule
category via $f$. By \cite[Thm.~4.4.2.8]{ha} we have a natural equivalence
$|\Bar_{\Mscr}^\bullet(\Mscr, \Nscr)| \simeq \Nscr$ as $(\Mscr,\Nscr)$-bimodules, and so we have 
$$
f^*(\Cscr)\simeq |\Cscr\otimes_{\Mscr}\Bar_{\Mscr}^\bullet(\Mscr,\Nscr)|\simeq
|\Bar_{\Mscr}^\bullet(\Cscr, \Nscr)|.
$$
Similarly, resolving $\Nscr\simeq|\Bar_{\Mscr}^\bullet(\Nscr, \Mscr),
\Nscr)|$ as an $(\Nscr,\Mscr)$-bimodule, we get 
$$
f^!(\Cscr)\simeq \Tot\big(\Fun_{\RMod_\Mscr}(\Bar_{\Mscr}^\bullet(\Nscr,\Mscr),\Cscr)\big).
$$
The category $f^*(\Cscr)$ is typically hard to understand explicitly, since colimits in $\infty$-categories are difficult to compute.
The situation is better for $f^!(\Cscr)$, as it is given by a limit. 
For later purposes, let us give a more explicit description of the cosimplicial diagram $\Fun_{\RMod_\Mscr}(\Bar_{\Mscr}^\bullet(\Nscr,\Mscr),\Cscr)$.

To this end, recall that the $n$-th term in the simplicial $(\Nscr,\Mscr)$-bimodule category $\Bar_{\Mscr}^\bullet(\Nscr, \Mscr)$ is identified with $\Nscr\times \Mscr^{\times n+1}$, with the right action of $\Mscr$ on the right-most term in the product $\Mscr^{\times n+1}$, and the left $\Nscr$-action on $\Nscr$. The boundary functors are induced by 
 \begin{align*}
	\partial_0\colon (N, M_0,\ldots, M_{n}) &\mapsto (N\otimes_{\Nscr} f(M_0),M_1, \ldots,M_n), &\\
	\partial_i\colon (N, M_0,\ldots, M_{n}) &\mapsto (N, M_0, \ldots , M_{i-1}\otimes_{\Mscr}M_i,\ldots, M_n),& \text{for $0<i\le n$}
\end{align*}
Restriction along the unique monoidal functor $e\colon \ast\to \Mscr$ induces an equivalence
$\Fun_{\RMod_\Mscr}(\Dscr\times \Mscr, \Cscr)\simeq \Fun(\Dscr,\Cscr)$, by the $(e^*, e_*)$-adjunction. Consequently, we get a formula for $f^!(\Cscr)$ as the totalization 
\begin{equation}\label{eq:formula for f^!}
f^!(\Cscr)\simeq  \Tot\big(\Fun_{\RMod_\Mscr}(\Nscr\times \Mscr^{\times \bullet+1}, \Cscr)\big) \simeq \Tot\big(\Fun( \Nscr\times \Mscr^{\times \bullet}, \Cscr)\big).
\end{equation}

\begin{remark}\label{rem:formula for differentials}
The coboundary functors $\delta_i\colon \Fun( \Nscr\times \Mscr^{\times n}, \Cscr) \rightarrow  \Fun( \Nscr\times \Mscr^{\times n+1}, \Cscr)$ in the cosimplicial object $\Fun( \Nscr\times \Mscr^{\times \bullet}, \Cscr)$ are described as follows.

\begin{enumerate}
	\item[(a)] For $0\leq i \leq n$, the functor $\delta_i$ is given by precomposition with $\partial_i\colon  \Nscr\times \Mscr^{\times n+1} \rightarrow  \Nscr\times \Mscr^{\times n}$.
	
	\item[(b)] For $i=n+1$, the functor 
$$
\delta_n\colon \Fun( \Nscr\times \Mscr^{\times n},\Cscr) \rightarrow \Fun(\Nscr \times \Mscr^{\times n+1},\Cscr)\simeq \Fun(\Nscr \times \Mscr^{\times n},\Fun(\Mscr, \Cscr))
$$
is induced by the functor $\Cscr\to \Fun(\Mscr, \Cscr)$ that is Cartesian-dual to the action map $\Cscr\times \Mscr\to \Cscr$ (using the rightmost copy of $\Mscr$).
\end{enumerate}

\end{remark}

\subsection{Digression: classifying categories} 
Let $\Cscr$ be a symmetric monoidal category with a unit object $\mathbf 1$. Let $\Alg^{\mathrm{aug}}(\Cscr):=\Alg(\Cscr_{\mathbf 1/  / \mathbf 1})$ and $\cAlg^{\mathrm{aug}}(\Cscr):=(\Alg^{\mathrm{aug}}(\Cscr^\op))^\op$ be the categories of augmented algebras and coalgebras in $\Cscr$. Recall that if $\Cscr$ admits geometric realizations and totalizations there is a pair of adjoint functors
$$
\Bar\colon \Alg^{\mathrm{aug}}(\Cscr) \rightleftarrows \cAlg^{\mathrm{aug}}(\Cscr) \ \! :\!\mathrm{Cobar}
$$
where at the level of underlying objects we have $\Bar(A)\simeq \mathbf 1\otimes_A \mathbf 1$ for an algebra $A$. 
The right adjoint sends a coalgebra $C$ to the bar construction of $C$, viewed as an algebra object in the opposite category. 
If the monoidal structure on $\Cscr$ is Cartesian, then any object has a canonical coalgebra structure, and $\cAlg^{\mathrm{aug}}(\Cscr)\simeq \Cscr_{\mathbf 1/  / \mathbf 1}\simeq  \Cscr_{\mathbf 1/}$.
In this case, $\mathrm{Cobar}(C)$ is given by the pullback $1\times_C 1$ at the level of underlying objects.

\begin{construction} 
Let $\Cscr=\Cat_{\infty}$ with the Cartesian monoidal structure. 
Given a monoidal category $\Mscr \in   \Alg(\Cat_\infty)\simeq  \Alg^{\mathrm{aug}}(\Cat_{\infty})$, we will define its \textit{classifying category} to be
\[
\mathrm B\Mscr:=\Bar(\Mscr)\in \Cat_{\infty,*}.
\]
More generally, let $\Cscr=\Alg_{\mathbf E_n}(\Cat_{\infty})$ be the $\infty$-category of $\mathbf E_n$-monoidal categories for $0 \le n\le \infty$. 
In this case, the bar construction can be viewed as a functor
	$$
	\Bar\colon \Alg_{\mathbf E_{n+1}}(\Cat_{\infty}) \rightarrow \Alg_{\mathbf E_n}(\Cat_{\infty}).
	$$
	Moreover, it is compatible with the forgetful functors $\Alg_{\mathbf E_{n}}(\Cat_{\infty}) \rightarrow \Alg_{\mathbf E_{n-1}}(\Cat_{\infty})$.
	Given an $\mathbf E_{n+1}$-monoidal category $\Mscr$ we will implicitly remember that $\mathrm B\Mscr= \Bar(\Mscr)$ has the natural structure of a $\mathbf E_n$-monoidal category.
\end{construction}	

\begin{remark} 	\label{rem:action of BM}
	By adjunction, given an $\mathbf E_n$-monoidal category $\Nscr$ and an $\mathbf E_{n+1}$-monoidal category $\Mscr$ we get a natural equivalence (of $\infty$-groupoids)
	$$
	\Map_{\Alg_{\mathbf E_{n}}(\Cat_{\infty})}(\mathrm{B}\Mscr, \Nscr)\simeq \Map_{\Alg_{\mathbf E_{n+1}}(\Cat_\infty)}(\Mscr, \mathrm{Cobar}(\Nscr)).
	$$
	Note that $\mathrm{Cobar}(\Nscr)$ is equivalent to the $\mathbf E_{n+1}$-space $\ast\times_\Nscr \ast = \Escr \mathrm{nd}_{\Nscr}(\mathbf 1)$, where the functor $* \rightarrow \Nscr$ is given by the monoidal unit $\mathbf 1\in \Nscr$.
	
	Assume now for simplicity that $\Mscr=M$ is an $\mathbf E_{2}$-space (viewed as an $\mathbf E_{2}$-monoidal category). Then 
	$$\Map_{\Alg_{\mathbf E_{1}}\!(\Cat_\infty)}(\mathrm{B}M, \Nscr)\simeq \Map_{\Alg_{\mathbf E_{2}}\!(\Sscr)}(M, \mathrm{Cobar}(\Nscr)^\simeq)\simeq \Map_{\Alg_{\mathbf E_{2}}(\Sscr)}(M, \Omega_{\mathbf 1}(\Nscr^\simeq)) ,$$ 
	since $\mathrm{Cobar}(\Nscr)^\simeq\simeq \mathrm{Cobar}(\Nscr^\simeq)\simeq \Omega_{\mathbf 1}(\Nscr^\simeq)$. 
	This allows to describe the datum of a monoidal action of $\mathrm B\Mscr$ on a category $\Cscr$; 
	namely, the latter is given by an $\mathbf E_1$-monoidal functor $\mathrm B\Mscr\rightarrow \Escr\mathrm{nd}(\Cscr)$, which by the above is the same as the datum of an $\mathbf E_2$-monoidal map from $\Mscr$ to the $\mathbf E_2$-center
	  $\mathrm{Z}_{\mathbf E_2}(\Cscr):=\End_{\Escr\mathrm{nd}(\Cscr)}(\id_{\Cscr})$ of $\Cscr$.   
\end{remark}	
\begin{remark}\label{rem:BZ is an idempotent algebra} Let us consider the functor 
	$$
	\mathrm B(-)\colon \Alg_{\mathbf E_{\infty}}(\Cat_{\infty}) \rightarrow \Alg_{\mathbf E_{\infty}}(\Cat_{\infty}).
	$$
	As a left adjoint it commutes with all colimits; moreover, it is given explicitly by the functor which sends a symmetric monoidal category $\Mscr$ to the relative tensor product $*\otimes_\Mscr *$.
	
	Let $\bN\rightarrow \bZ$ be the natural embedding. Then the map $\bZ\simeq \bZ\otimes_{\bN}\bN\rightarrow \bZ \otimes_\bN \bZ$ in commutative monoids (in spaces) is an equivalence. Indeed, it is enough to check on the level of underlying $\bZ$-modules, and $\bZ \otimes_\bN \bZ\simeq \bZ \otimes_\bN \colim (\cdots \xrightarrow{+1} \bN\xrightarrow{+1} \bN \xrightarrow{+1} \cdots) \simeq \colim (\ldots \xrightarrow{+1} \bZ\xrightarrow{+1} \bZ \xrightarrow{+1} \ldots)\simeq \bZ$ since the map $+1\colon \bZ \rightarrow \bZ$, $n\mapsto n+1$ is an isomorphism. Thus we get
	$$
	\B\bZ\otimes_{\B\bN}\B\bZ\we\B\bZ.
	$$
\end{remark}

\subsection{Homotopy invariants $\Cscr^{h\Mscr}$}

We specialize the discussion in Section \ref{ssec:bar-construction} to the case of the (unique) monoidal functor $p\colon \Mscr \to *$. 
\begin{notation}
 Let $\Cscr\in \RMod_\Mscr(\Cat_{\infty})$ be an $\Mscr$-module category. The \textit{homotopy invariants} $\Cscr^{h\Mscr}$ are defined as $p^!(\Cscr):=\Fun_{\Mscr}(*,\Cscr)$. The functor $*\to \Mscr$ given by the monoidal unit in $\Mscr$ produces a natural functor $F\colon \Cscr^{h\Mscr} \to \Cscr$.
\end{notation}	

The main goal of this section is to show that the functor $F\colon  \Cscr^{h\Mscr} \to \Cscr$ is monadic under some natural assumptions on $\Mscr$ and $\Cscr$.

\begin{definition}
	We call a monoidal category $\Mscr$ \textit{group-like} if for any object $M\in \Mscr$ the corresponding tensor multiplication functor $-\otimes_{\Mscr} M \colon \Mscr \to \Mscr$ is an equivalence of categories. 
\end{definition}	

\begin{remark}\label{rem:the side of tensor product doesn't matter}
	The condition that $-\otimes M$ is an equivalence of categories implies that there is an object $X\in \Mscr$ and an equivalence $\alpha\colon X\otimes M\xrightarrow{\sim} \mathbf 1_\Mscr $. Then there is a unique equivalence $M\otimes X \xrightarrow{\sim} \mathbf 1_\Mscr$ that is mapped to $\id_M\otimes \alpha \colon M\otimes X \otimes M \xrightarrow{\sim} M$ and thus the left multiplication functor $M\otimes -\colon \Mscr \to \Mscr$ is also an equivalence (with inverse given by left multiplication by $X$). It also follows that given any $\Mscr$-module category $\Cscr$, any object $M\in\Mscr$ acts by an automorphism (with inverse given by action of $X$).
\end{remark}	
In this case we have the following:

\begin{proposition}\label{prop:categorical invariants are monadic} Let $\Mscr$ be a group-like monoidal category, and let $\Cscr \in \RMod_{\Mscr}(\Cat_\infty)$. If $\Cscr$ has $\Mscr$-indexed colimits, then the forgetful functor 
		$$
		F\colon \Cscr^{h\Mscr} \rightarrow \Cscr 
		$$
 is monadic, with the corresponding monad $T$ given explicitly by the colimit-averaging $X \mapsto T(X):=\colim_{M\in \Mscr} M(X)\in \Cscr$ of the $\Mscr$-action.

\end{proposition}	

\begin{proof}
	By \eqref{eq:formula for f^!} we can compute the homotopy invariants via the bar resolution, obtaining an equivalence
	$
	\Cscr^{h\Mscr} \simeq \Tot(\Fun(\Mscr^{\times \bullet},\Cscr)). 
	$ 
	We will prove that the cosimplicial $\infty$-category $\Fun(\Mscr^{\times\bullet},\Cscr)$ satisfies the Lurie--Barr--Beck conditions, or, more precisely, the conditions in \cite[Thm~4.7.5.2]{ha}.
	Remark \ref{rem:formula for differentials} shows that for each cosimplicial degree $n$, the functor $\delta_0\colon \Fun(\Mscr^{\times n},\Cscr)\to  \Fun(\Mscr^{\times n+1},\Cscr)$ is given by precomposing with the projection $\Mscr^{\times n+1} \to \Mscr^{\times n}$ along the leftmost coordinate.
	It follows that $\delta_0$ has a left adjoint $\delta_0^\L$, given by the functor
	$$
	\Fun(\Mscr^{\{0,\ldots, n\}},\Cscr)\simeq \Fun(\Mscr, \Fun(\Mscr^{\{1,\ldots, n\}},\Cscr)) \xrightarrow{\colim_{\Mscr}} \Fun(\Mscr^{\{1,\ldots, n\}},\Cscr).
	$$
	This is well-defined because $\Cscr$ has $\Mscr$-indexed colimits. 
	It thus remains to check that the commutative square 
$$
	\begin{tikzcd}
	\Fun(\Mscr^{\times n},\Cscr) \arrow[r,"\delta_0"]\arrow[d, "\delta_i"'] & 	\Fun(\Mscr^{\times n+1},\Cscr) \arrow[d, "\delta_{i+1}"]\\	
	\Fun(\Mscr^{\times n+1},\Cscr)\arrow[r, "\delta_0"] & \Fun(\Mscr^{\times n+2},\Cscr)
	\end{tikzcd}
$$
is left adjointable for any $i$ and $n$: namely,
the natural transformation 
\begin{equation}\label{eq:adjointable transformation}
	\delta_0^\L\circ \delta_{i+1} \to \delta_i\circ \delta_0^\L
\end{equation}
 for the adjoint square 
$$
\begin{tikzcd}
	\Fun(\Mscr^{\times n},\Cscr)  \arrow[d, "\delta_i"'] &  \Fun(\Mscr^{\times n+1},\Cscr) \arrow[d, "\delta_{i+1}"] \arrow[l,"\delta_0^\L"']\\ 
	\Fun(\Mscr^{\times n+1},\Cscr) & \Fun(\Mscr^{\times n+2},\Cscr) \arrow[l, "\delta_0^\L"'] \arrow[lu, Rightarrow, shorten=3ex]
\end{tikzcd}
$$
is an equivalence.
For $1\le i \le n+1$, we can identify $\delta_{i+1}\colon \Fun(\Mscr, \Fun(\Mscr^{\times n},\Cscr))\to \Fun(\Mscr, \Fun(\Mscr^{\times n+1},\Cscr))$ with the map induced by
$\delta_{i}\colon \Fun(\Mscr^{\times n},\Cscr)\to \Fun(\Mscr^{\times n+1},\Cscr)$.
For $1\le i\le n$, the functor $\delta_i$ is given by restriction, so
the statement follows from restriction functors preserving colimits
(which are computed pointwise in a functor category). 
For $i=n+1$, the functor $\delta_i$ arises from the $\Mscr$-action on $\Cscr$, and the statement follows from the action preserving $\Mscr$-indexed colimits; this holds because objects of $\Mscr$ act by autoequivalences (Remark \ref{rem:the side of tensor product doesn't matter}). 
The most tricky is the case $i=0$, where we are concerned with the diagram
$$
\begin{tikzcd}
	\Fun(\Mscr^{\times n},\Cscr)  \arrow[d, "F\mapsto \underline F"'] &[1pc] & \Fun(\Mscr, \Fun(\Mscr^{\times n},\Cscr)) \arrow[d, "F\mapsto F\circ m_\Mscr"] \arrow[ll,"\colim_{\Mscr}"']\\ 
	\Fun(\Mscr, \Fun(\Mscr^{\times n},\Cscr)) & \Fun(\Mscr, \Fun(\Mscr, \Fun(\Mscr^{\times n},\Cscr))) \arrow[r,"\sim"]   \arrow[l, "\colim_\Mscr"']  &\Fun(\Mscr^{\times 2}, \Fun(\Mscr^{\times n},\Cscr).
\end{tikzcd}
$$
Here $\underline{F}\colon \Mscr \to \Fun(\Mscr^{\times n},\Cscr)$ denotes the constant functor with value $F\in \Fun(\Mscr^{\times n},\Cscr)$ and $m_\Mscr \colon \Mscr\times \Mscr \to \Mscr$ is the multiplication.  Evaluation of the transformation \eqref{eq:adjointable transformation} on a given functor $G\colon \Mscr \to \Fun(\Mscr^{\times n},\Cscr)$ yields a natural transformation
$$ 
\colim_{M\in\Mscr} G (m_\Mscr(M,-)) \longrightarrow \underline{\colim_{M\in \Mscr} G(M)}
$$ 
in $\Fun(\Mscr, \Fun(\Mscr^{\times n},\Cscr))$. On an object $N\in \Mscr$, it is given by the natural map 
$$
\colim_{M\in \Mscr} G(M\otimes N) \to \colim_{M\in\Mscr} G(M) ,
$$
induced by the functor $-\otimes N \colon \Mscr \to \Mscr$. This functor, and consequently the map above, is an equivalence, since we assumed $\Mscr$ to be group-like.  By \cite[Thm.~4.7.5.2(3)]{ha} we get that the functor $F\colon \Cscr^{h\Mscr}\simeq \Tot(\Fun(\Mscr^{\bullet},\Cscr)) \rightarrow \Cscr$ is monadic and \cite[Thm.~4.7.5.2(3)]{ha} identifies the underlying functor of the monad with $\delta_0^\L\circ \delta_1\colon \Cscr \rightarrow\Cscr$.
Noting that $\delta_1\colon \Cscr \rightarrow \Fun(\Mscr,\Cscr)$ sends $X\in \Cscr$ to the functor sending $Y\in \Mscr$ to $Y(X)\in \Cscr$, 
this is exactly given by the above formula.
\end{proof}	

\begin{remark}\label{rem:invariants are also comonadic} Given $\Cscr\in \Cat_{\infty}$, one has a natural equivalence of monoidal categories $\End(\Cscr^\op)\simeq \End(\Cscr)^{\mathrm{rev}}$. Consequently, the datum of an $\Mscr$-module structure on $\Cscr$ is equivalent to the datum of an $\Mscr^{\mathrm{rev}}$-module structure on $\Cscr^{\op}$.
	Replacing $\Cscr$ by $\Cscr^\op$ and $\Mscr$ by $\Mscr^{\mathrm{rev}}$ in Proposition \ref{prop:categorical invariants are monadic},
	 we get the statement that if $\Cscr$ has $\Mscr$-indexed limits, then $F\colon  \Cscr^{h\Mscr} \rightarrow \Cscr$ is comonadic with comonad given by limit-averaging $X\mapsto \lim_{M\in \Mscr} M(X)$. Finally, if $\Cscr$ has both $\Mscr$-indexed limits and colimits, then $F$ is both monadic and comonadic with the  monad and comonad given by colimit and limit averaging of the action.
\end{remark}	
 
\begin{example}\label{ex: invariants of actions}\leavevmode
	\begin{enumerate}
		\item[(1)] Let $\Mscr =\bZ$. Then an $\Mscr$-action on $\Cscr$ is equivalent to the datum of an autoequivalence $F\colon \Cscr \xrightarrow{\sim} \Cscr$. The corresponding monad $T$ is given by the countable coproduct $X\mapsto \sqcup_{i\in\bZ} F^i(X)$, provided this exists in $\Cscr$.
		
		\item[(2)] Let $\Mscr =B\bN$, with monoidal structure induced by addition on $\bN$. By Remark \ref{rem:action of BM}, a $B\bN$-action on a category $\Cscr$ amounts to an $\mathbf E_2$-monoid map $\bN \to \mathrm{Z}_{\mathbf E_2}(\Cscr):=\End_{\Escr\mathrm{nd}(\Cscr)}(\id_{\Cscr})$. Let $\psi\colon \id_\Cscr \rightarrow \id_\Cscr$ be the image of $1\in \bN$. We note that the $B\bN$-action is not just determined by $\psi$, as $\bN$ is not the free $\bE_2$-algebra on one generator.
		The underlying endofunctor of the corresponding monad $T\in \End(\Cscr)$ is then given by 
		$$
		X \mapsto \colim_{B\bN} (X\curvearrowleft\psi_X) \simeq \mathrm{coeq}_{\psi_X,\id}(X):=\mathrm{coeq} \!\!\begin{tikzcd}(X\arrow[r, shift left=0.5ex, "\psi_X"]\arrow[r, shift right=0.5ex, "\id_X"']& X)\end{tikzcd}\!.
		$$
		Here $(X\curvearrowleft\psi_X)$ denotes the functor $B\bN\rightarrow\Cscr$ sending the base point to $X$ and $1\in\bN$ to $\psi_X$. 
		To identify its colimit with a coequalizer, we use that $B\bN$ is equivalent
		to the pushout of $\infty$-categories $* \sqcup_{*\sqcup *}[1]$, where $[1]$ is a single arrow and $*\sqcup * \to [1]$ is the embedding of endpoints. A functor $F\colon B\bN \to \Cscr$ is given by an object $X:=F(*)$ and an endomorphism $\psi_X:=F([1])$, and has colimit
	$$
	\colim_{B\bN} F\simeq \colim\big(\colim_*F\leftarrow \colim_{*\sqcup *}F\rightarrow \colim_{[1]}F\big) \simeq \colim\big(X\xleftarrow{\id_X\sqcup \id_X} X\sqcup X\xrightarrow{\psi_X\sqcup \id_X} X\big).
	$$
	 This pushout is equivalent to the coequalizer $\mathrm{coeq}_{\psi_X,\id}(X)$.
		
		\item[(3)]\label{item: example of BN^I} Let $\Mscr = B\bN^{\oplus I}$, where $I$ is some indexing set. An $\Mscr$-action on $\Cscr$ then amounts to an $\mathbf E_2$-monoid map $\bN^{\oplus I}\to \mathrm{Z}_{\mathbf E_2}(\Cscr)$, in particular producing an $I$-tuple of homotopy-commuting endotransformations $\psi_i\colon \id_\Cscr \rightarrow \id_\Cscr$. The corresponding monad $T$ is given by
		$$
		X\mapsto \colim_{B\bN^{\oplus I}} (X,\{\psi_{i,X}\}).
		$$
		In the case when $I$ is countable, choosing an identification $I\simeq \bN$ we can also identify the latter colimit with the iterated coequalizer 
		$$
		\colim_{n\in \bN} (\mathrm{coeq}_{\psi_n,\id}\circ \ldots \circ \mathrm{coeq}_{\psi_1,\id})(X).
		$$
		
	\end{enumerate}
\end{example}

\begin{remark}\label{rem:fixed points for 1-categories} Let $\mathcal M=B\Nscr$ for some $\bE_2$-space $\Nscr$. Let $\mathcal C$ be a 1-category with has $\Mscr$-indexed colimits, endowed with an $\Mscr$-action. We claim that the forgetful functor $F\colon\Cscr^{h\Mscr}\to \Cscr$ is a  fully-faithful embedding with image spanned by objects $X\in \Cscr$ for which the induced $\Nscr$-action is trivial: namely, the homomorphism $a_X\colon \Nscr\to \End_{\Cscr}(X)$ induced by the action factors through $\id_X$ (note that the target is discrete, so the last condition is a property and not a structure).
	
	Indeed, following the notation in the proof of Propoisiton \ref{prop:categorical invariants are monadic} $F$ is induced by the transformation $\Fun(\Mscr^{\times \bullet},\Cscr)\to \Fun(*,\Cscr)\simeq \Cscr$ of functors $\Delta\to \Cat_{\infty}$ induced by $*\to \Mscr$ after passing to totalization. For fully faithfulness, since $\Cscr$ is a 1-category, we need to see that functor $\Fun(\Mscr,\Cscr)\simeq \Fun(B\pi_0(\Nscr),\Cscr)\to \Cscr$ for $\bullet=1$ is faithful (note that for $\bullet=0$ the functor is $\id_{\Cscr}\colon \Cscr \to \Cscr$ and is fully faithful); this follows since $\Mscr$ is connected and $\Cscr$ is a 1-category. The image is exactly spanned  by those $X\in \Cscr$ for which the two objects $\delta^0(X)$ and $\delta^1(X)$ in $\Fun(\Mscr,\Cscr)$ are isomorphic. Since this objects are given by the constant functor $M\mapsto \underline{X}$ and $M\mapsto M(X)$ correspondingly, we get that this is equivalent to the triviality of $\Nscr$-action on $X$.
\end{remark}

\begin{remark}
	The group-like assumption on $\Mscr$ in Proposition \ref{prop:categorical invariants are monadic} is necessary in general. For example, let $\Mscr=\bN$, the monoid of natural numbers. Then an $\bN$-action on a category $\Cscr$ is given by an endofunctor $F\in \End(\Cscr)$, and $\Cscr^\bN=\lim_{B\bN} \Cscr$ is described as the equalizer $\mathrm{eq}_{F,\id_\Cscr}\!(\Cscr)$. An object in $\mathrm{eq}_{F,\id_\Cscr}\!(\Cscr)$ is thus given by the datum of an object $X\in\Cscr$ and an isomorphism $\alpha\colon F(X)\xrightarrow{\sim} X$. On the other hand, the monad $T$ in Proposition \ref{prop:categorical invariants are monadic} is given by $X\mapsto \sqcup_{n\in\bN} F^n(X)$, and a module over it rather encodes a datum of an object $X\in\Cscr$ and a map $\alpha\colon F(X)\to X$, which is \textit{not} necessarily an isomorphism. Instead, the latter category can be identified with the lax version of invariants, namely the category $\Fun_{\bN\text{-lax}}(*,\Cscr)$ of lax $\bN$-module functors.
\end{remark}

\subsection{Categorical perfection $\Cscr\rightsquigarrow \Cscr^\perf$}

Let $I$ be a set, and let $\bN^{\oplus I}\to \bZ^{\oplus I}$ be the natural homomorphism of commutative monoids, inducing a symmetric monoidal functor $f\colon B\bN^{\oplus I} \to B\bZ^{\oplus I}$. In this section we will discuss some properties of the corresponding pullback functors $f^*$ and $f^!$.
Recall that we have the forgetful functor $f_*\colon \Mod_{B\bZ^{\oplus I}}(\Cat_{\infty}) \to \Mod_{B\bN^{\oplus I}}(\Cat_{\infty})$, for which $f^*$ is a left adjoint.

\begin{lemma} The natural transformation 
	$$
	f^*f_*\to  \id_{\Mod_{B\bZ^{\oplus I}}} 
	$$
is an equivalence. 	In particular, $f_*\colon \Mod_{B\bZ^{\oplus I}}(\Cat_{\infty}) \to  \Mod_{B\bN^{\oplus I}}(\Cat_{\infty})$ is fully faithful.
	
\end{lemma}	
\begin{proof}
	On an object $\Cscr$, the above transformation is given by the action map $\Cscr \otimes_{B\bN^{\oplus I}}B\bZ^{\oplus I} \to \Cscr$, which in turn is equivalent to 
	$$
	\Cscr \otimes_{B\bN^{\oplus I}}B\bZ^{\oplus I}\simeq \Cscr \otimes_{B\bZ^{\oplus I}} (B\bZ^{\oplus I}\otimes_{B\bN^{\oplus I}}B\bZ^{\oplus I}) \longrightarrow \Cscr \otimes_{B\bZ^{\oplus I}} B\bZ^{\oplus I} \simeq \Cscr.
	$$
	This is an equivalence because $B\bZ^{\oplus I}$ is an idempotent algebra object in $\Mod_{B\bN^{\oplus I}}(\Cat_{\infty})$, by Remark \ref{rem:BZ is an idempotent algebra}.
\end{proof}

\begin{notation} Given $\Cscr\in  \Mod_{B\bN^{\oplus I}}(\Cat_{\infty})$, we denote
	$$
	\Cscr_{\perf}:= f^*(\Cscr)\simeq \Cscr \otimes_{B\bN^{\oplus I}}B\bZ^{\oplus I}, \qquad\qquad \Cscr^\perf:=f^!(\Cscr)\simeq \Fun_{B\bN^{\oplus I}}(B\bZ^{\oplus I},\Cscr).
	$$
	Abusing notation slightly, we will use the same notation for the images of $\Cscr_{\perf}, \Cscr^{\perf}$ in $\Mod_{B\bN^{\oplus I}}(\Cat_{\infty})$ under $f_*$. By adjunction, we have natural $B\bN^{\oplus I}$-linear functors $\alpha_{\Cscr}\colon \Cscr \to \Cscr_{\perf}$ and $\beta_\Cscr\colon \Cscr^{\perf}\to \Cscr$.
\end{notation}	

Let $\Cscr$ be an $\infty$-category with a $B\bN^{\oplus I}$-action. 
As in Example \ref{ex: invariants of actions}, we denote by $\{\psi_i\}_{i\in I}$ the corresponding collection of endotransformations $\psi_i\colon \id_{\Cscr} \to  \id_{\Cscr}$. Denote by $\psi_{X,i}\colon X\to X$ their evaluations on an object $X\in \Cscr$. Let us denote by $(\bZ^{\oplus I}, \le)$ the poset where $(a_i)_{i\in I}\le (b_i)_{i\in I}$ if and only if $a_i\le b_i$ for all $i\in I$; note that this condition is equivalent to $(a_i-b_i)_{i\in I} \in \bN^{\oplus I}$. It is easy to see that the category $(\bZ^{\oplus I}, \le)$ is filtered.
\begin{lemma}\label{lem:categorical inverse limit perfection} The functor $\beta_\Cscr\colon \Cscr^{\perf}\to \Cscr$ is a fully faithful embedding with image spanned by objects $X$ for which the maps $\psi_{X,i}\colon X\to X$ are invertible for all $i\in I$. If $\Cscr$ has colimits indexed by $(\bZ^{\oplus I}, \le)$, then $\beta_\Cscr$ has a left adjoint $\beta_{\Cscr}^L\colon \Cscr \to \Cscr^{\perf}$, making $\Cscr^{\perf}$ a (reflexive) localization of $\Cscr$.
\end{lemma}	

\begin{proof}
	The functor $\beta_\Cscr\colon \Cscr^{\perf}\to \Cscr$ is equivalent to the restriction $\Fun_{B\bN^{\oplus I}}(B\bZ^{\oplus I},\Cscr) \to \Fun_{B\bN^{\oplus I}}(B\bN^{\oplus I},\Cscr)$, 
	and can thus be analyzed via the natural transformation between the cosimplicial bar constructions $\Bar_{B\bN^{\oplus I}}^\bullet(B\bN^{\oplus I}, B\bN^{\oplus I})\to \Bar_{B\bN^{\oplus I}}^\bullet(B\bZ^{\oplus I}, B\bN^{\oplus I})$. 
	More precisely, $\beta_{\Cscr}$ is given by the totalization of the natural transformation 
	\begin{equation}\label{eq:transformation from Bar for BZ to BN}
	\mathrm{res}^\bullet \colon \Fun( B\bZ^{\oplus I}\times (B\bN^{\oplus I})^{\times \bullet}, \Cscr) \longrightarrow \Fun( B\bN^{\oplus I}\times (B\bN^{\oplus I})^{\times \bullet}, \Cscr)
	\end{equation}
	induced by the embedding $B\bN^{\oplus I} \to B\bZ^{\oplus I}$. Now note that for any $\infty$-category $\Dscr$, the restriction functor $\mathrm{res}_{\Dscr}\colon \Fun(B\bZ^{\oplus I},\Dscr) \to \Fun(B\bN^{\oplus I},\Dscr)$ is fully faithful: indeed, $f\colon B\bN^{\oplus I}\to B\bZ^{\oplus I}$ is a cofinal functor, since the comma category over the (unique) object in $B\bZ^{\oplus I}$ is given by the poset $(\bZ^{\oplus I}, \le)$, which is filtered. If $\Dscr$ has colimits indexed by $(\bZ^{\oplus I}, \le)$, then the left Kan extension along $f$ exists and provides a left adjoint $\mathrm{res}_{\Dscr}^L$ to the restriction functor $\mathrm{res}_{\Dscr}$. Indeed, the colimit formula for left Kan extensions shows that $\mathrm{res}_{\Dscr}^L$ sends $X\colon B\bN^{\oplus I}\to \Dscr$ to the colimit of the diagram $(\bZ^{\oplus I}, \le)\to \Dscr$ sending
	$$
	\big((a_i)\leq (b_i)\big)\quad \longmapsto \quad \big(X(\ast)\xrightarrow{X(b_i-a_i)} X(\ast)\big).
	$$
	This diagram looks like a lattice with vertices $X(\ast)$ and arrows given by copies of the endomorphisms $\alpha_i$ corresponding to the generators of $\bN^{\oplus I}$. The colimit comes with an $I$-tuple of autoequivalences, induced by the $I$-tuple of autoequivalences of the above diagram by shifting by $1$ in the $i$-th direction.

	 We now use the following general fact: given a natural transformation $F^\bullet \colon \Cscr^\bullet \to \Escr^\bullet$ of cosimplicial $\infty$-categories which is fully faithful in each cosimplicial degree, the induced functor $\Tot(F^\bullet)\colon \Tot(\Cscr^\bullet) \to \Tot(\Escr^\bullet)$ is also fully faithful. Moreover, $X\in  \Tot(\Escr^\bullet)$ is in the image of  $\Tot(F^\bullet)$ if and only if its image $\delta(X)\in \Escr^0$ is in the image of $F^0$. Applying this, and the discussion about $\mathrm{res}_{\Dscr}$, to $\Dscr= \Fun((B\bN^{\oplus I})^{\times \bullet}, \Cscr)$ in \eqref{eq:transformation from Bar for BZ to BN} we get that $\beta_\Cscr=\Tot(\mathrm{res}^\bullet)$ is fully faithful. Moreover, an object $X\in\Cscr$ lies in the image of $\beta_\Cscr$ exactly when the functor $\delta_1(X)\in \Fun(B\bN^{\oplus I}, \Cscr)$ given by $(X,\phi_{X,i})$ comes from $\Fun(B\bZ^{\oplus I}, \Cscr)$ or, in other words, if  all endomorphisms $\phi_{X,i}\colon X\to X$ are invertible.
	
Finally, if $\Cscr$ has colimits indexed by $(\bZ^{\oplus I}, \le)$, then so does $\Fun((B\bN^{\oplus I})^{\times \bullet}, \Cscr)$. Consequently, all individual functors in \eqref{eq:transformation from Bar for BZ to BN} have left adjoints. We claim that for each $\delta_i\colon [n]\to [n+1]$, the square 
	$$
	\begin{tikzcd}
		\Fun( B\bZ^{\oplus I}\times (B\bN^{\oplus I})^{\times n}, \Cscr)\arrow[d,"\delta_i"'] \arrow[r, "\mathrm{res}^n"]&\Fun( B\bN^{\oplus I}\times (B\bN^{\oplus I})^{\times n}, \Cscr) \arrow[d, "\delta_i"] \\
		\Fun( B\bZ^{\oplus I}\times (B\bN^{\oplus I})^{\times n+1}, \Cscr)\arrow[r, "\mathrm{res}^{n+1}"] & \Fun( B\bN^{\oplus I}\times (B\bN^{\oplus I})^{\times n+1}, \Cscr)		
	\end{tikzcd}
	$$ 
	is horizontally left adjointable. When $1\leq i\leq n+1$, we can use Remark \ref{rem:formula for differentials} to identify the above square with a square of the form
	$$\begin{tikzcd}
	\Fun(B\bZ^{\oplus I}, \Dscr)\arrow[r, "\mathrm{res}"]\arrow[d, "F_{i*}"'] & \Fun(B\bN^{\oplus I}, \Dscr)\arrow[d, "F_{i*}"]\\
	\Fun(B\bZ^{\oplus I}, \Dscr')\arrow[r, "\mathrm{res}"] & \Fun(B\bN^{\oplus I}, \Dscr')
	\end{tikzcd}$$
	where $F_i\colon \Dscr\to \Dscr'$ is a functor preserving colimits indexed by $(\bZ^{\oplus I}, \le)$:
	for $1\leq i\leq n$, the functor $F_i$ is a restriction functor and for $i=n+1$ it arises from the $\bN^{\oplus I}$-action on $\Cscr$. For $i=0$, the square takes the form
	$$\begin{tikzcd}
	\Fun(B\bZ^{\oplus I}, \Escr)\arrow[r, "\mathrm{res}"]\arrow[d, "m^*"'] &  \Fun(B\bN^{\oplus I}, \Escr)\arrow[d, "m^*"]\\
	\Fun(B\bZ^{\oplus I}\times B\bN^{\oplus I}, \Escr)\arrow[r, "\mathrm{res}"] & \Fun(B\bN^{\oplus I}\times B\bN^{\oplus I}, \Escr)
	\end{tikzcd}$$
	where $m$ denotes the action of $B\bN^{\oplus I}$ on $B\bZ^{\oplus I}$ and itself. This square is horizontally left adjointable because the action map $m\colon B\bZ^{\oplus I}\times B\bN^{\oplus I}\to B\bZ^{\oplus I}$ is a Cartesian fibration (see \cite[Lem.~9.2.7, Cor.~9.3.4]{riehl-verity}).
	
	By \cite[Cor.~4.7.4.18]{ha}, the collection of functors $\mathrm{res}^{\bullet, L}$ then lifts canonically to a transformation 
	$$
	\mathrm{res}^{\bullet,L} \colon \Fun( B\bZ^{\oplus I}\times (B\bN^{\oplus I})^{\times \bullet}, \Cscr) \longrightarrow \Fun( B\bN^{\oplus I}\times (B\bN^{\oplus I})^{\times \bullet}, \Cscr),
	$$
	whose totalization provides a left adjoint to $\beta_{\Cscr}$.
\end{proof}	

\begin{remark}\label{rem:formula for perfection monad}
	Being a localization, $\Cscr^\perf$ is monadic over $\Cscr$ with the monad given by $T=\beta_\Cscr\circ \beta_\Cscr^L\in \End(\Cscr)$. Following the construction of $\beta_\Cscr^L$ in the proof of Lemma \ref{lem:categorical inverse limit perfection}, one can explicitly describe the action of $T$ on objects of $\Cscr$: namely, given $X\in \Cscr$, viewed as a $B\bN^{\oplus I}$-linear functor $X\colon B\bN^{\oplus I}\to \Cscr$, one has  $T(X)=\colim_{(\bZ^{\oplus I},\le)}X(*)$ with the colimit along a diagram $(\bZ^{\oplus I},\le)\to \Cscr$ that sends a morphism $((a_i)\le (b_i))_{i\in I}$ to $(X(b_i-a_i)\colon X(*)\to X(*))_{i\in I}$. When $I=\{*\}$ is a singleton, a diagram $X\colon B\bN \to \Cscr$ amounts to an object $X\in \Cscr$ and an endomorphism $\psi_X\colon X\to X$. Then the formula for $T(X)$ simplifies to 
	$$
	T(X)=\colim(\ldots \xrightarrow{\psi_X} X \xrightarrow{\psi_X} X  \xrightarrow{\psi_X}  X  \xrightarrow{\psi_X}  \ldots).
	$$
\end{remark}	

\begin{remark}\label{rem:right adjoint to embedding of prefect modules} Applying Lemma \ref{lem:categorical inverse limit perfection} to $\Cscr^\op$ in place of $\Cscr$ (together with the induced $(B\bN^{\oplus I})^{\mathrm{rev}}\simeq B\bN^{\oplus I}$-action on $\Cscr^\op$) we can get a dual statement: namely, that if $\Cscr$ admits $(\bZ^{\oplus I},\le)$-indexed limits, then $\beta_\Cscr$ admits a right adjoint $\beta_\Cscr^R$, making $\Cscr^\perf$ a reflexive colocalization of $\Cscr$. Indeed, there is a natural functor of $B\bN^{\oplus I}$-module categories $\beta_{\Cscr^{\op}}^{\op}\colon f^!(\Cscr^{\op})^\op \to \Cscr$  by adjunction inducing a functor of $B\bZ^{\oplus I}$-linear categories $\alpha\colon f^!(\Cscr^{\op})^\op \to f^!(\Cscr)$ whose composition with $\beta_\Cscr$ is isomorphic to $\beta_{\Cscr^{\op}}^{\op}$. Since by Lemma \ref{lem:categorical inverse limit perfection} both $\beta_\Cscr$ and $\beta_{\Cscr^{\op}}^{\op}$ are fully faithful embeddings with the same image, it follows that $\alpha$ is an equivalence, and then $\beta_{\Cscr^{\op}}^L$ provides a right adjoint to $\beta_\Cscr$.
	
	By Remark \ref{rem:formula for perfection monad}, in this case $\Cscr^\perf$ is comonadic over $\Cscr$ with the comonad $T$ given by $ \beta_\Cscr\circ \beta^R_\Cscr$. When $I$ is a singleton one has a simple formula for $T(X)$:
	$$
	T(X)=\lim(\ldots \xrightarrow{\psi_X} X \xrightarrow{\psi_X} X  \xrightarrow{\psi_X}  X  \xrightarrow{\psi_X}  \ldots).
	$$
\end{remark}

\begin{remark}\label{rem:one can first perfect and then take invariants}
	Let us also note that decomposing the projection $B\bN^{\oplus I}\to *$ as the composition $B\bN^{\oplus I}\to B\bZ^{\oplus I} \to *$, one gets natural equivalences
	$$
	\Cscr^{hB\bN^{\oplus I}}\simeq (\Cscr^\perf)^{hB\bZ^{\oplus I}}\simeq (\Cscr^\perf)^{hB\bN^{\oplus I}}.
	$$
	
\end{remark}	

We can also identify $\Cscr_\perf$. To this end, recall that for an $\infty$-category $\Cscr$
with a class of morphisms $S\subset \mathrm{Ar}(\Cscr)$,
a functor $g\colon \Cscr \to \Cscr[S^{-1}]$ identifies $\Cscr[S^{-1}]$ with the Dwyer--Kan localization of $\Cscr$ with respect to $S$ if for any $\infty$-category $\Dscr$ the restriction $g^*\colon \Fun(\Cscr[S^{-1}],\Dscr) \to \Fun(\Cscr,\Dscr)$ is fully faithful with image given by those functors $F\colon \Cscr \to \Dscr$ that send any $s\in S$ to an isomorphism.

\begin{corollary}\label{cor:localization} The functor $\alpha_{\Cscr}\colon \Cscr \to \Cscr_{\perf}$ identifies  $\Cscr_{\perf}$ with the Dwyer--Kan localization of $\Cscr$ with respect to the collection of morphisms $\{\psi_{X,i}\}_{X\in \Cscr,i\in I}$.
\end{corollary}	

\begin{proof}
	We need to check that for any $\infty$-category $\Dscr$ the induced functor $\alpha_\Cscr^*\colon \Fun(\Cscr_{\perf}, \Dscr) \to \Fun(\Cscr, \Dscr)$ is fully faithful with image given by those functors that send the morphisms $\{\psi_{X,i}\}_{X\in \Cscr,i\in I}$ to isomorphisms. Note that we can rewrite $\alpha_\Cscr^*$ as a functor between $B\bN^{\oplus I}$-linear functor categories, namely as
\begin{equation}\label{eq:linear version of restriction from Cperfection}
	\Fun(\Cscr_{\perf}, \Dscr)\simeq \Fun_{B\bN^{\oplus I}}(\Cscr_{\perf}, \Fun(B\bN^{\oplus I}, \Dscr)) \xrightarrow{\alpha_\Cscr^*}  \Fun_{B\bN^{\oplus I}}(\Cscr, \Fun(B\bN^{\oplus I}, \Dscr)) \simeq \Fun(\Cscr,\Dscr).
\end{equation}
	For brevity, let $\Escr:=\Fun(B\bN^{\oplus I}, \Dscr)$. The functor $\alpha_\Cscr\colon \Cscr \to \Cscr_\perf$ is obtained by applying geometric realization to the natural transformation $\Bar_{B\bN^{\oplus I}}^\bullet(B\bN^{\oplus I},\Cscr) \to \Bar_{B\bN^{\oplus I}}^\bullet(B\bZ^{\oplus I},\Cscr)$ between bar constructions; on the $n$-th term it is given by the natural functor $(f,\id, \id)\colon B\bN^{\oplus I} \times (B\bN^{\oplus I})^{n-1} \times \Cscr \to B\bZ^{\oplus I} \times (B\bN^{\oplus I})^{n-1} \times \Cscr$. Here we will use the $B\bN^{\oplus I}$-module structure given by the action on the left-most term. 
	
	We get that the transformation \eqref{eq:linear version of restriction from Cperfection} is obtained as the totalization of the diagram of restriction functors
	$$
	\Fun_{B\bN^{\oplus I}}(B\bZ^{\oplus I} \times (B\bN^{\oplus I})^{\bullet} \times \Cscr, \Escr) \longrightarrow \Fun_{B\bN^{\oplus I}}(B\bN^{\oplus I} \times (B\bN^{\oplus I})^{\bullet} \times \Cscr, \Escr).
	$$
	By Lemma \ref{lem:categorical inverse limit perfection} this functor is fully faithful, with image given by those functors for which the transformations defined by the first copy of $B\bN^{\oplus I}$ are invertible. By the same argument as in the proof of \textit{loc.cit.}, we then get that $\alpha^*_\Cscr$ is fully faithful. Moreover, its image is given by those functors $F\in \Fun(\Cscr,\Dscr)$ that are sent to functors in $\Fun(\Cscr,\Fun(B\bZ^{\oplus I},\Dscr))$ by the functor $\Fun(\Cscr,\Dscr) \to \Fun(\Cscr,\Fun(B\bN^{\oplus I},\Dscr))$ dual to the $B\bN^{\oplus I}$-action on $\Cscr$. As $\Fun(B\bZ^{\oplus I},\Dscr)\to \Fun(B\bN^{\oplus I},\Dscr)$ is fully faithful, this happens if and only if $F(\psi_{X,i})$ is invertible for all $X\in\Cscr$ and $i\in I$. 
\end{proof}

\begin{remark}
One can consider the composition 
	$$\begin{tikzcd}
	\gamma_\Cscr \colon f^*(\Cscr)=\Cscr^\perf \arrow[r,"\beta_\Cscr"]& \Cscr \arrow[r,"\alpha_\Cscr"]& \Cscr_{\perf}=f^!(\Cscr),
\end{tikzcd}$$
a kind of ``norm map''. 
A natural question is whether it is an equivalence. 
In the context of $\Cat_{\infty}$, this is typically not the case, 
even though it is not far from being one. 
More precisely, let $S$ be the collection of morphisms in $\Cscr$ given by $\{\psi_{X,i}\}_{X\in\Cscr, i\in I}$. 
It is not hard to see from Lemma \ref{lem:categorical inverse limit perfection} that $\Cscr^\perf\subset \Cscr$ is exactly the subcategory of $S$-local (and $S$-colocal) objects, i.e., of those objects $X\in \Cscr$ such that $s^*\colon \Map_\Cscr(Y',X)\to \Map_\Cscr(Y,X)$ is an equivalence for all $s\colon Y\to Y'$ in $S$. 
It follows that $\gamma_{\Cscr}$ is fully faithful.
Indeed, the universal property of the Dwyer--Kan localization implies that for any $X\in \Cscr$, the map
\[
\alpha_{\Cscr}\colon \Map_{\Cscr}(-, X)\longrightarrow \Map_{\Cscr_{\perf}}\big(\alpha_\Cscr(-), \alpha_{\Cscr}(X)\big)
\]
exhibits the target as the initial presheaf with a map from $\Map_{\Cscr}(-, X)$ 
that sends all maps in $S$ to equivalences. 
If $X\in \Cscr^{\perf}$ is $S$-local, this map is an equivalence.

However, the functor $\gamma_{\Cscr}$ is typically not essentially surjective. 
For instance, when $\Cscr$ has colimits indexed by $(\bZ^{\oplus I}, \le)$, then 
$\beta^L_\Cscr\colon \Cscr \to \Cscr^\perf$ is the Dwyer--Kan localization of $\Cscr$ with respect to the class $W$ of $S$-local equivalences, namely those morphisms $w$ in $\Cscr$ for which $\beta^L_\Cscr(w)$ is an isomorphism. One has an inclusion $S\subset W$, but it is not a bijection; rather $W$ is a strong saturation of $S$.

This issue disappears in the presentable setting. To see this, consider the monoidal left adjoint functor $f\colon \mathrm{PShv}(B\bN^{\oplus I})\to \mathrm{PShv}(B\bZ^{\oplus I})$ 
and let $\Cscr\in \Mod_{\mathrm{PShv}(B\bN^{\oplus I})}(\PrL)$.
Then the left adjoint $f_*f^!(\Cscr)\to \Cscr$ is equivalent to the functor
	$\beta_{\Cscr}$ from Lemma \ref{lem:categorical inverse limit perfection}. 
	This follows from the fact that limits in $\PrL$ are computed on the level of underlying $\infty$-categories.
On the other hand, the left adjoint $\Cscr\to f_*f^*(\Cscr)$ is the Dwyer--Kan
	localization of $\Cscr$ with respect to the class $W$ of $S$-local equivalences.
	This follows from the proof of Corollary~\ref{cor:localization}, 
	applied to the presentable setting.
In particular, the composite $f_*f^!(\Cscr)\to f_*f^*(\Cscr)$ is then an 
equivalence.
\end{remark}

\section{Deriving a monad}\label{sec:deriving}


In this section we prove a general result about functors between categories of the form $\Mod_k$
and $\Mod_A$ for commutative rings $k$ and $A$ and use it to classify derived plethories.
The material here is strongly influenced by the work~\cite{brantner-mathew}.

\subsection{Derived functors}

Throughout this section, let $k$ and $A$ be commutative rings. 
We will make use of the following subcategories of the $\infty$-category $\Mod_k$ of $k$-module spectra.

\begin{notation}
	Let $k$ be a commutative ring. We will write
	$$
	\Proj_k^\omega\subseteq \Mod_k^\flat\subseteq\Mod_k^\heart
	$$
	for the full subcategories of compact projective $k$-modules and flat $k$-modules, respectively. 
	Recall that flat modules are filtered colimits of free modules; 
	Lazard's theorem~\cite[\href{https://stacks.math.columbia.edu/tag/058G}{Tag 058G}]{stacks} says
that $\Ind(\Proj_k^\omega)\we\Mod_k^\flat$.

	For $a\leq b$, let us write $\Mod_k^{w[a, b]}$ for the full subcategory of $\Mod_k$
	on those $a$-connective $k$-modules with Tor-amplitude in $[a, b]$.
	Let us furthermore write $\Mod_k^{w\leq b}=\bigcup_a \Mod_k^{w[a,b]}$.
	Alternatively, one can characterize $\Mod_k^{w\leq b}$ as the smallest
	subcategory of $\Mod_k$ which contains all $b$-fold suspensions of flat $k$-modules and is closed under coskeletal totalizations
	(see Definition~\ref{def:cosk}).
\end{notation}

We will be interested in the process of extending a functor $\Mod_k^\flat\rightarrow \Mod_A^\flat$
to a functor $\Mod_k\rightarrow \Mod_A$ \cite{brantner-mathew, raksit}, 
and especially in the compatibility of this process with composition.
As a warm-up, let us recall how this works in the connective setting using animation~\cite[Sec.~5.5.8]{htt}.

\begin{lemma}\label{lem:sift-completion}
	Let $\Cscr$ and $\Dscr$ be $(n, 1)$-categories for some $1\leq n\leq \infty$, such that
	$\Dscr$ has sifted colimits, that is, colimits indexed by sifted $\infty$-categories. 
	Suppose that $\iota\colon \Cscr_0\hookrightarrow \Cscr$ is a small full
	subcategory such that for each object $Y\in \Cscr$, the $\infty$-category
	$(\Cscr_0)_{/Y}$ is sifted 
	and the canonical cone $(\Cscr_0)_{/Y}^{\rhd}\to \Cscr$ (with cone point $Y$) is a colimit.
	Then the restriction functor 
	$$
	\iota^*\colon \Fun(\Cscr, \Dscr)\rightarrow \Fun(\Cscr_0, \Dscr)
	$$
	admits a fully faithful left adjoint, whose essential image consists of the functors
	 $F\colon \Cscr\to \Dscr$ that preserve those sifted colimits in $\Cscr$ that are preserved by the functors 
	 $\Map_{\Cscr}(X, -)\colon \Cscr\rightarrow \Spc_{\leq n-1}$ with $X\in \Cscr_0$.
\end{lemma}

\begin{proof}
	Let $F\colon \Cscr_0\rightarrow \Dscr$ be a functor. By our assumption on $\Cscr_0$ and
	the pointwise formula for Kan extensions, $F$ admits a left Kan extension 
	$\iota_!F\colon \Cscr\rightarrow \Dscr$ \cite[Lem.~4.3.2.13]{htt}. 
	This shows that the restriction functor has a fully faithful left adjoint.	
		
	To identify its essential image, let us first check that $\iota_! F$ preserves 
	the desired sifted colimits. To see this, note that there exists a fully faithful functor
	$\Dscr\hookrightarrow \Dscr'$ preserving sifted colimits, 
	where $\Dscr'$ is an $n$-category with small colimits (see \cite[Sec.~5.3.6, 5.5.6]{htt}).
	We may therefore assume that $\Dscr$ has small colimits. In that case,
	the functor $\iota_! F$ fits into a commuting diagram of the form
	$$\begin{tikzcd}[column sep=3pc]
	\Cscr_0\arrow[r, "\iota"]\arrow[rrd, "F"{swap}] & \Cscr\arrow[r, "h"]\arrow[rd, "\iota_!F"] & \Fun(\Cscr_0^{\op}, \Spc)\arrow[r, "\tau_{\leq n-1}"]\arrow[d, "F'"] & \Fun(\Cscr_0^{\op}, \Spc_{\leq n-1})\arrow[ld, "F''"]\\
	& & \Dscr.
	\end{tikzcd}$$
	Here $h$ is the restricted Yoneda embedding, which is fully faithful by (b).
	The functor $F'$ is the unique colimit-preserving extension of $F$, i.e.\ the left
	Kan extension of $F$ along the Yoneda embedding. 
	By transitivity of Kan extensions and $h$ being fully faithful, 
	we have that $F'\we h_!\iota_!F$ and hence $F'\circ h \we \iota_!F$.
	As $\Dscr$ is an $n$-category, $F'$ descends to a unique colimit-preserving
	functor $F''$, as indicated: this follows from \cite[Cor.~5.5.6.22]{htt}, 
	which does not rely on $\Dscr$ being presentable.
	It follows that $\iota_!F\we F''\circ \tau_{\leq n}\circ h$ preserves all sifted
	colimits in $\Cscr$ that are preserved by 
	$\tau_{\leq n}\circ h\colon \Cscr\rightarrow \Fun(\Cscr_0^{\op}, \Spc_{\leq n-1})$.
	These are precisely the colimits from the statement.
	
	Conversely, suppose that $G\colon \Cscr\rightarrow\Dscr$ preserves
	all sifted colimits in $\Cscr$ that are preserved by the functors 
	 $\Map_{\Cscr}(X, -)\colon \Cscr\rightarrow \Spc_{\leq n-1}$ with $X\in \Cscr_0$.
	 Then the counit map $\iota_!\iota^*G\to G$ is an equivalence: 
	 it is an equivalence for all $X\in \Cscr_0$ and both functors preserve 
	 the sifted colimit $Y\simeq \colim_{X\in (\Cscr_0)_{/Y} X}$.
\end{proof}

\begin{example}\label{ex:animated-modules}
	The inclusion $\Proj_k^{\omega}\subseteq \Mod_k^{\cn}$ satisfies the condition
	from Lemma~\ref{lem:sift-completion}, as each $(\Proj_k^{\omega})_{/M}$ 
	has finite coproducts. For each compact projective $k$-module $P$, the functor
	$\Map_k(P, -)\colon \Mod_k^{\cn}\to \Spc$ preserves all sifted colimits.
	Lemma~\ref{lem:sift-completion} therefore reduces to the statement that $\Mod_k^{\cn}$
	is the animation of $\Proj_k^{\omega}$, that is, the free
	$\infty$-category generated by $\Proj_k^{\omega}$ under sifted colimits \cite[Sec.~5.5.8]{htt}.
		
	Lemma~\ref{lem:sift-completion} also applies to the inclusion $\Proj_k^{\omega}\subseteq \Mod_k^{\heart}$.
	In this case, $\Map_k(P, -)\colon \Mod_k^{\heart}\to \Spc_{\leq 0}$ preserves all sifted colimits.
	It follows that $\Mod_k^{\heart}$ is the free $1$-category generated by $\Proj_k^{\omega}$
	under sifted colimits.
\end{example}

\begin{proposition}\label{prop:animation}
	There is a commuting diagram of equivalences of $\bB\bM^\otimes$-monoidal categories
	$$\begin{tikzcd}
	\Fun^{\flat}_{\Sigma}(\Mod^\cn_k, \Mod_A^{\cn})^{\otimes}\arrow[rr, "\we"]\arrow[rd, "\we"{swap}] & & \Fun^{\flat}_{\Sigma}(\Mod^\heart_k, \Mod_A^{\heart})^{\otimes}\arrow[ld, "\we"]\\
	& \Fun_\omega(\Mod_k^{\flat}, \Mod_A^{\flat})^{\otimes}
	\end{tikzcd}$$
	where:
	\begin{enumerate}
	\item[{\em (a)}] $\Fun^{\flat}_{\Sigma}(\Mod^\cn_k, \Mod_A^{\cn})^{\otimes}\subseteq 
	\Fun(\Mod^{\cn}_k, \Mod_A^{\cn})^{\otimes}$ and 
	$\Fun^{\flat}_{\Sigma}(\Mod^\heart_k, \Mod_A^\heart)^{\otimes}\subseteq 
	\Fun(\Mod^\heart_k, \Mod_A^\heart)^{\otimes}$ denote the full $\bB\bM^\otimes$-monoidal subcategories
	of functors preserving sifted colimits and flat modules.
	
	\item[{\em (b)}]  $\Fun_\omega(\Mod_k^{\flat}, \Mod_A^{\flat})^{\otimes}\subseteq \Fun(\Mod_k^{\flat}, \Mod_A^{\flat})^{\otimes}$ 
	is the full $\bB\bM^\otimes$-monoidal subcategory of functors preserving filtered colimits.
	
	\item[{\em (c)}] The two downwards pointing equivalences are given by restriction.
	
	\item [{\em (d)}] The horizontal equivalence sends $F\colon \Mod_k^{\cn}\to \Mod_A^{\cn}$
	to $\pi_0F\colon \Mod_k^{\heart}\to \Mod_A^\heart$, and likewise for endofunctors.
	\end{enumerate}
\end{proposition}

\begin{proof}
	Let us start by noting that the two downwards pointing functors, induced by restriction, 
	do indeed inherit a $\BM^\otimes$-monoidal structure.
	Since $\Mod_k^{\cn}$ is the animation of $\Proj_k^{\omega}$ and $\Mod_k^{\flat}\simeq \Ind(\Proj_k^{\omega})$ by Lazard's theorem,
	restriction induces equivalences
	$$
	\Fun_{\Sigma}(\Mod_k^{\cn}, \Mod_A^{\cn})\simeq \Fun_\omega(\Mod_k^{\flat}, \Mod_A^{\cn})\simeq \Fun(\Proj_k^{\omega}, \Mod_A^{\cn}).
	$$
	This identifies the full subcategories of functors preserving flat modules. 
	Replacing $k$ by $A$ (or vice versa), it follows that the left restriction functor is
	a $\BM^\otimes$-monoidal equivalence.
	Likewise, the right restriction functor is a $\BM^\otimes$-monoidal equivalence 
	because $\Mod_k^{\heart}$ is the free $1$-category with sifted colimits generated
	by $\Proj_k^{\omega}$ by Example \ref{ex:animated-modules}.
	
	We therefore obtain a unique horizontal equivalence making the triangle commute.
	Let us write $F'\colon \Mod_k^{\heart}\rightarrow \Mod_A^{\heart}$ for the image of
	$F\colon  \Mod_k^{\cn}\rightarrow \Mod_A^{\cn}$ under this equivalence.
	This gives functors $F'\pi_0\colon \Mod_k^{\cn} \rightarrow \Mod_A^{\heart}$ 
	and $\pi_0F\colon \Mod_k^{\cn} \rightarrow \Mod_A^{\heart}$,
	which preserve sifted colimits and agree on flat modules. 
	It follows that $F'\pi_0\we \pi_0F$. Restricting to discrete $k$-modules, we
	obtain a natural equivalence between $F'$ and $\pi_0F\colon \Mod_k^\heart\rightarrow \Mod_A^{\heart}$.
\end{proof}

Let us now discuss an extension of the above result to the unbounded derived category.

\begin{definition}[(Co)skeletal objects]\label{def:cosk}
    A simplicial diagram $X\colon\Delta^\op\rightarrow\Cscr$ is skeletal if it is left Kan extended
    from its restriction $X|_{\Delta_{\leq n}^\op}\colon\Delta_{\leq n}^\op\rightarrow\Cscr$ for
    some $n\geq 0$. Similarly, a cosimplicial diagram $Y\colon\Delta\rightarrow\Cscr$ is coskeletal if it
    is right Kan extended from its restriction $Y|_{\Delta_{\leq n}}\colon\Delta_{\leq
    n}\rightarrow\Cscr$ for some $n\geq 0$.
    
    A functor $F\colon \Cscr\rightarrow \Dscr$ is said to preserve coskeletal totalizations
    if for any coskeletal cosimplicial diagram $X\colon\Delta\rightarrow\Cscr$
    with totalization $\Tot(X)$, the augmented cosimplicial diagram $F(\Tot(X))\to F(X)$
    is a limit in $\Dscr$, and similarly for skeletal geometric realizations.
    Note that we do not require the diagram $F(X)$ to be coskeletal itself. 
\end{definition}

\begin{definition}
    Let $\Perf_k^\cn\subseteq\Mod_k^\cn$ denote the full subcategory of connective
    perfect $k$-modules. 
    Equivalently, $\Perf_k^\cn$ consists of the compact objects in $\Mod_k^\cn$. 

    Let $\Perf_k^{w\leq 0}\subseteq\Mod_k$ denote the full subcategory consisting of $k$-linear
    duals of the objects in $\Perf_k^\cn$. Equivalently, $\Perf_k^{w\leq 0}$ consists of the
    perfect $k$-modules with Tor-amplitude in $[-n,0]$ for some $n\geq 0$. If $M\in\Perf_k^{w\leq
    0}$, we say that $M$ is a coconnectively flat perfect complex. The notation is chosen to evoke
    the notion of weight structures, although we do not use these elsewhere.
\end{definition}

\begin{warning}
    Note that a $k$-module which is perfect and coconnective need not be in $\Perf_{k}^{w\leq 0}$. An
    example is given by $\bZ/p$ as a $\bZ$-module.
\end{warning}

The following lemma is a standard application of the philosophy of
animation~\cite[Sec.~5.5.8]{htt}, but restricted to small objects as in~\cite{brantner-mathew}.

\begin{lemma}\label{lem:universal}\label{lem:tfae}
    Let $k$ be a commutative ring.
\begin{enumerate}
        \item[{\em (a)}]
            If $\Cscr$ admits sifted colimits, then the restriction functor
            $$\Fun(\Perf_k^\cn,\Cscr)\rightarrow\Fun(\Proj_k^\omega,\Cscr)$$
            admits a fully faithful left adjoint, whose essential image 
            $\Fun_\sigma(\Perf_k^\cn,\Cscr)\subseteq\Fun(\Perf_k^\cn,\Cscr)$ 
            consists of the functors $F\colon\Perf_k^\cn\rightarrow\Cscr$ 
            satisfying the following equivalent conditions:
            \begin{enumerate}
                \item[{\em (i)}] $F$ is left Kan extended from its restriction to $\Proj_k^\omega$;
                \item[{\em (ii)}] $F$ preserves skeletal geometric realizations;
                \item[{\em (iii)}] $F$ preserves geometric realizations of skeletal simplicial
                    diagrams of finitely generated projective $k$-modules;
                \item[{\em (iv)}] for each $X\in\Perf_k^\cn$ there is some object
                    $X_\bullet\in \s\Proj_k^\omega$ such that $|X_\bullet|\we X$ and
                    $|F(X_\bullet)|\we F(X)$.
            \end{enumerate}
        \item[{\em (b)}] If $\Dscr$ admits sifted limits, then the restriction functor
            $$\Fun(\Perf_k^{w\leq 0},\Dscr)\rightarrow\Fun(\Proj_k^\omega,\Cscr)$$
            admits a fully faithful right adjoint, whose essential image 
            $\Fun_\epsilon(\Perf_k^{w\leq 0},\Dscr)\subseteq\Fun(\Perf_k^{w\leq 0},\Dscr)$ 
            consists of the functors $F\colon\Perf_k^{w\leq 0}\rightarrow\Dscr$ 
            satisfying the following equivalent conditions:
            \begin{enumerate}
                \item[{\em (i)}] $F$ is right Kan extended from its restriction to $\Proj_k^\omega$;
                \item[{\em (ii)}] $F$ preserves coskeletal totalizations; 
                \item[{\em (iii)}] $F$ preserves totalizations of coskeletal cosimplicial
                    diagrams of finitely generated projective $k$-modules;
                \item[{\em (iv)}] for each $X\in\Perf_k^{w\leq 0}$ there is some object
                    $X^\bullet\in \c\Proj_k^\omega$ such that $X\we\Tot(X^\bullet)$ and
                    $F(X)\we\Tot(F(X^\bullet))$.
            \end{enumerate}
    \end{enumerate}
\end{lemma}

\begin{proof}
	Note that part (b) follows from part (a) using the equivalence
	$\Perf_k^{w\leq 0}\we(\Perf_k^\cn)^\op$ induced by $k$-linear duality.
	We therefore only give the proof of (a). 
	The inclusion $\Proj_k^{\omega}\subseteq \Perf_k^{\cn}$ satisfies the 
	conditions of Lemma~\ref{lem:sift-completion}.
	The restriction functor therefore admits a fully faithful left adjoint, 
	whose essential image consists of functors $F\colon \Perf_k^\cn\rightarrow \Cscr$ 
	that are left Kan extended from their restriction to $\Proj_k^{\omega}$.

	It remains to prove that conditions (i)-(iv) are equivalent. 
	Condition (i) implies (ii) by Lemma~\ref{lem:sift-completion}, because 
	mapping out of a compact projective $k$-module preserves skeletal geometric
	realizations in $\Perf_k^{\cn}$.
	It is clear that (ii) implies (iii), and (iii) implies (iv) because 
	every object $X\in \Perf_k^\cn$ can be obtained as the geometric realization of a skeletal simplicial diagram with values in $\Proj_k^\omega$.
	Finally, note that for two functors $F$ and $G$ satisfying (iv), a natural transformation
	$\alpha\colon G\rightarrow F$ is an equivalence if its restriction to $\Proj_k^\omega$ is such.
	In particular, this applies if we take $G$ to be the left Kan extension
	of the restriction of $F$ to $\Proj_k^\omega$, and $\alpha$ to be the counit map. 
	This shows that (iv) implies (i).
\end{proof}

\begin{remark}
    Using Lemma~\ref{lem:universal}, the left Kan extension $\L F$ of a functor
    $F\colon\Proj_k^\omega\rightarrow\Cscr$
    can be computed as follows. For
    $X\in\Perf_k^\cn$ one takes a skeletal
    simplicial object $P_\bullet\in\Fun(\Delta^\op,\Proj_k^\omega)$ with the property that
    $\colim_{\Delta^{\op}}\we X$. For example, one can find an explicit non-negatively
    graded, bounded above
    chain complex in $\Proj_k^\omega$ which is quasi-isomorphic to $X$ and apply the
    Dold--Kan theorem to obtain $P_\bullet$. This object will be skeletal, for example
    by~\cite[\href{https://kerodon.net/tag/05GA}{Tag 05GA}]{kerodon}.
    Then, $(\L F)(X)\we\colim_{\Delta^\op} F(P_\bullet)$.
    A similar remark applies to the value of the right Kan extension at coconnectively flat perfect
    objects.
\end{remark}


\begin{remark}
    As $\Mod_k^\cn$ is the animation of $\Proj_k^\omega$, to compute the value of the left Kan extension
    of a functor $F\colon\Proj_k^\omega\rightarrow\Cscr$ on $X\in\Perf_k^\cn$,
    by~\cite[Lem.~5.5.8.14]{htt}, one can use an
    arbitrary simplicial flat module $P_\bullet$ with $|P_\bullet|\we X$ and find that $(\L
    F)(X)\we\colim_{\Delta^\op} F(P_\bullet)$. It is not necessary to restrict to skeletal
    simplicial objects. An entirely similar remark applies to the case of right Kan extension,
    except that one uses cosimplicial objects of $\Pro(\Proj_k^\omega)$ 
    (which is equivalent to the opposite of $\Flat_k^\omega$).
\end{remark}

\begin{remark}\label{rem:additive-case}
    The equivalence $\Fun_\sigma(\Perf_k^\cn,\Cscr)\we\Fun(\Proj_k^\omega,\Cscr)$
    restricts to an equivalence $$\Fun_\rex(\Perf_k^\cn,\Cscr)\we\Fun_\sqcup(\Proj_k^\omega,\Cscr),$$
    where $\Fun_\rex(\Perf_k^\cn,\Cscr)\subseteq\Fun(\Perf_k^\cn,\Cscr)$ is the full
    subcategory of functors which preserve finite colimits and $\Fun_\sqcup(\Proj_k^\omega,\Cscr)$
    is the full subcategory of functors which preserve finite coproducts.
    Similarly, for functors out of $\Perf_k^{w\leq 0}$, restriction induces an equivalence
    $$\Fun_\lex(\Perf_k^{w\leq 0},\Cscr)\we\Fun_\sqcap(\Proj_k^\omega,\Cscr),$$ with the obvious
    notation.
\end{remark}

\begin{example}
    Let $\L\eta_p\colon\Perf_\bZ^\cn\rightarrow\Mod_\bZ$ denote the (derived) $\eta_p$
    functor, restricted to connective perfect complexes. On a cochain complex of free abelian
    groups $M^\bullet$, one identifies $\L\eta_p(M)$
    with the subcomplex $\eta_p(M)^\bullet\subseteq M[1/p]^\bullet$ with $\eta_p(M)^i=\{x\in
    p^iM^i\colon \d(x)\in p^{i+1}M^{i+1}\}$. There is a natural identification $\H^*(\L\eta_pM)\iso\H^*(M)/\H^*(M)[p]$, the quotient of the cohomology of $M$ by the
    $p$-torsion in the cohomology.

    We see from the formula for cohomology that $\L\eta_p$ (restricted to connective perfect complexes as above) preserves
    finite coproducts.
    It also preserves the Bar complexes for $\bZ[n]$. For example, the natural map
    $|\L\eta_p(\B_\bullet(\bZ[n-1]))|\rightarrow\L\eta_p\bZ[n]\we\bZ[n]$ is an equivalence.
    If, as a consequence, $\L\eta_p$ were to
    preserve all skeletal geometric realizations of perfect complexes, it would therefore be left
    Kan extended off of $\Proj_\bZ^\omega$. But, on $\Proj_\bZ^\omega$ it gives the identity
    functor, so we would find $M\we\L\eta_pM$ for all $M$. As this is not the case, we see that,
    even in the presence of additivity, preserving Bar constructions of shifts of free objects is
    not sufficient to preserve all skeletal geometric realizations. For a specific example,
    $\L\eta_p(\bF_p[0])\we 0$, but if $\L\eta_p$ preserved skeletal geometric realizations, we would
    find $\L\eta_p(\bF_p[0])\we\bF_p[0]$.
\end{example}

\begin{remark}
    In the work above, we are not using additivity of $\Proj_k^\omega$ in any particular way. The
    main thing, besides idempotent completeness, is the fact that $\Perf_k^\cn$ is freely
    generated under skeletal geometric realizations by $\Proj_k^\omega$.
\end{remark}

The following definition appears in~\cite{brantner-mathew}.

\begin{definition}
    Assume that $\Cscr$ admits sifted colimits.
    Let $\Fun_\delta(\Perf_k^{w\leq 0},\Cscr)\subseteq\Fun(\Perf_k^{w\leq 0},\Cscr)$ denote the full
    subcategory of functors which preserve colimits of skeletal simplicial diagrams
    $X_\bullet\in\Fun(\Delta^\op,\Perf_k^{w\leq 0})$ whose colimit in $\Mod_k$ is in $\Perf_{k}^{w\leq
    0}$.
\end{definition}

\begin{lemma}\label{lem:bm}
    Suppose that $\Cscr$ is an $\infty$-category admitting sifted colimits.
    Restriction induces an equivalence $\Fun_\Sigma(\Mod_k,\Cscr)\rightarrow\Fun_\delta(\Perf_k^{w\leq
    0},\Cscr)$ with inverse given by left Kan extension.
\end{lemma}

\begin{proof}
    This is~\cite[Prop.~3.14]{brantner-mathew}, but in the proof, which relies
    on~\cite[Prop.~3.11]{brantner-mathew}, one must ensure that, when it appears in their proof, $\Perf_{k,[-n,0]}$ refers to perfect
    $k$-modules having Tor-amplitude in $[-n,0]$ as Brantner and Mathew work only over fields.
\end{proof}

\begin{construction}
    Recall that $k$ and $A$ are commutative rings.
    Let $\Fun(\Mod_k,\Mod_A)^\otimes\rightarrow\BM^\otimes$ be the $\BM^\otimes$-monoidal
    $\infty$-category classifying $\Fun(\Mod_k,\Mod_A)$ together with its
    $(\End(\Mod_A),\End(\Mod_k))$-bimodule structure. The category $\BM^\otimes_{\langle 1\rangle}$
    has objects (colors) $\afrak_l$, $\mfrak$, and $\afrak_r$. The fibers of $\Fun(\Mod_k,\Mod_A)^\otimes$
    over these objects are $\End(\Mod_A)$, $\Fun(\Mod_k,\Mod_A)$, and $\End(\Mod_k)$.

    We let $\Fun_\Sigma(\Mod_k,\Mod_A)^\otimes\subseteq\Fun(\Mod_k,\Mod_A)^\otimes$ be the
    sub-$\BM$-monoidal category on the functors that preserve sifted colimits and we let
    $\Fun_\Sigma^{\flat}(\Mod_k,\Mod_A)^\otimes\subseteq\Fun_\Sigma(\Mod_k,\Mod_A)^\otimes$ be the
    $\BM$-monoidal subcategory of functors which furthermore preserve flat objects. Restriction
    induces a
    $\BM^\otimes$-monoidal functor
    $$\Fun_\Sigma^{\flat}(\Mod_k,\Mod_A)^\otimes\rightarrow\Fun_\omega(\Mod_k^\flat,\Mod_A^\flat)^\otimes,$$
    where on the right we have a $\BM^\otimes$-monoidal $1$-category of functors on flat modules
    which preserve filtered colimits.
\end{construction}

\begin{definition}
    Suppose that $\Cscr$ is an $\infty$-category which admits finite limits and colimits.
    We say that a functor $F\colon \Proj_k^\omega\rightarrow\Cscr$ is right extendable if its right
    Kan extension $\R F\colon\Perf_k^{w\leq 0}\rightarrow\Cscr$ belongs to $\Fun_\delta(\Perf_k^{w\leq
    0},\Cscr)$, i.e., if $\R F$ preserves skeletal simplicial diagrams whose colimit in $\Mod_k$ is  
    in $\Perf_k^{w\leq 0}$.
\end{definition}


\begin{proposition}\label{prop:joost}
    Let $\Fun_\Sigma'(\Mod_k,\Mod_A)^\otimes\subseteq\Fun_\Sigma^{\flat}(\Mod_k,\Mod_A)^\otimes$ be the
    sub-operad on the functors which preserve coskeletal totalizations of cosimplicial flat
    modules. The restriction functor
    $$\Fun_\Sigma'(\Mod_k,\Mod_A)^\otimes\rightarrow\Fun_\omega(\Mod_k^\flat,\Mod_A^\flat)^\otimes$$
    is a fully faithful map of $\infty$-operads with essential image the sub-operad on the functors
    $\Mod_k^\flat\rightarrow\Mod_A^\flat$ which preserve filtered colimits and whose restrictions
    to functors $\Proj_k^\omega\rightarrow\Mod_A$ are right extendable.
\end{proposition}

\begin{proof}
    To prove fully faithfulness, we can reduce to multilinear mapping spaces by
    Lemma~\ref{lem:fully_faithful_operad_map} and further split by cases depending on the colors of the objects in
    question. We give the argument in the most complicated case. Namely, suppose that
    $F_1,\ldots,F_m\in\End_\Sigma'(\Mod_A)$, $G,J\in\Fun'_\Sigma(\Mod_k,\Mod_A)$, and
    $H_1,\ldots,H_n\in\End_\Sigma'(\Mod_k)$. Let $(-)^\flat$ denote the restrictions to flat
    modules. We have that
    $$\Mul_{\Fun(\Mod_k,\Mod_A)^\otimes}(F_1,\ldots,F_m,G,H_1,\ldots,H_n,J)\we\Map_{\Fun(\Mod_k,\Mod_A)}(F_1\circ\cdots\circ
    F_m\circ G\circ H_1\circ\cdots\circ H_n,J)$$ and similarly for the restrictions to flats.
    Note that the composition, say $K=F_1\circ\cdots\circ F_m\circ G\circ H_1\circ\cdots\circ H_n$,
    is not necessarily in $\Fun'_\Sigma(\Mod_k,\Mod_A)$ but lies in
    $\Fun^{\flat}_\Sigma(\Mod_k,\Mod_A)$. However, preservation of sifted colimits implies that
    \begin{equation}\label{eq:kj}
        \Map_{\Fun_\Sigma(\Mod_k,\Mod_A)}(K,J)\we\Map_{\Fun_\delta(\Perf_{R,\leq
        0},\Mod_A)}(K|_{\Perf_k^{w\leq 0}},J|_{\Perf_k^{w\leq 0}})
    \end{equation}
    by Lemma~\ref{lem:bm}.
    As $J$ is right Kan extended from $\Proj_k^\omega$, the terms of~\eqref{eq:kj} are further
    equivalent to
    $$\Map_{\Fun(\Proj_k^\omega,\Mod_A)}(K|_{\Proj_k^\omega},J|_{\Proj_k^\omega})$$
    using the adjunction of Lemma~\ref{lem:universal}(b).
    Now, note that $K^\flat\we F_1^\flat\circ\cdots\circ F_m^\flat\circ G^\flat\circ H_1^\flat\circ
    H_n^\flat$ and
    $$\Map_{\Fun_\omega(\Mod_k^\flat,\Mod_A^\flat)}(K^\flat,J^\flat)\we\Map_{\Fun(\Proj_k^\omega,\Mod_A^\flat)}(K|_{\Proj_k^\omega},J|_{\Proj_k^\omega}).$$
    However, $\Fun(\Proj_k^\omega,\Mod_A^\flat)\subseteq\Fun(\Proj_k^\omega,\Mod_A)$ is precisely
    the full subcategory of functors $\Mod_k^\flat\rightarrow\Mod_A$ taking flat values.
    This proves fully faithfulness.

    The statement about the essential image can be checked at the level of the fibers
    $\End(\Mod_A)$, $\End(\Mod_k)$, and $\Fun(\Mod_k,\Mod_A)$. We explain the case of
    $\Fun(\Mod_k,\Mod_A)$. If $G\colon\Mod_k\rightarrow\Mod_A$ preserves sifted
    colimits, then its restriction to $\Perf_k^{w\leq 0}$ preserves coconnective skeletal geometric
    realizations and coskeletal totalizations. The latter property implies it is right Kan extended
    off of $\Proj_k^\omega$, so the restriction $G^{\flat}$ is contained
    in the claimed subcategory. Conversely, let $G'\colon \Mod_k^\flat\rightarrow\Mod^{\flat}_A$ 
    be a functor preserving filtered colimits, whose restriction to $\Proj_k^\omega$ is right extendable. 
    Its right Kan extension along $\Proj_k^\omega\to \Perf_k^{w\leq 0}$ preserves coconnective skeletal geometric
    realizations, by definition. Under the equivalence
    $\Fun_\Sigma(\Mod_k,\Mod_A)\we\Fun_\delta(\Perf_k^{w\leq 0},\Mod_A)$ of Lemma~\ref{lem:bm},
    this condition implies $G'\simeq G^{\flat}$ for a functor $G\colon \Mod_k\to \Mod_A$ 
    preserving sifted colimits and coskeletal totalizations of cosimplicial objects in $\Proj^\omega_k$. 
    Since every $n$-coskeletal cosimplicial diagram of flat modules is a filtered colimit
    of $n$-coskeletal objects in $\Proj^\omega_k$ and $n$-coskeletal totalizations commute with filtered colimits, we conclude that $G\in \Fun'_\Sigma(\Mod_k, \Mod_A)$, as desired.
\end{proof}

We can use Proposition~\ref{prop:joost} to construct derived monads and functors between
$\infty$-categories of modules for derived monads. Let $\End_\omega^\RE(\Mod_k^{\mathrm{flat}})$ be
the $\infty$-category of endofunctors of $\Mod_k^\flat$ which preserve filtered colimits and are
right extendable.

\begin{corollary}
    Let $k$ be a ring and let
    $\Alg(\End'_\Sigma(\Mod_k))\hookrightarrow \Alg(\End(\Mod_k))$
    be the full subcategory of monads on $\Mod_k$ which preserve sifted colimits, flat modules, and
    coskeletal totalizations of cosimplicial flat modules. Restriction to the flat $k$-modules defines an equivalence
    $\Alg(\End'_\Sigma(\Mod_k))\to \Alg(\End_\omega^\RE(\Mod_k^{\mathrm{flat}}))$.
\end{corollary}

\begin{proof}
    We use Lemma~\ref{lem:suboperad_algebras} and Proposition~\ref{prop:joost}.
\end{proof}

\begin{example}
    Let $A$ be a flat $k$-algebra. The monoid $\Map_{\End(\Mod_A)}(\id_{\Mod_A},\id_{\Mod_A})$ of
    $k$-linear natural transformations of the identity functor is given by the usual center of $A$
    since it can be reduced to the discrete, flat case.
\end{example}

\begin{remark}
    Suppose that $T$ is a monad on $\Mod_k$ which preserves sifted colimits and flat modules.
    Suppose further that $T$, as a functor, is a filtered colimit of excisively polynomial
    functors. Then, $T$ preserves coskeletal totalizations of cosimplicial flat modules
    by~\cite[Prop.~3.38]{brantner-mathew}. Consequently, $T$ belongs to $\Alg(\End'_\Sigma(\Mod_k))$ and
    is thus ``derived''.
    Conversely, consider a monad $T$ on $\Mod_k^{\mathrm{flat}}$ whose underlying functor is a filtered colimit of functors of finite degree, preserving filtered colimits. Then $T$ is right extendable by~\cite[Thm.~3.27]{brantner-mathew}.
\end{remark}

\begin{example}
	The monad $\Sym\colon \Mod_k^{\mathrm{flat}}\rightarrow \Mod_k^{\mathrm{flat}}$
	is a direct sum of functors of finite degree and hence extends to a monad
	$\LSym\colon \Mod_k\rightarrow \Mod_k$. Its category of algebras is the 
	category $\DAlg_k$ of derived commutative $k$-algebras.
\end{example}

\subsection{Derived plethories}\label{sec:derived_plethories}

Let $k$ be a commutative ring. Consider the full sub-categories
$$
\End_\Sigma^{\flat}(\DAlg_k)\subseteq \End(\DAlg_k), \qquad \End_\Sigma^{\flat}(\Ani_k)\subseteq \End(\Ani_k)\qquad \End_\Sigma^{\flat}(\CAlg_k^\heart)\subseteq \End(\CAlg_k^\heart)
$$ 
of endofunctors that preserve sifted colimits and send $\Sym(M)$ to a flat (discrete) $k$-algebra
whenever $M$ is a flat $k$-module.
Note that these subcategories are closed under composition and 
that any endofunctor in $\End_\Sigma^{\flat}(\DAlg_k)$ preserves animated $k$-algebras.

The results of the previous sections then have the following consequence.

\begin{proposition}\label{prop:derived-monad}
	There is a commuting square of monoidal $\infty$-categories
	$$\begin{tikzcd}
	\End^{\flat}_{\Sigma}(\DAlg_k)\arrow[r]\arrow[d, "\we"] & \End_{\Sigma}^{\flat}(\Ani_k)\arrow[r, "\we"]\arrow[d, "\we"] & \End_{\Sigma}^{\flat}(\CAlg_k^\heart)\arrow[d, "\we"]\\
	\BMod_{\LSym}\big(\End^{\flat}_\Sigma(\Mod_k)\big)\arrow[r] & \BMod_{\LSym}\big(\End^{\flat}_\Sigma(\Mod^{\cn}_k)\big)\arrow[r, "\we"] & \BMod_{\Sym}\big(\End^{\flat}_\Sigma(\Mod^{\heart}_k)\big),
	\end{tikzcd}$$
	where the vertical functors conjugate an endofunctor by the free-forgetful adjunction. 
	The left horizontal functors are given by restriction to connective objects,
	and the right horizontal functors send endofunctors $F\colon \Ani_k\rightarrow \Ani_k$ and $G\colon \Mod_k^{\cn}\rightarrow \Mod_k^{\cn}$ to
	$\pi_0F\colon \CAlg_k^\heart\rightarrow \CAlg_k^\heart$ and $\pi_0G\colon \Mod_k^{\heart}\rightarrow \Mod_k^{\heart}$.
\end{proposition}

\begin{proof}
	The vertical equivalences all follow in the same way from Corollary~\ref{cor:bimodules_monoidal}. 
	Let us therefore only discuss the equivalence between $\End^{\flat}_{\Sigma}(\DAlg_k)$
	and $\LSym$-bimodules in $\End^{\flat}_{\Sigma}(\Mod_k)$.
	By Corollary~\ref{cor:bimodules_monoidal}, conjugation with the free-forgetful
	adjunction $\LSym\colon \DAlg_k\leftrightarrows \Mod_k\colon U$ induces 
	an equivalence of monoidal $\infty$-categories
	$$
	\End'(\DAlg_k)\we \BMod_{\LSym}(\End'(\Mod_k))
	$$
	of endofunctors preserving geometric realizations.
	It remains to show that $F\in \End^\flat_\Sigma(\DAlg_k)$ if and only if
	$U\circ F\circ \LSym\in \End^{\flat}_{\Sigma}(\Mod_k)$.	
	To see this, suppose that $F\colon \DAlg_k\to \DAlg_k$ preserves geometric realizations. 
	If $F$ preserves sifted colimits, then $U\circ F\circ \LSym$ preserves sifted colimits as well.
	Conversely, suppose that $U\circ F\circ \LSym$ preserves sifted colimits.
	Then $F\circ \LSym$ preserves sifted colimits as well, so that each term 
	in the bar construction $\Bar_\bullet(F\circ \LSym, U\circ \LSym, U)$
	preserves sifted colimits. Taking the geometric realization, we find that $F$
	preserves sifted colimits.
	
	The commutativity of the left square follows directly from the fact that the free-forgetful
	adjunction $\LSym\colon \DAlg_k\leftrightarrows \Mod_k\colon U$ preserves connective
	objects. For the right square, recall from Proposition~\ref{prop:animation}
	that there are unique monoidal equivalences
	$\End^{\flat}_\Sigma(\Mod^{\cn}_k)\we \End^{\flat}_\Sigma(\Mod^{\heart}_k)
	\we \End_\omega(\Mod_k^{\flat})$, where the right equivalence and the 
	composite are given by restriction to flat modules.
	This induces monoidal equivalences on categories of bimodules
	$$
	\BMod_{\LSym}\big(\End^{\flat}_\Sigma(\Mod^{\cn}_k)\big)\we 
	\BMod_{\Sym}\big(\End^{\flat}_\Sigma(\Mod^{\heart}_k)\big)
	\we \BMod_{\Sym}\big(\End_\omega(\Mod_k^{\flat})\big)
	$$
	where the right equivalence and the composite equivalence are given by restriction.
	There is then a unique monoidal equivalence 
	$\End^{\flat}_{\Sigma}(\Ani_k)\we \End^{\flat}_{\Sigma}(\CAlg_k^\heart)$
	making the right square commute. To identify the underlying functor,
	let us write $F'\in \End^{\flat}_{\Sigma}(\CAlg_k^\heart)$ for the endofunctor corresponding
	to $F\in \End^{\flat}_{\Sigma}(\Ani_k)$. We claim that $F'$ is naturally equivalent to
	$\pi_0F\colon \CAlg_k^\heart\to \CAlg_k^\heart$. Since both of these functors preserve sifted colimits, 
	it suffices	to provide a natural isomorphism between their restrictions to polynomial algebras.
	Unravelling the proof of Corollary~\ref{cor:bimodules_monoidal} and
	Proposition \ref{prop:animation}, we have natural isomorphisms of functors 
	$\Poly_k\rightarrow \CAlg_k^\heart$ of the form
	$$
	F'(A)\we \big|\Bar_\bullet(F'\circ \Sym, U\circ \Sym, UA)\big| \we 
	\big|\Bar_\bullet(F\circ \LSym, U\circ \LSym, UA)\big|\we \pi_0 F(A).
	$$
	Indeed, the first and last isomorphisms follow from the description of the equivalence
	in Corollary~\ref{cor:bimodules_monoidal}.
	The middle isomorphism follows from the fact that $U\circ F'\circ \Sym$ and 
	$U\circ F\circ \LSym$ restrict to naturally isomorphic $\Sym$-bimodules
	in $\End_\omega(\Mod_k^{\flat})$ by construction.
\end{proof}

\begin{corollary}\label{cor:dalg-monads}
    Consider the sub-operad $\End_\Sigma'(\DAlg_k)\subseteq\End^{\flat}_\Sigma(\DAlg_k)$
    of sifted colimit-preserving functors $F$ such that $F\circ\LSym_k$ preserves
     flat objects and coskeletal totalizations of cosimplicial flat modules.
     \begin{enumerate}
     \item[{\em (a)}] The composite 
     $
     \End_\Sigma'(\DAlg_k)\hookrightarrow\End^{\flat}_\Sigma(\DAlg_k)\rightarrow \End^{\flat}_{\Sigma}(\CAlg_k^\heart)
     $
     induces an equivalence between $\End_\Sigma'(\DAlg_k)$ and
     the sub-operad of endofunctors with the additional property that 
     $U\circ F\circ \LSym_k\colon \Proj_k^{\omega}\rightarrow \Mod_k^{\flat}$
     is right extendable.
     
     \item[{\em (b)}] $\End_\Sigma'(\DAlg_k)$ is equivalent to the sub-operad of 
     $(\Sym,\Sym)$-bimodules in $\End_\omega(\Mod_k^\flat)$ which are right extendable.
    \end{enumerate} 
     
\end{corollary}

\begin{proof}
	Under the vertical equivalences from Proposition~\ref{prop:derived-monad}, these
	equivalences are induced by the fully faithful map of $\infty$-operads
	$
	\End_\Sigma'(\Mod_k)^{\otimes}\rightarrow\End_\Sigma(\Mod_k^{\heart})
	\we \End_\omega(\Mod_k^\flat)^{\otimes}
	$
	by taking $(\LSym, \LSym)$-bimodules.
\end{proof}

Note that the composite equivalence in Theorem~\ref{thm:extending-monad-intro} is a consequence of part (b) of the above corollary, while the second equivalence in Theorem~\ref{thm:extending-monad-intro} follows from Proposition~\ref{prop:joost}.

\begin{example}[Frobenius]\label{ex:frobenius}
    Let $k$ be a commutative ring with $\Spec k$ connected.
    Consider the monoid $\Map_{\End(\DAlg_{k})}(\id_{\DAlg_k},\id_{\DAlg_k})$ of natural
    endomorphisms of the identity functor on $\DAlg_k$. By Corollary~\ref{cor:bimodules}, this monoid agrees
    with the monoid of endomorphisms of $\LSym$ as an $(\LSym,\LSym)$-bimodule in the
    $\BM^\otimes$-monoidal category $\Fun(\Mod_k,\Mod_k)^\otimes$. By Proposition~\ref{prop:joost},
    the monoid is equivalent to the monoid of endomorphisms of $\Sym$ as a $(\Sym,\Sym)$-bimodule on flats.
    Such an endomorphism is now determined by the image of $x$ under $k[x]\rightarrow k[x]$. Let $f(x)$
    be the image. We must have $f(rs)=f(r)f(s)$, so $f(x)=x^n$ for some $n$ (locally on
    $\Spec k$ and hence on all of $\Spec k$). Additivity then implies that $n$ is a prime power
    $p^f$ where $p$ is the characteristic of $k$.
    We see that the Frobenius endomorphism of derived
    commutative $\bF_p$-algebras does indeed exist.
\end{example}

We will now apply these results in the setting of plethystic algebra studied 
by Borger and Wieland in~\cite{borger-wieland}. Recall that a plethory is a monad $T\colon \CAlg_k^\heart\rightarrow \CAlg_k^\heart$
which admits a right adjoint $R_T\colon \CAlg_k^\heart\rightarrow \CAlg_k^\heart$.
The right adjoint $R_T$ can also be viewed as an affine ring scheme: indeed, we can identify
$$
R_T\cong \Spec(T(k[x]))
$$ 
for a certain biring structure on $T(k[x])$. The data of a plethory can then equivalently
be encoded as an algebra structure on $T(k[x])$ in the category of $k$-birings.
This is a kind of transmutation argument in the sense of Drinfeld. We will call a plethory flat if
$T(k[x])$ is a flat commutative $k$-algebra.

We will be interested in the following variant of this notion for derived rings.

\begin{definition}[Derived flat plethories]\label{def:derived_plethory}
    Let $k$ be a ring. A monad $T\colon \DAlg_k\rightarrow \DAlg_k$ is
    called a derived flat plethory if it satisfies the following conditions.
    
    \begin{enumerate}
    \item[(a)] $T$ admits a right adjoint $R_T\colon \DAlg_k\rightarrow \DAlg_k$.
    
    \item[(b)] $T(k[x])$ is a flat discrete $k$-algebra.
    
    \item[(c)] Suppose that $M\rightarrow N\rightarrow k[-n]$ is
    a cofiber sequence in $\Perf_k^{w\leq 0}$ and let
    $B^\bullet\colon \Delta\to \DAlg_k$
    be the \v{C}ech complex of $T(\LSym_k(M))\rightarrow T(\LSym_k(N))$. 
    Then $T(\LSym_k(M))\rightarrow \Tot(B^\bullet)$ is an equivalence.
    \end{enumerate}
    We will write $\LSym_k^T\colon \Mod_k\rightarrow \Mod_k$ for the composite monad.
\end{definition}

\begin{remark}\label{rem:flat-plethory}
	Let $T$ be a derived flat plethory. Then $\LSym_k^T(M)$ is a flat $k$-algebra
	for any flat $k$-module $M$. This follows from condition (b)
	by taking sums and filtered colimits. 
	Similarly, $T$ preserves animated rings, because it preserves sifted colimits and
	sends polynomial rings to (discrete) animated rings.
\end{remark}

\begin{lemma}\label{lem:cech descent vs right extended}
	Let $T\colon \DAlg_k\rightarrow \DAlg_k$ be a left adjoint functor.
	Then $T$ satisfies condition \emph{(c)} from Definition \ref{def:derived_plethory} 
	if and only if $\LSym_k^T\colon \Perf_k^{w\leq 0}\rightarrow \Mod_k$
	is right Kan extended from its restriction to $\Proj_k^\omega$.
\end{lemma}

\begin{proof}
	The \v{C}ech complex of the map of derived algebras $\LSym^T_k(M)\to \LSym^T_k(N)$ is equivalent
	to the image of the \v{C}ech conerve of $M\to N$ under $\LSym_k^T$, since $\LSym_k^T$ preserves colimits.
	If $\LSym_k^T\colon \Perf_k^{w\leq 0}\rightarrow \Mod_k$ is right Kan extended
	from its restriction to $\Proj_k^\omega$, then condition (c) therefore holds
	by Lemma~\ref{lem:universal}. 
	
	Conversely, let $\R\LSym^T_k$ denote the right Kan extension of the restriction
	of $\LSym_k^T$ to $\Proj_k^\omega$.
    Thus, there is a natural transformation $\mu\colon \LSym^T_k\rightarrow\R\LSym^T_k$ of functors
    $\Perf_k^{w\leq 0}\rightarrow\DAlg_k$.
    We will prove by induction on $n\geq 0$ that 
    $$
    \mu\colon \LSym^T_k(M)\rightarrow\R\LSym^T_k(M)
    $$
    is an equivalence for all $M\in \Perf_k^{w[-n, 0]}$.
    For $n=0$, we have that $M\in \Proj_k^{\omega}$ and $\mu$ is an equivalence by construction.
    
    Let us now take $n\geq 1$ and $M\in \Perf_k^{w[-n, 0]}$. 
    Let $r\geq 1$ be the minimal number of generators of $\pi_{-n}(M)$, i.e.\ the 
    smallest integer such that there exists a map $\gamma\colon k[-n]^{\oplus r}\rightarrow M$
    with cofiber in $\Perf_k^{w[1-n, 0]}$.
    We will prove by induction on $r$ that $\mu$ is an equivalence.
    Consider the pushout diagram of $k$-modules
    $$\begin{tikzcd}
    k[-n]\arrow[r, "{(\id, 0, \dots, 0)}"]\arrow[d] &[3ex] k[-n]^{\oplus r}\arrow[r, "\gamma"] & M\arrow[d]\\
    0\arrow[rr] & & N
    \end{tikzcd}$$
    and let $N^\bullet$ be the \v{C}ech complex of the map $M\rightarrow N$.
    Since the cofiber of $M\rightarrow N$ is $k[1-n]$, each term of the
    \v{C}ech complex is given by $N^s\we N\oplus k[1-n]^{\oplus s}$.
    Consequently, each $N^s\in \Perf_k^{w[1-n, 0]}$ if $r=1$. If $r>1$, then
    each $N^s\in \Perf_k^{w[-n, 0]}$ has a smaller number of generators in degree $-n$.
    Applying the functors $\LSym^T_k$ and $\R\LSym^T_k$, we obtain a natural transformation of augmented cosimplicial diagrams
        $$\begin{tikzcd}[column sep=1pc]
    \LSym^T_k(M)\arrow[r]\arrow[d, "\mu"{swap}] & \LSym^T_k(N^0)\arrow[r, yshift=0.5ex]\arrow[r, yshift=-0.5ex]\arrow[d, "\mu"{swap}] & \LSym^T_k(N^1)\arrow[d, "\mu"]\arrow[r, yshift=1ex]\arrow[r, yshift=-1ex]\arrow[r] & \dots\\
    \R\LSym^T_k(M)\arrow[r] & \R\LSym^T_k(N^0)\arrow[r, yshift=0.5ex]\arrow[r, yshift=-0.5ex] & \R\LSym^T_k(N^1)\arrow[r, yshift=1ex]\arrow[r, yshift=-1ex]\arrow[r] & \dots
    \end{tikzcd}$$
     The bottom row is a limit diagram by Lemma~\ref{lem:universal} and 
     the top row is a limit by condition (c).
     Since the maps $\mu\colon \LSym^T_k(N^s)\rightarrow \R\LSym^T_k(N^s)$ are equivalences
    for all $s\geq 0$ by inductive hypothesis, the result follows.
\end{proof}

\begin{lemma}\label{lem:ccff condition}
	Let $T\colon \DAlg_k\rightarrow \DAlg_k$ be a left adjoint functor such that $T(k[x])$ is flat.
	If the map 
	$$
	\LSym^T(k[-n])\rightarrow \LSym^T(0)=k
	$$
	is coconnectively faithfully flat
	for each $n>0$, then $T$ satisfies condition \emph{(iii)} from
	Definition~\ref{def:derived_plethory}. The converse holds if $k$ is a ring such 
	that every finitely presented $k$-module has finite flat dimension.
\end{lemma}

\begin{proof}
	For each cofiber sequence $M\rightarrow N\rightarrow k[-n]$, the map
	$\LSym^T(M)\rightarrow \LSym^T(N)$ is the base change of $\LSym^T(k[-n-1])\rightarrow k$.
	It is therefore ccff, and $\LSym^T(M)$ is indeed the totalization of the \v{C}ech complex.
	For the converse, it suffices to verify that the fiber of $\LSym^T(k[-n])\rightarrow k$
	is a $k$-module of Tor-amplitude $\leq -1$ by \cite[Prop.~2.10]{mathew-mondal}.
	Pick a cosimplicial diagram $N^\bullet\colon \Delta\rightarrow \Proj_k^\omega$
	such that $\Tot(N^\bullet)\we k[-n]$ and $N^s\we 0$ for all $s<n$; one can for example
	take the reduced cochains on a suitable simplicial model for the $n$-sphere.
	It follows from Lemma~\ref{lem:cech descent vs right extended} that
	$$
	\fib(\LSym^T(k[-n])\rightarrow k)\simeq \Tot\big(\fib(\LSym^T(N^\bullet)\rightarrow k\big)
	$$
	is the totalization of a cosimplicial diagram of flat $k$-modules, which is zero
	in cosimplicial degrees $<n$. 
	The fiber then has tor-amplitude $\leq -1$ by \cite[Rem.~3.17]{mathew-mondal}.
\end{proof}

\begin{theorem}\label{thm:deriving-plethories}
	Let $k$ be a commutative ring.
    Consider the full subcategory 
    $\Pleth^{\flat}_k\subseteq \Alg(\End(\DAlg_k))$ 
    of derived flat plethories and the full subcategory
    $\Pleth^{\heart, \flat}_k\subseteq \Alg(\End(\CAlg_k^\heart))$
    of flat plethories. Then there is a natural fully faithful functor
    $$
    \Pleth^{\flat}_k\longrightarrow\Pleth_k^{\heart,\flat}; \quad T\longmapsto \pi_0T\big|_{\CAlg_k^{\heart}}.
    $$
    Its essential image consists of the flat plethories $T'$ such that
    $\R\Sym^{T'}_k\colon \Perf_k^{w\leq 0}\rightarrow \DAlg_k$ has the following property:
    for each fiber sequence $M'\rightarrow M\rightarrow M''$
    in $\Perf_k^{w\leq 0}$, the map 
    $$
    R\Sym^{T'}_k(M)\otimes_{\R\Sym^{T'}_k(M')} k \rightarrow \R\Sym^{T'}_k(M'')
    $$
    is an equivalence.
\end{theorem}

\begin{proof}
	We have a commuting diagram of fully faithful functors
	$$\begin{tikzcd}[column sep=5pc]
	\Pleth_k^{\flat}\arrow[r, dotted]\arrow[d, hookrightarrow] & \Pleth_k^{\heart, \flat}\arrow[d, hookrightarrow]\\
	\Alg(\End'_\Sigma(\DAlg_k))\arrow[r, hookrightarrow, "{T\mapsto \pi_0T|_{\CAlg_k^{\heart}}}"] & \Alg(\End_{\Sigma}(\CAlg_k^{\heart})).
	\end{tikzcd}$$
	For the left vertical inclusion, this uses that every derived plethory is contained 
	in $\End'_\Sigma(\DAlg_k)$ by Lemma~\ref{lem:cech descent vs right extended}.
	The bottom functor is the fully faithful functor induced by Corollary~\ref{cor:dalg-monads}.
	To see that this functor preserves plethories, let $T\colon \DAlg_k\rightarrow \DAlg_k$
	be a derived flat plethory. By construction, the two functors 
	$$
	\pi_0T\circ \Sym\colon \Proj_k^{\omega}\to \CAlg_k^{\heart, \flat}
	\qquad\qquad T\circ \Sym\colon \Proj_k^{\omega}\to \CAlg_k^{\heart, \flat}\subseteq \DAlg_k
	$$
	are naturally equivalent. As $T$ was a left adjoint, these equivalent functors 
	preserve finite coproducts. The monad 
	$\pi_0T\colon \CAlg_k^{\heart}\rightarrow \CAlg_k^{\heart}$
	therefore preserves finite coproducts 
	of polynomial algebras. Since $\pi_0T$ preserves sifted colimits
	and every $k$-algebra is a sifted colimit of 
	polynomial algebras,
	it follows that $\pi_0T$ preserves finite coproducts. This means that it preserves
	all colimits, so it is a left adjoint and hence a plethory.
	In addition, note that the two functors
	$$
	\R\Sym^{\pi_0T}_k\colon \Perf_k^{w\leq 0}\rightarrow \DAlg_k\qquad\qquad \LSym^T_k\colon \Perf_k^{w\leq 0}\to \DAlg_k
	$$
	are naturally equivalent as well, as both are right Kan extended from $\Proj_k^{\omega}$.
	The fact that $\R\Sym^{\pi_0T}_k$ sends fiber sequences to pushouts 
	of derived algebras	then follows from the fact that 
	$\LSym^T_k\colon \Mod_k\to \DAlg_k$	is a left adjoint.
	
	Conversely, suppose that $T'\colon \CAlg_k^{\heart}\rightarrow \CAlg_k^{\heart}$
	is a flat plethory such that $\R\Sym_k^{T'}$ sends fiber sequences to pushouts
	of derived algebras. This implies that $\R\Sym_k^{T'}$ preserves all pushouts,
	and hence all finite colimits, in $\Perf_k^{w\leq 0}$ that are preserved by
	the inclusion into $\Perf_k$. 
	In particular, $\R\Sym_k^{T'}$ preserves skeletal geometric realizations 
	that are preserved by the inclusion into $\Perf_k$.
	This means that $\Sym_k^{T'}\colon \Proj_k^\omega\rightarrow \DAlg_k$ is right extendable.
	Since the forgetful functor $\DAlg_k\rightarrow \Mod_k$ detects limits and
	sifted colimits, the functor 
	$\Sym_k^{T'}\colon \Proj_k^\omega\rightarrow \Mod_k$ is right extendable as well.
	By Corollary~\ref{cor:dalg-monads}, there then exists a unique monad 
	$T\in \Alg(\End'_\Sigma(\DAlg_k))$ such that $T'=\pi_0T$.
	To see that $T$ is a derived flat plethory, it remains to verify that it preserves
	colimits. Since $T$ preserves sifted colimits and every derived algebra is
	a sifted colimit of free derived algebras $\LSym_k(M)$ with $M\in \Perf_k^{w\leq 0}$,
	it suffices to show that $\LSym^T\colon \Perf_k^{w\leq 0}\rightarrow \DAlg$ 
	preserves finite coproducts. But this functor is equivalent to $\R\Sym^{T'}_k$, 
	which preserved all finite colimits preserved by the inclusion into $\Perf_k$.
\end{proof}

Let us add a useful criterion for a classical flat plethory $T$ to extend to a 
derived flat plethory. To this end, let us say that a commutative affine flat group
scheme $\CAlg_k^{\heart}\rightarrow \Ab$ is strongly unipotent if for each $n\geq 0$, 
the fpqc stack $\B^n H\colon \CAlg_k^{\heartsuit}\rightarrow \Spc$
is affine in the sense of Toën~\cite{toen-affines}.
If $k$ is a ring such that every finitely presented $k$-module has finite flat dimension, 
this is equivalent to $H$ being unipotent, in the sense that $\B H$ is affine, 
by \cite[Prop.~4.14]{mathew-mondal}.	

\begin{proposition}\label{prop:image restriction plethories}
	Let $k$ be a ring and $T\in \Pleth_k^{\heart, \flat}$ be a flat plethory. 
	If $R_T\colon \CAlg_k^{\heart}\rightarrow \Ab$ is strongly unipotent, 
	then $T$ lies in the essential image of the fully faithful inclusion of 
	Theorem \ref{thm:deriving-plethories}.
	The converse holds if every finitely presented $k$-module has finite flat dimension.
\end{proposition}

\begin{proof}
	Note that $R_T$ comes with a natural $k$-linear structure. Using this, we obtain
	a functor $\Perf_k^{\cn}\rightarrow \mathrm{St}_k$ sending each $M\in \Perf_k^{\cn}$
	to the fpqc-stack $R_T\otimes_k M\colon\CAlg_k^{\heart}\rightarrow \Spc$.
	This sends $k[n]$ to $\B^n R_T$ and each cofiber sequence in $\Perf_k^{\cn}$ to a 
	fiber sequence of stacks. Since each $M\in \Perf_k^{\cn}$ can be 
	obtained from the suspensions $k[n]$ using finite limits and retracts,
	it follows that each $R_T\otimes_k M$ is an affine stack. Furthermore, we have 
	a natural equivalence of derived $k$-algebras
	$$
	\Gamma(R_T\otimes_k M, \Oscr)\we \R\Sym^T_k(M^\vee).
	$$
	Now let $M''\rightarrow M\rightarrow M'$ be a cofiber sequence in $\Perf_k^{\cn}$.
	In particular, $\pi_0M\rightarrow \pi_0M'$ is surjective, so that
	$R_T\otimes_k M\rightarrow R_T\otimes_k M'$ is an effective epimorphism for the
	fpqc topology. The map $\R\Sym^T_k(M'^\vee)\rightarrow \R\Sym^T_k(M^\vee)$
	is then coconnectively faithfully flat, so that
	$$
	\R\Sym^T_k(M^\vee)\otimes_{\R\Sym^T_k(M'^\vee)} k \rightarrow \R\Sym^T_k(M''^\vee)
	$$
	is a map of coconnective derived rings. This map is an equivalence by \cite[Cor.~2.2.3]{toen-affines}
	and \cite[Cor.~3.7]{mathew-mondal}. We conclude that $T$ lies in the essential image 
	of the embedding from Theorem~\ref{thm:deriving-plethories}.
	
	For the converse, let $k$ be a ring such that every finitely presented $k$-module
	has finite flat dimension and suppose that $T$ extends to a derived flat plethory.
	Lemma~\ref{lem:ccff condition} then implies that $\R\Sym^T_k(k[-n])\rightarrow k$ is ccff
	for each $n\geq 1$. It follows by induction that $\Spec(\R\Sym^T_k(k[-n]))\we \B^n R_T$,
	so $R_T$ is strongly unipotent.
\end{proof}

\begin{definition}
    Let $T\colon \DAlg_k\rightarrow \DAlg_k$ be a derived flat plethory. 
    We will write $\DAlg_k^T$ for the category of $T$-algebras.
    The forgetful functor is part of both a monadic adjunction and a comonadic adjunction
    $$\begin{tikzcd}
        \mathrm{Free}_T\colon \DAlg_k\arrow[r, yshift=1ex]& \DAlg_k^T\colon \mathrm{oblv}\arrow[l, yshift=-1ex] & \mathrm{oblv}\colon \DAlg^T_k\arrow[r, yshift=1ex]& \DAlg_k\colon \mathrm{coFree}_{R_T},\arrow[l, yshift=-1ex]
    \end{tikzcd}$$
\end{definition}

\begin{lemma}\label{lem: induced monads}
	Let $T$ be a derived flat plethory and consider the following pullback squares in $\Cat_\infty$.
	$$\begin{tikzcd}
	\CAlg_k^{T,\heartsuit} \ar[r, hook] \ar[d, "U^{\heartsuit}"{swap}] &
	\DAlg_k^{T, \cn}\ar[d, "U^\cn"]\ar[r, hook] & \DAlg_k^T\ar[d, "U"]\\
			\CAlg_k^\heartsuit \ar[r,hook] & \DAlg_k^{\cn}\ar[r, hook] & \DAlg_k.
	\end{tikzcd}$$
	\begin{enumerate}[{\em (1)}]
	\item The right square is a pullback in $\PrL$ and is vertically left adjointable.
	\item The left square is a pullback in $\PrR$, the functor $U^\heartsuit$ is monadic and the associated monad is
	isomorphic to the flat plethory $\pi_0T\big|_{\CAlg_k^{\heart}}$
	from Theorem \ref{thm:deriving-plethories}.
	\end{enumerate}
\end{lemma}

\begin{proof}
	Assertion (1) follows from the fact that the derived plethory $T$ preserves
	animated rings by Remark \ref{rem:flat-plethory}: 
	the left adjoint to $U$ then restricts to a left adjoint $\DAlg_k^{\cn}\to \DAlg_k^{T, \cn}$ to $U^{\cn}$.
	For (2), recall that $\CAlg_k^\heart\subseteq \DAlg_k^{\cn}$ is the
	full subcategory of $0$-truncated objects. 
	Because $U^{\cn}$ is a conservative right adjoint,
	$\CAlg_k^{T, \heart}\subseteq \DAlg_k^{T, \cn}$ also coincides with
	the full subcategory of $0$-truncated objects. 
	It follows that $U^{\heart}$ is the restriction of the right adjoint $U^{\cn}$
	to $0$-truncated objects. Note that $U^{\cn}$ has a further right adjoint $R^\cn$, 
	such that the comonad $U^{\cn}R^{\cn}$ is the right adjoint of $T$, restricted
	to animated rings. This right adjoint $R^{\cn}$ preserves $0$-truncated 
	objects as well, so that $U^{\heart}$ is both monadic and comonadic.
	The corresponding comonad is the restriction of the comonad $U^\cn\circ R^\cn$
	to $0$-truncated objects. 
	Taking left adjoints, one sees that the corresponding monad is isomorphic to $\pi_0T\big|_{\CAlg_k^{\heart}}$.
\end{proof}

\begin{corollary}\label{cor:center of DAlg_T is discrete} Let $T$ be a derived plethory. Then restriction induces an equivalence
	$$\End(\id_{\DAlg_{k}^T})\simeq \End(\id_{\CAlg_k^{T,\heartsuit}})$$ 
	of $\mathbb E_2$-spaces. In particular, $\End(\id_{\DAlg_{k}^T})$ is a discrete commutative monoid.  
	
\end{corollary}	
\begin{proof} We have a chain of equivalences
	$$
	\End(\id_{\DAlg_{k}^T})\simeq \End_{\BMod_{\LSym_k^T}}(\LSym_k^T)\simeq \End_{\BMod_{\Sym_k^{T,\heartsuit}}}(\Sym_k^{T,\heartsuit}) \simeq  \End(\id_{\CAlg_k^{T,\heartsuit}}).
	$$
	Here the first and the last equivalences are given by Corollary \ref{cor:bimodules_monoidal}, and the middle one by Corollary \ref{cor:dalg-monads}.
\end{proof}	

\section{The simplicial cosimplicial comparison theorem}\label{sec:cosimplicial}

Let $T$ be a derived flat plethory and let us write $\CAlg_k^{T, \heart}$ for the $1$-category of $T$-algebras
whose underlying $k$-module is discrete. The forgetful functor $\CAlg_k^{T, \heart}\rightarrow \CAlg_k^{\heart}$
admits a left adjoint $\pi_0 T$. This exhibits $\CAlg_k^{T, \heart}$ as the category of algebras
over the flat plethory $\pi_0 T\colon \CAlg_k^{\heart}\rightarrow \CAlg_k^{\heart}$ by Lemma~\ref{lem: induced monads}.

\begin{notation}\label{not:simplicial cosimplicial}
Let $\Dscr$ be an ordinary category with small limits and let us write $\mathrm{sc}\Dscr$ for the category of
simplicial cosimplicial objects in $\Dscr$. Restriction and right Kan extension then determine an equivalence
$\mathrm{sc}\Dscr\simeq \Fun^{\R}(\mathrm{sSet}^{\op}, \mathrm{c}\Dscr)$. Given a simplicial cosimplicial object
$X^\bullet_\bullet$ and a simplicial set $K$, let us write $X_K^\bullet$ for the value at $K$ of the right
adjoint functor associated to $X^\bullet_\bullet$. For instance, when $K=\partial\Delta^n$, this reproduces the
usual $n$-matching object.

The category $\mathrm{sc}\Dscr$ is cotensored over $\mathrm{sSet}$ in two ways, using either the simplicial or
cosimplicial direction:
$$
\big({}^K X\big)^t_s = X^t_{K\times \Delta^s} \qquad\qquad\qquad \big(X^K\big)^t_s = \prod_{K_t} X^t_s.
$$
Taking $K=\Delta^1$, this gives a simplicial and a cosimplicial version of the path space. Consequently, we have
both a simplicial and a cosimplicial notion of homotopy in $\mathrm{sc}\Dscr$.
\end{notation}
\begin{construction}\label{con:simplicial cosimplicial modules}
Let $k$ be a discrete ring and let $\mathrm{scMod}_k^{\heart}$ denote the category of simplicial cosimplicial
(discrete) $k$-modules. There is a natural functor
$$\begin{tikzcd}
\Tot^{\oplus}\colon \mathrm{scMod}_k^{\heart}\arrow[r] & \Mod_k; \qquad \Tot^{\oplus}(M^\bullet_\bullet)=
|\Tot(M^\bullet_\bullet)|=\colim_m \lim_n M^n_m.
\end{tikzcd}$$
We will say that a map of simplicial cosimplicial $k$-modules $f\colon N^\bullet_\bullet\rightarrow
M^\bullet_\bullet$ is:
\begin{enumerate}
    \item[(1)] a \emph{weak equivalence} if $\Tot^{\oplus}(N^\bullet_\bullet)\rightarrow \Tot^{\oplus}(M^\bullet_\bullet)$
is an equivalence;

\item[(2)] a \emph{fibration} if the maps $N^\bullet_n\to
N^\bullet_{\partial\Delta^n}\times_{M^\bullet_{\partial\Delta^n}} M^\bullet_n$ are surjective in each
cosimplicial degree and if for each horn inclusion $\Lambda^n_t\hookrightarrow \Delta^n$, the map of cosimplicial
$k$-modules
$$\begin{tikzcd}
N^\bullet_n\arrow[r] & N^\bullet_{\Lambda^n_t}\times_{M^\bullet_{\Lambda^n_t}} M^\bullet_n
\end{tikzcd}$$
is a degreewise surjection such that the totalization of its fiber is $0$-connective;

\item[(3)] a \emph{trivial fibration} if each map of cosimplicial $k$-modules
$$\begin{tikzcd}
N^\bullet_n\arrow[r] & N^\bullet_{\partial\Delta^n}\times_{M^\bullet_{\partial\Delta^n}} M^\bullet_n
\end{tikzcd}$$
is a degreewise surjection such that the totalization of its fiber is $0$-connective.
\end{enumerate}
This defines a model structure on $\mathrm{scMod}_k$, and the functor $\Tot^\oplus\colon
\mathrm{scMod}_k\rightarrow \Mod_k$ exhibits $\Mod_k$ as its localization \cite[Thm.~5.5]{brantner-campos-nuiten}. Furthermore, the model structure is simplicial
with respect to the simplicial cotensoring ${}^KM$ coming from the simplicial direction.
\end{construction}
\begin{theorem}\label{thm:cosimplicial}
Let $T\colon \DAlg_k\rightarrow \DAlg_k$ be a derived flat plethory,
and consider the category $\mathrm{scCAlg}_k^{T,
\heart}=\Fun(\Delta^{\op}\times \Delta, \DAlg_k^{T, \heart})$ of simplicial cosimplicial $T$-algebras.
This admits a
simplicial model structure in which a map is a weak equivalence (fibration) if the underlying map of simplicial
cosimplicial $k$-modules is a weak equivalence (fibration) in the sense of Construction \ref{con:simplicial
cosimplicial modules}. Furthermore, the canonical functor
$$\begin{tikzcd}
\Tot^\oplus\colon \mathrm{scCAlg}_k^{T, \heart}\arrow[r] & \DAlg^T_k; \quad A^\bullet_\bullet\arrow[r, mapsto] &
{|\Tot(A^\bullet_\bullet)|}
\end{tikzcd}$$
exhibits $\DAlg^T_k$ as the localization of $\mathrm{scCAlg}_k^{T, \heart}$ at its weak equivalences.
\end{theorem}

\begin{proof}
We will transfer the model structure on simplicial cosimplicial $k$-modules along the free-forgetful adjunction
$\Sym^{\pi_0 T}_k\colon \mathrm{scMod}_k^\heart\leftrightarrows \mathrm{scCAlg}^{T, \heartsuit}\colon
\mathrm{oblv}$.
The resulting model structure will then be simplicial for the simplicial cotensoring $A{}^K$, which is preserved
by the forgetful functor.

Since the forgetful functor preserves filtered colimits, it will suffice to verify the following condition for
the existence of the transferred model structure: given a generating trivial cofibration $i\colon
V^\bullet_\bullet\rightarrow W^\bullet_\bullet$ in $\mathrm{scMod}_k^\heart$ and a pushout of the form
$$\begin{tikzcd}
\Sym^{T^\heart}_k(V^{\bullet}_\bullet)\arrow[r] \arrow[d] & \Sym^{T^\heart}_k(W^\bullet_\bullet)\arrow[d]\\
A\arrow[r] & B,
\end{tikzcd}$$
the map $\Tot^{\oplus}(A)\rightarrow \Tot^\oplus(B)$ is an equivalence in $\Mod_k$. To see this, it suffices to
note that any generating trivial cofibration $i\colon V^\bullet_\bullet\rightarrow W^\bullet_\bullet$ in
$\mathrm{scMod}^{\heart}_k$ (cf.\ \cite[Rem.~5.6]{brantner-campos-nuiten}) is the inclusion of a strong deformation retract for either the
simplicial or the cosimplicial notion of homotopy from Notation \ref{not:simplicial cosimplicial}. Such strong
deformation retracts are preserved by the functor $\Sym^{\pi_0 T}_k$ and by pushouts, so that $A\rightarrow B$
is a strong deformation retract as well. This implies that $A\rightarrow B$ is a weak equivalence, because
$\Tot^\oplus$ sends both simplicial and cosimplicial homotopy equivalences to equivalences in $\Mod_k$.

Finally, let us consider the commutative diagram
$$\begin{tikzcd}
{\mathrm{scCAlg}_k^{T, \heart}[W^{-1}]}\arrow[r, "\Tot^\oplus"]\arrow[d, "\mathrm{oblv}"] & \DAlg^T_k\arrow[d,
"\mathrm{oblv}"]\\
{\mathrm{scMod}_k^{\heart}[W^{-1}]}\arrow[r, "\Tot^\oplus"] & \Mod_k
\end{tikzcd}$$
where the bottom functor is an equivalence by~\cite[Thm.~5.5]{brantner-campos-nuiten}.
The vertical functors are conservative right adjoints preserving sifted colimits, so in particular monadic. For the left
vertical functor, this uses that every filtered colimit in $\mathrm{scCAlg}_k^{T, \heart}$ is a homotopy colimit,
since weak equivalences are stable under filtered colimits, and that geometric realizations of simplicial objects
can be computed using the diagonal; this implies that the right Quillen functor $\mathrm{scCAlg}_k^{T,
\heart}\rightarrow \mathrm{scMod}_k^{T, \heart}$ preserves filtered homotopy colimits and geometric realizations.

Since the vertical functors are monadic, it then suffices to show that for any cofibrant $M^\bullet_\bullet$, the
natural comparison map
$$\begin{tikzcd}
\mu\colon \LSym^T_k\big(\Tot^{\oplus}(M^\bullet_\bullet)\big)\arrow[r] &
\Tot^\oplus\big(\Sym_k^{\pi_0 T}(M^\bullet_\bullet)\big)
\end{tikzcd}$$
is an equivalence. Note that by construction, $\mu$ only depends on the image of $M^\bullet_\bullet$ in
$\mathrm{scMod}_k^{\heart}[W^{-1}]\simeq \Mod_k$, i.e.\ on $\Tot^\oplus(M^\bullet_\bullet)$.
Since the domain and codomain of $\mu$ preserve sifted colimits, we can therefore reduce to the case where
$\Tot(M^\bullet_\bullet)\simeq V[-n]$, where $V$ is a finite type free $k$-module and $n\geq 0$.

An explicit (cofibrant) simplicial cosimplicial representative of $V[-n]$ is given by the $n$-coskeletal
cosimplicial $k$-module $M^\bullet=\coker\big(V\to V^{\partial\Delta^{n+1}}\big)^\bullet$. On such cosimplicial $k$-modules,
constant in the simplicial direction, $\Tot^\oplus$ is simply given by the usual totalization. Since $\LSym^T_k$
preserves totalizations of $n$-coskeletal cosimplicial diagrams of free $k$-modules, it then suffices to verify
that $\mu$ is an equivalence on each individual term $M^t$. On each of these finite type free modules $M^t$,
$\mu$ can be identified with the map $\LSym^T_k(M^t)\rightarrow \Sym^{\pi_0T}_k(M^t)$.
This is an equivalence because $T$ and $\pi_0T$ agree on flat $k$-modules.
\end{proof}

\section{Comparison results}\label{sec:comparison}

In this section, we assemble the previous material together to analyze various higher categories of
derived commutative rings equipped with extra structures.

\subsection{Endomorphisms}\label{sec:endomorphisms}

Let $k$ be a commutative ring and consider the $1$-category $\CAlg_k^{\End,\heartsuit}$ of commutative $k$-algebras
equipped with an endomorphism. 
\begin{lemma}\label{lem:endo plethory}
The forgetful functor $\CAlg_k^{\End,\heartsuit}\to \CAlg_k^{\heartsuit}$ is monadic and the corresponding monad $T^{\heartsuit}_{\End}$ is a flat plethory.
\end{lemma}
\begin{proof}
The forgetful functor can be identified with the functor $\Fun(\B\bN, \CAlg_k^{\heart})\to \CAlg_k^{\heart}$ restricting along $\ast\to \B\bN$. This is conservative and (co)monadic, because Kan extension provides a left and a right adjoint. 
The associated monad $T_{\End}^{\heartsuit}$ preserves colimits and $T_{\End}^{\heartsuit}(k[x]) = k[x, f(x), f^2(x), \dots]$ is flat over $k$, so $T_{\End}^{\heartsuit}$ is a flat plethory. 
\end{proof}
The group scheme underlying the plethory $T_{\End}^{\heartsuit}$ can be identified with $\prod \Ga\simeq \Spec(k[x, f(x), \dots])$. Using this, one can see that $T_{\End}^{\heartsuit}$ is contained in the essential image of the functor from Theorem~\ref{thm:deriving-plethories} and hence admits a derived monad $ T_{\End}$. Then, there are at least three possibilities for the definition
of a derived commutative $k$-algebra with an endomorphism:
\begin{equation*}
    \Mod_{
    T^{\End}}(\DAlg_k),\quad\quad\Fun(\B\bN,\DAlg_k),\quad\quad(\s\c\CAlg_k^\End)[W^{-1}].
\end{equation*}
Note that by Lemma \ref{lem:functors_to_mod}, the category  $\Fun(\B\bN,\DAlg_k)$ is equivalent to the category of modules over the monad on $\Fun(\B\bN,\Mod_k)$ defined by post-composition with $\LSym_k$.

\begin{remark}
    There is yet another way we could attempt to construct a notion of a derived commutative ring
    with an endomorphism. We could look at the free commutative ring monad on $\Fun(\B\bN,\Mod_k^\heartsuit)$
    and attempt to derive it. A priori, this is a different construction from deriving $T_{\End}^{\heartsuit}$. However, the natural $t$-structure and
    symmetric monoidal structure on $\Fun(\B\bN,\Mod_k)$ does not form a derived algebraic context
    in the sense of~\cite{raksit}. 
\end{remark}

\begin{proposition}\label{lem:derived_endomorphism}
    There are natural equivalences  $(\s\c\CAlg_k^\End)[W^{-1}] \simeq \Mod_{
    	T^{\End}}(\DAlg_k)\simeq \Fun(\B\bN,\DAlg_k)$.
\end{proposition}

\begin{proof} The first equivalence follows from Theorem \ref{thm:cosimplicial}. The forgetful functor $\Fun(\B\bN,\DAlg_k)\rightarrow\DAlg_k$ preserves all limits and colimits and
	is additionally conservative; consequently, it is monadic with associated monad $T$. This takes any commutative derived $k$-algebra $R$ to the countable coproduct $R^{\otimes \bN}$. In particular, $T(k[x])$ is flat. By Lemma \ref{lem:ccff condition}, it remains to show that the map $T(\LSym_k(k[-n]))\to k$ is ccff for every $n\ge 1$, but this is equivalent to a countable coproduct of ccff maps. Finally, the category of discrete $T$-modules is naturally identified with  $\Fun(\B\bN,\CAlg_k^{\heartsuit})$. This implies the second equivalence.

\end{proof}

\subsection{Derived $\delta$-rings}\label{sec:delta}

Fix a prime $p$. Recall from~\cite[App.~A]{bhatt-lurie-prismatization} that a $p$-typical animated $\delta$-ring is given by an animated ring $A \in \Ani_{\bZ}$ together with
an endomorphism $\phi$ and a homotopy $h \colon \phi \otimes_{\bZ} \bF_p \simeq \Fr_p$ in $\End_{\Ani_{\bF_p}}(A \otimes_\bZ \bF_p)$.\\
The category of animated $\delta$-rings is equivalent to the animation $\Pscr_\Sigma(\Poly^\delta)$ of the $1$-category of discrete $\delta$-rings as was shown in loc.~cit. In~\cite{holeman-derived}, following~\cite{raksit}, the animated $\delta$-ring monad was extended to a monad $\LSym^\delta_\bZ$ on $\Mod_\bZ$ and the category of derived $\delta$-rings defined as $\DAlg_\bZ^\delta = \LMod_{\LSym^\delta_\bZ}(\Mod_\bZ)$. By construction, this comes with an equivalence $\DAlg_\bZ^{\delta,\cn}\simeq \Pscr_\Sigma(\Poly^\delta)$ and it is natural to ask whether the description of the connective part in terms of Frobenius lifts extends to the whole of $\DAlg_\bZ^\delta$. In the present section, we show that it does.

\begin{definition}\label{def:Frobenius}
	We write $\Fr_p$ for the natural endomorphism of $\id_{\DAlg_{\bF_p}}$ defined in Example~\ref{ex:frobenius}.
\end{definition}

\begin{definition}[Derived $\varphi$-algebras]\label{def:derived phi-algebras}
    We define the category of \textit{derived $p$-typical $\varphi$-algebras} as the pullback
    $$\begin{tikzcd}
    	\DAlg^\varphi_{\bZ}\ar[r]\ar[d]&\Fun(\B\bN,\DAlg_\bZ)\ar[d, "(-)\otimes_\bZ\bF_p"]\\[3ex]
        \DAlg_{\bF_p}\ar[r, "\Fr_p"] &\Fun(\B\bN,\DAlg_{\bF_p})
    \end{tikzcd}$$
    in $\iCat$. Here, the bottom map classifies the Frobenius endomorphism and the right map is given by
    reduction modulo $p$.
\end{definition}

The pullback in Definition~\ref{def:derived phi-algebras} is both a pullback in $\PrL$ and $\PrR$. It follows that the forgetful functor
$$U^\varphi \colon \DAlg^\varphi_{\bZ}\rightarrow\Fun(\B\bN,\DAlg_\bZ)\rightarrow\DAlg_\bZ$$
is conservative and preserves small limits and colimits. In particular, it exhibits the category of derived $\varphi$-algebras as both monadic and comonadic over $\DAlg_\bZ$. We denote
the associated colimit preserving monad by $T^\varphi$ and its adjoint limit preserving comonad by $C^\varphi$.

\begin{remark}[Animated $\delta$-rings]\label{rmk:animated delta}
	We write $\Ani_\bZ^\delta$ for the category of animated $\delta$-rings defined by the pullback from Definition~\ref{def:derived phi-algebras} with
	$\Ani_\bZ$ and $\Ani_{\bF_p}$ in place of derived rings and derived $\bF_p$-algebras. Then, 
	by~\cite[Rem.~A.22]{bhatt-lurie-prismatization}, there exists an equivalence $\Ani_\bZ^\delta \simeq \Pscr_\Sigma(\Poly^\delta_\bZ)$ of $\infty$-categories which is compatible with the
	forgetful functors to animated rings. In particular, it follows that the forgetful functor $\Ani_{\bZ}^\delta \to \Ani_{\bZ}$ is monadic and that the associated monad $T^\delta$
	on animated rings preserves all colimits.
\end{remark}

%

\begin{lemma}\label{lem:restriction to animated delta}
	The derived $\varphi$-algebra monad $T^\varphi$ restricts to the animated $\delta$-ring monad $T^\delta$ on $\Ani_\bZ$.
\end{lemma}

\begin{proof}
	Let $S \in \DAlg_\bZ^\varphi$ be any derived
	$\varphi$-algebra. We construct a natural animated $\delta$-structure on $\tau_{\geq 0}S$ together with a natural counit map $\tau_{\geq 0}S\rightarrow S$ of derived
    $\varphi$-algebras as follows: the Frobenius endomorphism on $S$ restricts by naturality to an
    endomorphism of $\tau_{\geq 0}S$. The mod $p$ reduction $(\tau_{\geq 0}S)/p$ thus obtains an
    endomorphism which we must argue is equivalent to the Frobenius endomorphism. However, it fits
    into a pullback square
    $$\xymatrix{
        (\tau_{\geq 0}S)/p\ar[r]\ar[d]&\tau_{\geq 0}(S/p)\ar[d]\\
        \pi_0((\tau_{\geq 0}S)/p)\ar[r]&\pi_0(S/p)
    }$$
    of derived commutative $\bF_p$-algebras equipped with endomorphisms. Here, the bottom map is injective with cokernel given
    by $(\pi_{-1}S)[p]$. Since the functor $\DAlg_{\bF_p}\rightarrow\Fun(\B\bN,\DAlg_{\bF_p})$
    classifying Frobenius commutes with all limits and colimits,
    this allows us to identify the endomorphism on $(\tau_{\geq 0}S)/p$ with
    Frobenius and hence provide $\tau_{\geq 0}S$ with the structure of a derived $\varphi$-algebra.
    There is an induced map $\tau_{\geq 0}S\rightarrow S$ of derived $\varphi$-algebras.
    We claim that $\tau_{\geq 0}S$ gives a right adjoint to the inclusion
    $\Ani^\delta_{\bZ}\subseteq\DAlg^\varphi_\bZ$. For this it is enough to note that if $R$ is a
    connective derived $\varphi$-algebra and $S$ is arbitrary, then the map $R/p\rightarrow S/p$
    factors uniquely through $(\tau_{\geq 0}S)/p$, which follows from the pullback square above.
    Now, if $R$ is a connective derived commutative ring and $S\in\DAlg^\varphi$, there is a
    natural equivalence $$\Map(T^\varphi(R),\tau_{\geq 0}S)\we\Map(T^\varphi(R),S).$$
    Applying this equivalence to $S=T^\varphi(R)$ itself, we find that the identity on $T^\varphi(R)$ factors
    through $\tau_{\geq 0}T^\varphi(R)\rightarrow T^\varphi(R)$, which implies that $T^\varphi(R)$
    is connective. This completes the proof.
\end{proof}

In~\cite{bhatt-lurie-prismatization}, the authors define and study animated $\delta$-rings via $p$-typical Witt vectors. In the following, we consider the analogue of this
perspective over $\DAlg_\bZ$ to exhibit filtrations on the derived $\varphi$-algebra monad $T^\varphi$ and the comonad $C^\varphi$.

\begin{definition}\label{def:derived W2}
	We denote by $W_2$ the endofunctor on $\DAlg_{\bZ}$ which is defined by the pullback
	$$\begin{tikzcd}
		W_2(R) \ar[r] \ar[d, "w_0"] \ar[dr, phantom, "\lrcorner", near start] & R \ar[d, "\can"] \\[3.5ex]
		R \ar[r, "\Fr"] & R \otimes_{\bZ} \bF_p
	\end{tikzcd}$$
	in $\DAlg_{\bZ}$ for any derived ring $R$. This comes equipped with a natural co-pointing given by $w_0\colon W_2(R) \to R$.
\end{definition}

\begin{remark}\label{rmk:properties W2}
	The functor $W_2$ is a right adjoint endofunctor on $\DAlg_\bZ$ which preserves sifted colimits. This follows from the following observations:
	The functor $(-)\otimes_\bZ\bF_p \in \End(\Mod_\bZ)$ preserves both limits and colimits since $\bF_p$ is dualisable in $\Mod_\bZ$. Sifted colimits commute with finite limits
	in $\DAlg_\bZ$ since this is true in $\Mod_\bZ$ and the $\LSym_\bZ$-monad preserves sifted colimits.
	It follows that the pullback in Definition~\ref{def:derived W2} is a pullback in $\End_\Sigma^R(\DAlg_\bZ)$. 
\end{remark}

\begin{definition}
	For any $n\geq 1$, we denote by $W_n^{\ani}$ the functor of animated $n$-truncated $p$-typical Witt vectors. It follows from~\cite[Prop.~A.3]{bhatt-lurie-prismatization}
	that the $W_n^{\ani}$ preserve sifted colimits and all limits as endofunctors on $\Ani_\bZ$.
\end{definition}

\begin{lemma}\label{lem:W2 and animated W2}
	For any animated ring $A$, there exists a natural equivalence $W_2(\iota A)\simeq \iota W_2^{\ani}(A)$ in $\DAlg_\bZ$. For any derived ring $R$, the canonical map
	$W_2^{\ani}(\tau_{\geq 0}R) \to \tau_{\geq 0}W_2(R)$ in $\Ani_\bZ$ is an equivalence.
\end{lemma}

\begin{proof}
	For the first part of the lemma, note that there exists a fibre sequence $M\xrightarrow{\cdot p} M\xrightarrow{\can} M\otimes_\bZ \bF_p$ in $\Mod_{\bZ}$ for any $M \in \Mod_\bZ$.
	This implies that for $A \in \Ani_\bZ$, the pullback in Definition~\ref{def:derived W2} defining $W_2(A)$ is a pullback in animated rings. This pullback is naturally equivalent to
	$W_2^{\ani}(A)$ by~\cite[Rem.~A.10]{bhatt-lurie-prismatization}. For the second part, note that for any derived ring $R$, applying $W_2$ to the counit
	$\tau_{\geq 0}R \to R$ in $\DAlg_\bZ$ yields a natural map $W_2^{\ani}(\tau_{\geq 0}R) \simeq W_2(\tau_{\geq 0}R) \to W_2(R)$ of derived rings. This agrees with the map
	induced by the commutative diagram
	$$\begin{tikzcd}
		\tau_{\geq 0}R \ar[r, "\Fr"] \ar[d] & (\tau_{\geq 0}R) \otimes_\bZ \bF_p \ar[d] &\tau_{\geq 0}R \ar[d] \ar[l, "\can"'] \\[3.5ex]
		R \ar[r, "\Fr"] & R \otimes_\bZ \bF_p & R \ar[l, "\can"']
	\end{tikzcd}$$
	in $\DAlg_\bZ$. Since $\fib(\can_{\tau_{\geq 0}R}) \simeq \tau_{\geq 0}R$ in $\Mod_\bZ$, we obtain a commutative diagram in $\Mod_\bZ$
	$$\begin{tikzcd}
		\tau_{\geq 0}R \ar[r] \ar[d] & W_2^\ani(\tau_{\geq 0}R) \ar[d] \ar[r] & \tau_{\geq 0}R \ar[d] \\[3.5ex]
		R \ar[r] & W_2(R) \ar[r, "w_0"] & R
	\end{tikzcd}$$
	whose rows are fibre sequences in $\Mod_\bZ$. It follows that the cofibre of the map $W_2^\ani(\tau_{\geq 0}R) \to W_2(R)$ is ($-1$)-truncated, hence that
	it induces an equivalence on the connective cover.
\end{proof}

\begin{definition}\label{def:F2}
	We denote by $\eta \colon\id \to F_2^\varphi$ the pointed endofunctor on $\DAlg_\bZ$ which is left adjoint to the map $w_0\colon W_2 \to \id$ from Definition~\ref{def:derived W2}.
\end{definition}

\begin{proposition}[Filtrations on $T^\varphi$ and $C^\varphi$]\label{prop:filtration on Tphi}\phantom{a}\par
	\begin{enumerate}
	\item There exists an $\bN$-indexed exhaustive filtration on $T^\varphi$ by endofunctors $F_i^\varphi \in \End^L(\DAlg_\bZ)$.
	\item There exists an $\bN^{\op}$-indexed co-exhaustive filtration on $C^\varphi$ by endofunctors $W_i \in \End^R_\Sigma(\DAlg_\bZ)$.
	\end{enumerate}
	Here, $F_i^\varphi$ is left adjoint to $W_i$ for each $i \geq 1$ and $F_1^\varphi = W_1 = \id$. Moreover, for any $i \geq 2$, there exists a fiber sequence
	$$\begin{tikzcd}
		U \ar[r] & UW_i \ar[r, "Ur_i"] & UW_{i-1}
	\end{tikzcd}$$
	in $\Fun_\Sigma(\DAlg_{\bZ}, \Mod_{\bZ})$, where $U\colon\DAlg_\bZ \to \Mod_\bZ$ denotes the forgetful functor.
\end{proposition}

\begin{proof}
	It is a formal consequence of Definition~\ref{def:derived W2} that for a given derived ring $R$, the datum $(\phi,h)$ of a derived $\varphi$-algebra structure on $R$ is equivalent to
	the datum 
	$$\begin{tikzcd}
		& W_2(R) \ar[dr, "w_0"] & \\[2.5ex]
		R \ar[ur, "s"] \ar[rr, "\id"] && R
	\end{tikzcd}$$
	of a section of $w_0$ in $\DAlg_{\bZ}$. This defines an equivalence $\DAlg^\varphi_{\bZ} \simeq \cAlg_{W_2}(\DAlg_{\bZ})$ between the category of derived
	$\varphi$-algebras and the category of coalgebras for $w_0\colon W_2 \to \id$ which is compatible with the forgetful functors to $\DAlg_{\bZ}$. Since $W_2$ preserves all limits
	by Remark~\ref{rmk:properties W2}, it follows from the dual of Proposition~\ref{prop:free monad} that the cofree $\varphi$-algebra comonad $C^\varphi$ on $\DAlg_{\bZ}$
	coincides with the cofree comonad on the co-pointed endofunctor $w_0\colon W_2 \to \id$. Dually, Remark~\ref{rmk:duality sections retracts} implies that the free derived
	$\varphi$-algebra monad $T^\varphi$ on $\DAlg_{\bZ}$ is equivalent to the free monad on the functor $\eta\colon \id \to F_2^\varphi$ from Definition~\ref{def:F2}.
	Existence of the $W_i$ and $F_i^\varphi$ now follows from Remark~\ref{rem:free monad explicit} and its dual together with Remark~\ref{rem:adjoint cofree comonad}. To see that
	each of the functors $W_i$ preserves sifted colimits, one uses the pullback formula from Remark~\ref{rmk:pb formula cofree} with $T = (-) \otimes_\bZ \bF_p$ together with the
	respective property of $W_2$ and the fact that sifted colimits commute with finite limits in $\DAlg_\bZ$.
	The second part of the proposition follows from the dual of Lemma~\ref{lem:cofibers of filtration} together with the observation that for $i=2$,
	existence of the fibre sequence follows from the fiber sequence $R\xrightarrow{\cdot p} R\xrightarrow{\can} R\otimes_\bZ \bF_p$ in $\Mod_{\bZ}$.
\end{proof}

\begin{corollary}\label{cor:properties Wi}
	For any $i\geq 1$, the functors $W_i$ from Proposition~\ref{prop:filtration on Tphi} restrict to endofunctors on animated rings.
\end{corollary}

\begin{proof}
	This follows immediately from the fibre sequences of Proposition~\ref{prop:filtration on Tphi} together with the result for $W_2$ from Lemma~\ref{lem:W2 and animated W2}.
\end{proof}

\begin{remark}
	In fact, the functors $W_i$ restrict to the animations of the classical $i$-truncated $p$-typical Witt vector functors on $\DAlg_\bZ^\cn$.
\end{remark}

\begin{lemma}\label{lem: Tphi and ccff}
	The map $(T^\varphi \circ \LSym_\bZ)(\bZ[-n] \to 0)$ is coconnectively faithfully flat for any $n > 0$.
\end{lemma}

\begin{proof}
	Since $T^\varphi \simeq \colim_{i \in \bN} F_i^\varphi$ in $\End^L(\DAlg_\bZ)$ by Proposition~\ref{prop:filtration on Tphi} and the collection of ccff maps is closed under
	filtered colimits, it suffices to show that the maps $(F_i^\varphi \circ \LSym_\bZ)(\bZ[-n] \to 0)$ are ccff for any $n >0$ and for all $i \geq 1$. Thus, fix $n >0$. We first
	observe:
	\begin{itemize}
	\item[($\ast$)] The space of maps from $(F_i^\varphi \circ \LSym_\bZ)(\bZ[-n])$ to $\bZ$ in $\DAlg_\bZ$ is connected for any $i \geq 1$.
	\end{itemize}
	By adjunction, there exists an equivalence $\Map_{\DAlg_\bZ}(F_i^\varphi \LSym_\bZ(\bZ[-n]),\bZ) \simeq \Omega^\infty W_i(\bZ)[n]$ for any $i \geq 1$.
	Since $W_i(\bZ)$ is connective by Corollary~\ref{cor:properties Wi} and $n > 0$ by assumption, it follows that $\pi_0 W_i(\bZ)[n] \simeq 0$ which gives the claim.
	We now argue by induction on $i$: For $i=1$, the map $\LSym_\bZ(\bZ[-n]) \to \bZ$ is ccff by~\cite[Cor.~2.15]{mathew-mondal}.
	For $i \geq 2$, Proposition~\ref{prop:filtration on Tphi} implies that there exists a pushout square
	$$\begin{tikzcd}[column sep=small]
		F_{i-1}^\varphi( \LSym(\bZ[-n])) \ar[d] \ar[r] & F_i^\varphi( \LSym(\bZ[-n])) \ar[d] \\[3ex]
		\bZ \ar[r] & \LSym(\bZ[-n])
	\end{tikzcd}$$
	in $\DAlg_\bZ$ which is adjoint to the fibre sequence for the $W_i$. In particular, this identifies the left hand vertical map with $(F_{i-1}^\varphi \circ \LSym_\bZ)(\bZ[-n] \to 0)$.
	The claim now follows from the fact that ccff maps are closed under base change and composition together with observation ($\ast$).
\end{proof}

\begin{lemma}\label{lem: LSym ccff}
	The map $\LSym_\bZ(\oplus_I \bZ[-n] \to 0)$ is ccff for any $n > 0$ and any index set $I$.
\end{lemma}

\begin{proof}
	By~\cite[Prop.~2.10]{mathew-mondal} it suffices to show that the fibre of the map in question has Tor amplitude in degrees $\leq -1$ over $\bZ$. Since the collection of
	$\bZ$-modules of Tor
	amplitude in degrees $\leq -1$ is closed under filtered colimits, we may assume that $I$ is a finite set. We are thus reduced to proving the claim for a finite tensor product
	of the map $f\colon\LSym_\bZ(\bZ[-n]) \to \bZ$ in $\Mod_\bZ$. This follows by induction: The case $n=1$ is~\cite[Cor.~2.15]{mathew-mondal}. For $n\geq 1$, writing
	$f^{\otimes n} \simeq (f\otimes\id)\circ(\id \otimes f^{\otimes n-1})$ shows that there exists a fibre sequence
	$ \fib(f^{\otimes n-1}) \otimes_\bZ \LSym_\bZ(\bZ[-n]) \to \fib(f^{\otimes n}) \to \fib(f)$ in $\Mod_\bZ$.
	Since both $\fib(f)$ and $\fib(f^{\otimes n-1})$ have Tor-amplitude in degrees $\leq -1$ and $\LSym_\bZ(\bZ[-n])$ has Tor-amplitude in degrees $\leq 0$ in $\Mod_\bZ$, the claim now follows from the fact that $\bZ$-modules of Tor amplitude $\leq k$ are closed under extensions.
\end{proof}

We record the following stronger version of Lemma~\ref{lem: Tphi and ccff}:

\begin{lemma}\label{lem: Tphi and Tor-amplitude of fib}
	For any $n > 0$ and any index set $I$, the fibre of the map $(T^\varphi \circ \LSym_\bZ)(\oplus_I\bZ[-n] \to 0)$ has Tor-amplitude in degrees $\leq -1$ in $\Mod_\bZ$. In particular,
	it is coconnectively faithfully flat.
\end{lemma}

\begin{proof}
	Since $T^\varphi$ preserves all colimits and taking fibres preserves filtered colimits, it suffices to prove the claim for finite $I$. We thus have to show that the fibre
	of the map $((T^\varphi \circ \LSym_\bZ)(\bZ[-n]) \to \bZ)^{\otimes_\bZ k}$ has Tor-amplitude in degrees $\leq -1$ in $\Mod_\bZ$. Arguing by induction on $k$ as in the proof
	of Lemma~\ref{lem: LSym ccff}, it suffices to show the following two claims:
	\begin{itemize}
		\item[($\ast_1$)] For any $n >0$, the fibre of $(T^\varphi \circ \LSym_\bZ)(\bZ[-n]) \to \bZ$ has Tor-amplitude in degrees $\leq -1$ in $\Mod_\bZ$.
		 \item[($\ast_2$)] For any $n > 0$, $(T^\varphi \circ \LSym_\bZ)(\bZ[-n])$ has Tor-amplitude in degrees $\leq 0$ in $\Mod_\bZ$.
	\end{itemize}
	By Lemma~\ref{lem: Tphi and ccff}, we know that the maps $(T^\varphi \circ \LSym_\bZ)(\bZ[-n]) \to \bZ$ are ccff for any $n > 0$. It follows that claim ($\ast_1$) is implied by
	claim ($\ast_2$). We show the second claim by induction on $n \geq 0$: For $n = 0$, Lemma~\ref{lem:restriction to animated delta} implies that $(T^\varphi \circ \LSym_\bZ)(\bZ)$ is
	equivalent to the free discrete $\delta$-ring $\delta\{x\}$ on a single generator. This is flat over $\bZ$. For $n >0$, let $M_n^\bullet$ denote the \v{C}ech nerve of the map
	$\bZ[-n] \to 0$ in $\Mod_\bZ$. Then $M_n^k \simeq \bZ[-n+1]^{\oplus k}$ for any $k\geq 0$ and $(T^\varphi \circ \LSym_\bZ)(M_n^\bullet)$ is the \v{C}ech nerve of the map
	$(T^\varphi \circ \LSym_\bZ)(\bZ[-n]) \to \bZ$ in $\DAlg_\bZ$. Observe that by inductive hypothesis, the individual terms in the cosimplicial diagram
	$(T^\varphi \circ \LSym_\bZ)(M_n^\bullet)$ have Tor-amplitude in degrees $\leq 0$ underlying in $\Mod_\bZ$. In particular, they are coconnective. Thus,
	coconnective faithfully flat descent~\cite[Prop.~2.18]{mathew-mondal} implies that the canonical map
	$(T^\varphi \circ \LSym_\bZ)(\bZ[-n]) \to \Tot(T^\varphi(\LSym_\bZ M_n^\bullet))$ in $\DAlg_\bZ$ is an equivalence. It follows that $(T^\varphi \circ \LSym_\bZ)(\bZ[-n])$ has
	Tor-amplitude in degrees $\leq 0$ in $\Mod_\bZ$ since the collection of $\bZ$-modules of Tor amplitude in degrees $\leq 0$ is closed under arbitrary limits in $\Mod_\bZ$.
\end{proof}

\begin{corollary}\label{cor:Tphi plethory}
	The derived $\varphi$-algebra monad $T^\varphi$ is a derived flat plethory.
\end{corollary}

\begin{proof}
	The monad $T^\varphi$ preserves all colimits by definition. Lemma~\ref{lem:restriction to animated delta} implies that $T^\varphi(\bZ[X]) \simeq T^\delta(\bZ[X]) \simeq \delta\{x\}$
	is equivalent to the free discrete $\delta$-ring on a single generator which is flat as a $\bZ$-module.
	Thus, by Lemma~\ref{lem:ccff condition}, it suffices to show that the maps $(T^\varphi \circ \LSym_\bZ)(\bZ[-n] \to 0)$ are coconnectively faithfully flat for any $n > 0$.
	This is Lemma~\ref{lem: Tphi and ccff}.
\end{proof}

\begin{theorem}[$\delta$-comparison]\label{thm:comparison delta}
    There are equivalences $\s\c\CAlg^\delta[W^{-1}]\we\DAlg^\delta_{\bZ}\we\DAlg^\varphi_{\bZ}$ which are compatible with the forgetful functors to $\DAlg_{\bZ}$.
\end{theorem}

\begin{proof}
	The first equivalence follows from the simplicial-cosimplicial comparison theorem, Theorem~\ref{thm:cosimplicial}.
    The second equivalence follows from Theorem~\ref{thm:deriving-plethories} together with the
    fact that both $T^\varphi$ and $\LSym^\delta_\bZ$ restrict to the animated $\delta$-ring monad
    on $\Mod_{\bZ}^{\cn}$.
\end{proof}

\subsection{Derived $\lambda$-rings}\label{sec:lambda}

The free $\lambda$-ring monad on $\Mod_\bZ^\heart$ is derivable, as observed in~\cite[Sec.~5]{antieau-spherical}. It thus extends to a sifted colimit preserving
monad $\LSym^\lambda_\bZ$ on $\Mod_\bZ$ and one may consider the $\infty$-category $\DAlg^\lambda_\bZ = \LMod_{\LSym^\lambda_\bZ}(\Mod_\bZ)$ of
$\LSym^\lambda_\bZ$-modules in $\Mod_\bZ$.
In this section, we show that the category $\DAlg^\lambda_\bZ$ is equivalent to an $\infty$-category of derived rings equipped with a structure of coherently compatible Frobenius lifts for all primes $p$ (see Definition~\ref{def:derived psi-algebras}). The latter was considered in~\cite{hubner_animated} in the animated case and shown to be equivalent to the animation
$\Pscr_\Sigma(\Poly^\lambda_\bZ) \simeq \DAlg^{\lambda,\cn}_\bZ$ of the $1$-category $\lambda$-rings.\\

Below, we write $\DAlg_\bZ^{\delta_p}$ for the category of derived $p$-typical $\delta$-rings which we identify with derived rings with a single $p$-Frobenius lift by Theorem~\ref{thm:comparison delta}.

\begin{notation}
	For a fixed prime $p$, we denote by $\bZ_{>0}^{p\nmid}$ the multiplicative monoid of natural numbers $n \in \bZ_{>0}$ which are not divisible by $p$. Prime factor
	decomposition determines an equivalence $\bZ_{>0}^{p\nmid} \simeq \oplus_{q\neq p}\bN$ in $\Mon_{\bE_\infty}(\Ani)$. In particular, there exist equivalences
	$B\bZ_{>0} \simeq B\bN \times B\bZ_{>0}^{p\nmid}$ in $\iCat$ for any prime $p$.
\end{notation}

\begin{definition}\label{def:derived psi-algebras}
	We define the category of \textit{derived $\psi$-algebras} as the pullback
	$$\begin{tikzcd}[column sep=4em]
		\DAlg^\psi_\bZ \ar[r] \ar[d] \ar[dr, phantom, "\lrcorner", near start] & \prod_p\Fun(B\bZ_{>0}^{p\nmid},\DAlg_{\bF_p}) \ar[d, "\prod_p \Fr_p"] \\[3ex]
		\Fun(B\bZ_{>0},\DAlg_\bZ) \ar[r, "\prod_p (-)\otimes_\bZ\bF_p"] &\prod_p\Fun(B\bZ_{>0},\DAlg_{\bF_p})
	\end{tikzcd}$$
	in $\iCat$.
\end{definition}

The pullback in Definition~\ref{def:derived psi-algebras} is both a pullback in $\PrL$ and $\PrR$. Since the $p$-Frobenius $\Fr_p$ is conservative, it follows that the forgetful functor
$$U^\psi \colon \DAlg^\psi_{\bZ}\rightarrow\Fun(B\bZ_{>0},\DAlg_\bZ)\rightarrow\DAlg_\bZ$$
is conservative and preserves both small limits and colimits. In particular, it exhibits the category of derived $\psi$-algebras as monadic over $\DAlg_\bZ$ with associated colimit-preserving
monad $T^\psi$.

\begin{remark}
	By definition, any object in $\DAlg^\psi_\bZ$ has an underlying derived ring $R$ which is equipped with commuting endomorphisms $\psi^k \in \End_{\DAlg_\bZ}(R)$
	for each positive integer $k$. For each prime $p$, $\psi^p$ is a lift of the $p$-Frobenius on $R \otimes_\bZ \bF_p$, and for any $k > 0$ and any prime $p$, $\psi^k$ is an
	endomorphism of the derived $p$-typical $\delta$-ring $(R,\psi^p)$. Note that the category $\DAlg^\psi_\bZ$ is not equivalent to the $\infty$-category defined
	in~\cite[Rem.~5.6]{antieau-spherical}.
\end{remark}

\begin{remark}\label{rmk: derived lambda via delta}
	Considering the definition of derived $\delta$-rings (Definition~\ref{def:derived phi-algebras}) one sees that the category of derived $\psi$-algebras admits an equivalent
	definition in terms of the following pullback square
	$$\begin{tikzcd}
		\DAlg^\psi_\bZ \ar[r] \ar[d] \ar[dr, phantom, "\lrcorner", near start] & \prod_p \Fun(B\bZ_{>0}^{p\nmid},\DAlg^{\delta_p}_\bZ) \ar[d] \\[3ex]
		\Fun(\B\bZ_{>0},\DAlg_\bZ) \ar[r] & \prod_p \Fun(\B\bZ_{>0},\DAlg_\bZ)
	\end{tikzcd}$$
	in $\iCat$. Here, the right vertical arrow is induced by the forgetful functors $\DAlg^{\delta_p}_\bZ \to \Fun(\B\bN,\DAlg_\bZ)$ and the bottom horizontal arrow is given
	by the diagonal. Note that all functors in the diagram are conservative and preserve both limits and colimits.
\end{remark}

\begin{lemma}\label{lem: restriction to animated lambda}
	The derived $\psi$-algebra monad $T^\psi$ restricts to the animated $\lambda$-ring monad $T^\lambda$ on $\Ani_\bZ$. 
\end{lemma}

\begin{proof}
	We claim that the inclusion $\Ani_\bZ \hookrightarrow\DAlg_\bZ$ induces a horizontally right adjointable pullback square
	$$\begin{tikzcd}
		\Ani_\bZ^{\lambda} \ar[r,hook] \ar[d, "U^\lambda"] & \DAlg_\bZ^{\psi} \ar[d, "U^\psi"] \\[3ex]
		\Ani_\bZ \ar[r,hook] & \DAlg_\bZ
	\end{tikzcd}$$
	in $\iCat$. Adjointability then shows that there exists an equivalence $F^\psi \circ \iota \simeq \iota \circ F^\lambda$ in $\Fun^L(\Ani_\bZ,\DAlg_\bZ)$ and thus
	an equivalence $T^\psi \circ \iota \simeq \iota \circ T^\lambda$ between the underlying endofunctors of the associated monads. Thus, observe that by
	Lemma~\ref{lem:restriction to animated delta}, the inclusion of animated into derived rings induces a commutative diagram
	$$\begin{tikzcd}
		(\DAlg_\bZ)^{B\bZ_{>0}} \ar[r, "\Delta"] & \prod_p (\DAlg_\bZ)^{B\bZ_{>0}} & \prod_p (\DAlg_\bZ^{\delta_p})^{B\bZ_{>0}^{p\nmid}} \ar[l, "\prod_p U^{\delta_p}"']\\[3ex]
		(\Ani_\bZ)^{B\bZ_{>0}} \ar[u,hook] \ar[r, "\Delta"] & \prod_p (\Ani_\bZ)^{B\bZ_{>0}} \ar[u,hook] & \prod_p (\Ani_\bZ^{\delta_p})^{B\bZ_{>0}^{p\nmid}} \ar[u, hook] \ar[l, "\prod_p U^{\delta_p}_{\ani}"']
	\end{tikzcd}$$
	in $\iCat$, whose squares are vertically right adjointable. This induces a fully faithful left adjoint $\Ani_\bZ^\lambda \hookrightarrow \DAlg_\bZ^\psi$ on the limit which
	sits in a horizontally adjointable square
	$$\begin{tikzcd}[column sep=small]
		\Ani_\bZ^{\lambda} \ar[r,hook] \ar[d, "U^\lambda"] & \DAlg_\bZ^{\psi} \ar[d, "U^\psi"] \\[3ex]
		(\Ani_\bZ)^{B\bZ_{>0}} \ar[r,hook] & (\DAlg_\bZ)^{B\bZ_{>0}}
	\end{tikzcd}$$
	in $\iCat$. Here, we used~\cite[Thm.~6.3]{hubner_animated} to identify the lower horizontal pullback with the category $\Ani_\bZ^\lambda$ of animated $\lambda$-rings.
	The claim now follows by pasting with the commutative square
		$$\begin{tikzcd}[column sep=small]
		(\Ani_\bZ)^{B\bZ_{>0}} \ar[r,hook] \ar[d, "\ev"] & (\DAlg_\bZ)^{B\bZ_{>0}} \ar[d, "\ev"] \\[3ex]
		\Ani_\bZ \ar[r,hook] & \DAlg_\bZ
	\end{tikzcd}$$
	which is evidently horizontally right adjointable.
\end{proof}

Observe that $\bZ_{(p)}$ admits a unique structure of a $\lambda$-ring for any prime $p$. Since animated $\lambda$-rings embed into derived $\psi$-algebras by
Lemma~\ref{lem: restriction to animated lambda}, it in particular defines an object in $\DAlg_\bZ^\psi$. In fact, $\bZ_{(p)}$ is an idempotent algebra in the category of
derived $\psi$-algebras with respect to the cocartesian symmetric monoidal structure, as can be checked on underlying objects in $\Mod_\bZ$.

\begin{definition}
	Let $p$ be a prime. We define the category $\DAlg^\psi_{\bZ_{(p)}}$ of derived $p$-local $\psi$-algebras as the slice $(\DAlg_\bZ^\psi)_{\bZ_{(p)}/}$.
\end{definition}

\begin{remark}\label{rmk: p-local psi-algebras}
	Since $\bZ_{(p)}$ is an idempotent algebra in derived $\psi$-algebras, it follows that the category $\DAlg_{\bZ_{(p)}}^\psi$ of derived $p$-local $\psi$-algebras is a
	localization of $\DAlg_\bZ^\psi$. There exists a pullback square
	$$\begin{tikzcd}
		\DAlg_{\bZ_{(p)}}^\psi \ar[r,hook] \ar[d, "U^\psi_p"]& \DAlg_\bZ^\psi \ar[d, "U^\psi"] \\[3ex]
		\DAlg_{\bZ_{(p)}} \ar[r,hook] & \DAlg_\bZ
	\end{tikzcd}$$
	in $\PrR$ which is horizontally left adjointable since $U^\psi \circ (-)\otimes_\bZ \bZ_{(p)} \simeq U^\psi(-) \otimes_\bZ \bZ_{(p)}$ in $\Fun^L(\DAlg_\bZ^\psi,\DAlg_\bZ)$.
	We denote the monad on derived $\bZ_{(p)}$-algebras associated to $U^\psi_p$ by $T^\psi_p$. This preserves all colimits and satisfies
	 $(-)\otimes_\bZ\bZ_{(p)} \circ T^\psi \simeq T^\psi_p \circ (-)\otimes_\bZ \bZ_{(p)}$ in $\Fun^L(\DAlg_\bZ,\DAlg_{\bZ_{(p)}})$.
\end{remark}

\begin{lemma}\label{lem: p-local psi-algebras and delta}
	 For any prime $p$, there exists an equivalence $\DAlg^\psi_{\bZ_{(p)}} \simeq \Fun(B\bZ_{>0}^{\nmid p},\DAlg_{\bZ_{(p)}}^{\delta_p})$ over $\DAlg_{\bZ_{(p)}}$.
\end{lemma}

\begin{proof}
	This follows immediately from the pullbacks from Remark~\ref{rmk: p-local psi-algebras} and Definition~\ref{def:derived psi-algebras} since the tensor products
	$R\otimes_\bZ\bF_q$ vanish for any derived $\bZ_{(p)}$-algebra $R$ for any prime $q \neq p$.
\end{proof}

\begin{proposition}\label{prop: lambda plethory}
	The derived $\psi$-algebra monad $T^\psi$ is a derived flat plethory.
\end{proposition}

\begin{proof}
	The monad $T^\psi$ preserves all colimits by definition. Lemma~\ref{lem: restriction to animated lambda} implies that $T^\psi(\bZ[X])$ is equivalent
	to $T^\lambda(\bZ[X])$ which is given by the free discrete $\lambda$-ring on a single generator. This is flat as a $\bZ$-module.
	Thus, by Lemma~\ref{lem:ccff condition} it suffices to show that the maps $(T^\psi \circ \LSym_\bZ)(\bZ[-n] \to 0)$ are coconnectively faithfully flat for any $n > 0$.
	Let us denote the fibre of the above map in $\Mod_\bZ$ by $F_n$. By~\cite[Prop.~2.10]{mathew-mondal}, it suffices to show that $F_n$ has Tor-amplitude in degrees
	$\leq -1$ as a $\bZ$-module. By~\cite[Lem.~15.67.15]{stacks} this can be checked $p$-locally, i.e.~it is sufficient to see that the fibre $F_{p,n} \simeq \bZ_{(p)}\otimes_\bZ F_n$
	of $\bZ_{(p)} \otimes_\bZ (T^\psi \circ \LSym_\bZ)(\bZ[-n] \to 0)$ has Tor-amplitude in degrees $\leq -1$ in $\Mod_{\bZ_{(p)}}$ for all primes $p$.
	Thus, fix a prime $p$. Then there exist equivalences
	\begin{align*}
		\bZ_{(p)} \otimes_\bZ (T^\psi \circ \LSym_\bZ)(\bZ[-n] \to 0) &\simeq (T^\psi_p \circ \LSym_{\bZ_{(p)}})(\bZ_{(p)}[-n] \to 0) \\
		&\simeq (T^{\delta_p}_p \circ \LSym_{\bZ_{(p)}})(\oplus_{m\in \bZ_{>0}^{p\nmid}} \bZ_{(p)}[-n] \to 0) \\
		&\simeq \bZ_{(p)} \otimes_\bZ (T^{\delta_p} \circ \LSym_\bZ)(\oplus_{m \in \bZ_{>0}^{p\nmid}}\bZ[-n] \to 0)
	\end{align*}
	in $\DAlg_{\bZ_{(p)}}$. Here, the first equivalence is by Remark~\ref{rmk: p-local psi-algebras} and the second follows from Lemma~\ref{lem: p-local psi-algebras and delta}.
	Now the fibre of $(T^{\delta_p} \circ \LSym_\bZ)(\oplus_{m \in \bZ_{>0}^{p\nmid}}\bZ[-n] \to 0)$ has Tor amplitude in degrees $\leq -1$ in $\Mod_\bZ$ by
	Lemma~\ref{lem: Tphi and Tor-amplitude of fib} and this is preserved by extension of scalars to $\Mod_{\bZ_{(p)}}$. 
\end{proof}

\begin{theorem}[$\lambda$-comparison]\label{thm:lambda-comparison}
    There are equivalences $\s\c\CAlg^\lambda[W^{-1}]\we\DAlg^\lambda_\bZ\we\DAlg^\psi_\bZ$ which are compatible with the forgetful functors to $\DAlg_{\bZ}$.
\end{theorem}

\begin{proof}
	The first equivalence follows from the simplicial-cosimplicial comparison theorem, Theorem~\ref{thm:cosimplicial}.
    The second equivalence follows from Theorem~\ref{thm:deriving-plethories} together with the
    fact that both $T^\psi$ and $\LSym^\lambda_\bZ$ restrict to the animated $\lambda$-ring monad
    on $\Mod_{\bZ}^{\cn}$.
\end{proof}


\subsection{Binomial rings}

In~\cite{kubrak-shuklin-zakharov}, the authors show that the free binomial monad $\Bin$ can be derived using the filtered
monad approach. This leads to the $\infty$-category $\DAlg^\Bin$ of derived binomial rings.
In~\cite{antieau-spherical}, one instead uses derived $\lambda$-rings and canonically trivializes
the Adams operations, leading to the $\infty$-category $\DAlg^{\psi=1}$ of binomial derived
commutative $\bZ$-algebras. There is also work of Horel~\cite{horel} studying cosimplicial
binomial rings which embed fully faithfully into simplicial-cosimplicial binomial rings. All three sources realize finite type nilpotent anima as full subcategories of the respective $\infty$-categories of derived binomial rings. In this section, we show that all three models of derived binomial rings agree.

Consider the category $\DAlg^\psi_{\bZ}\simeq \DAlg^\lambda_{\bZ}$ of derived $\lambda$-rings. By Corollary \ref{cor:center of DAlg_T is discrete} it has a natural action of $\B\bZ_{> 0}\simeq \B(\oplus_p\bN)$ given by the Adams operations $\psi_p$ for all primes $p$. Let us denote the corresponding invariants by
$$
\DAlg^{\psi_\star=1}_{\bZ}:= (\DAlg^\lambda_{\bZ})^{\h\B\bZ_{> 0}^\times}.
$$
By Proposition \ref{prop:categorical invariants are monadic} and Remark \ref{rem:invariants are also comonadic}, $\DAlg^{\psi_\star =1}_\bZ$ is both monadic and comonadic over $\DAlg^\lambda_\bZ$. We can consider the composition
$$
\DAlg^{\psi_\star =1}_\bZ \rightarrow \DAlg^\lambda_\bZ \rightarrow \DAlg_\bZ.
$$ 
This functor is conservative and commutes with all small limits and colimits. Let $T^{\lambda, \psi_\star=1}$ be the corresponding monad on $\DAlg_\bZ$. Let us also denote by $\LSym^{\psi_\star =1}_\bZ$ the composition of $T^{\lambda, \psi_\star=1}$ and $\LSym_\bZ$.

\begin{proposition} \label{prop:binomial plethory}
The following assertions hold.
\begin{enumerate}\item The functor 
	$$
		\DAlg^{\psi_\star =1}_\bZ \rightarrow \Mod_{T^{\lambda, \psi_\star=1}}(\DAlg_\bZ)
		$$
		is an equivalence; in other words, $\DAlg^{\psi_\star =1}_\bZ$ is monadic over $\DAlg_\bZ$.
		\item $T^{\lambda, \psi_\star=1}$ is a derived flat plethory. 
		\end{enumerate}
\end{proposition}	
\begin{proof}
	Part 1 follows from the Barr--Beck--Lurie theorem and the fact that the forgetful functor  $\DAlg^{\psi_\star =1}_\bZ \rightarrow \DAlg_\bZ$ is conservative and commutes with all small (in particular sifted) colimits. For part 2 we need to check conditions (1)-(3) in Definition \ref{def:derived_plethory}. We already know that $T^{\lambda, \psi_\star=1}$ commutes with all small colimits since the forgetful functor $\DAlg^{\psi_\star =1}_\bZ\to \DAlg_\bZ$ does. For (2) it is enough to check that $\LSym^{\psi_\star =1}_\bZ(\bZ)$ is a torsion-free abelian group.  By Example \ref{ex: invariants of actions}, $\LSym^{\psi_\star =1}_\bZ(\bZ)$ is computed by applying iterated coequalizers of Adams operations to the free $\lambda$-ring $\LSym^{\lambda}_\bZ(\bZ)$. The latter is identified with functions on the big Witt group scheme $\mathbf W$, and under this identification the $p$-th Adams operation $\psi_p$ corresponds to the Frobenius endomorphism $F_p$ of $\mathbf W$. Let us order the prime numbers in some way $k\mapsto p_k$, and denote $\psi_k:=\psi_{p_k}$, $F_k:=F_{p_k}$. Following \cite{ToenSchem}, consider the sequence of group schemes over $\bZ$
	$$
	\mathbf H \subset \ldots \subset \mathbf H_k \subset \mathbf H_{k-1} \subset \ldots \subset \mathbf H_2 \subset \mathbf H_1 \subset \mathbf H_0 =:\mathbf W,
	$$
	defined inductively as $\mathbf H_k:= \ker(F_{k}-\id \colon \mathbf H_{k-1} \rightarrow \mathbf H_{k-1})$ and where $\mathbf H:=\lim_k \mathbf H_k$ is the Hilbert additive group scheme $\Spec \Bin(\mathbf Z)$. In other words, $\mathbf H_k\subset \mathbf W$ is the subgroup of $\mathbf W$ fixed by the first $k$ Frobenius endomorphisms of $\mathbf W$. As shown in \cite[Sec.~3]{ToenSchem}, each map $F_{_k}-\id \colon \mathbf H_{k-1} \rightarrow \mathbf H_{k-1}$ is a flat surjection; in particular, each $\mathbf H_k$ and $\mathbf H$ are flat over $\bZ$. This shows that we have an equivalence
	$$
	(\mathrm{coeq}_{\psi_k,\id}\circ \mathrm{coeq}_{\psi_{k-1},\id}\circ \ldots \circ \mathrm{coeq}_{\psi_{1},\id})(\LSym^{\lambda}_\bZ(\bZ))\simeq \mathcal{O}(\mathbf H_n).
	$$
	Passing to the colimit over $k\in \bN$, we obtain equivalences
	$$\LSym^{\psi_\star =1}_\bZ(\bZ)\simeq \mathcal O(\mathbf H)\simeq \Bin(\bZ),$$
	which show that $\LSym^{\psi_\star =1}_\bZ(\bZ)$ is a free $\bZ$-module. 
	
	It remains to verify (3). By Lemma~\ref{lem:ccff condition} it is enough to show that the natural map $\LSym^{\psi_\star =1}_\bZ(\bZ[-n])\to \bZ$ is ccff for all $n\ge 1$. Since filtered colimits of ccff maps are ccff, it is enough to show that the 
	map 
	$$
	\LSym^{\psi_{\le k} =1}_\bZ(\bZ[-n]):=(\ldots(((\LSym^{\lambda}_\bZ(\bZ[-n]))^{\psi_1=1})^{\psi_2=1})\ldots)^{\psi_k=1}\to \bZ
	$$ is ccff for any $k\ge 0$. Passing to the corresponding affine stacks, the map in question is induced by the map $\Spec \bZ \to (\B^n\mathbf W)^{F_{\le k}=1}$, where the target is the iterated fiber of the maps $F_i-\id$ for $i=1,\ldots,k$. Since $F_{i+1}-\id\colon \mathbf H_i \to\mathbf H_i$ is a surjective flat map for any $i$, we get by induction that $(\B^n\mathbf W)^{F_{\le k}=1}\simeq \B^n\mathbf H_{k}$. By \cite[Prop.~2.10, Rem.~3.17]{mathew-mondal} the map $R\Gamma(\B^n\mathbf H_{k},\mathcal O) \to \bZ$ is ccff for any $n\ge 1$.
\end{proof}

Let $\CAlg^{\Bin, \heart}$ be the 1-category of classical binomial rings. By definition, $\CAlg_{\heart}^{\Bin}\subset \CAlg^{\heartsuit}_{\bZ}$ is the full subcategory spanned by torsion-free commutative rings $A$ that are binomially closed: in other words, for any $a\in A$ and $n\in\bN$ the binomial coefficient ${a\choose n}\in A\otimes_{\bZ}\bQ$ also lies in $A$ (see \cite[Section 1]{elliott}).
\begin{theorem}[Binomial comparison]\label{thm:Bin comparison}
    There are equivalences $\s\c\CAlg^\Bin[W^{-1}]\we\DAlg_\bZ^\Bin\we\DAlg^{\psi=1}_\bZ$.
\end{theorem}

\begin{proof} The forgetful functor $\DAlg_\bZ^\Bin\to \DAlg_\bZ$ is monadic: indeed, it is conservative and commutes with all small limits and colimits by \cite[Cor.~2.66]{kubrak-shuklin-zakharov}. Let $T$ be the corresponding monad on $\DAlg_\bZ$; the corresponding monad $\LSym^T$ on $\Mod_{\bZ}$ is identified with the derived binomial monad $\LBin$ from \cite[Def.~2.56]{kubrak-shuklin-zakharov}.  The monad $T$ is a derived flat plethory: indeed, it commutes with small colimits, $T(\mathbb Z[t])\simeq \LBin(\bZ)\simeq \bZ[{t \choose n}]_{n\ge 1}$ is flat, and the condition (3) from Definition \ref{def:derived_plethory}  follows from \cite[Prop.~2.72]{kubrak-shuklin-zakharov}. Also, via \cite[Prop.~2.31, Rem.~2.27]{kubrak-shuklin-zakharov} one identifies the corresponding classical category $\CAlg_k^{T,\heartsuit}$ with $\CAlg^{\Bin}$. 
    The first equivalence then follows from Theorem~\ref{thm:cosimplicial}. 
    
    Since $T^{\lambda, \psi_\star=1}$ is a derived flat plethory by Proposition~\ref{prop:binomial plethory}, for the second equivalence it is enough to identify $\CAlg^{\Bin,\heartsuit}$ and $\CAlg_\bZ^{T^{\lambda, \psi_\star=1},\heartsuit}$ (as categories over $\CAlg_\bZ^\heartsuit$). We have a natural identification $\CAlg_\bZ^{T^{\lambda, \psi_\star=1},\heartsuit}\simeq (\CAlg^{\lambda,\heartsuit}_\bZ)^{\psi_\star=1}$; by Remark \ref{rem:fixed points for 1-categories} the target is naturally identified the 1-category of $\lambda$-rings all of whose Adams operations are the identity map. By a theorem of Elliott  \cite[Thm.~9.1(1)]{elliott}, this category is equivalent $\CAlg^{\Bin,\heartsuit}$.
\end{proof}

\begin{remark}
	We note that the category $\DAlg^{\psi=1}_\bZ$ agrees with the category of binomial derived rings in \cite[Definition 5.10]{antieau-spherical}, defined as $B\mathbf Q_{>0}$-invariants of the category of \textit{perfect} $\lambda$-rings. This follows from Remark \ref{rem:one can first perfect and then take invariants}.
\end{remark}	

\begin{remark}
    One can consider the analogue of binomial rings but for derived commutative rings equipped with a single
    endomorphism as we did in Section~\ref{sec:endomorphisms}. Thus, one can consider
    $$\left(\DAlg_\bZ^{\B\bN}\right)^{\h\B\bN}.$$ However, this category is
    equivalent to $\DAlg_\bZ$ via the natural free functor.
\end{remark}


%
%
%

\subsection{Boolean rings}

In classical algebra, a $p$-Boolean ring is a commutative $\bF_p$-algebra for
which the Frobenius endomorphism is the identity. This definition naturally extends to derived commutative $\bF_p$-algebras:

\begin{definition}\label{def:derived boolean}
    We define the category of \textit{$p$-Boolean derived rings} as $\DAlg_{\bF_p}^{\varphi=1}=(\DAlg_{\bF_p})^{\h\B\bN}$.
    Here, the $B\bN$-action is given by the derived $p$-Frobenius.
\end{definition}

Recall from~\cite[Thm.~4.17]{antieau-spherical} that there exists an equivalence $\Pro(\Sscr_{p\fin})^\op\we\DAlg_{\bF_p}^{\varphi=1}$ which identifies
$X \in \Sscr_{p\fin}$ with the $\bF_p$-cochains $\bF_p^X$.

\begin{proposition}\label{prop:boolean derivable}
	Consider the inclusion of the full subcategory of $p$-Boolean rings in discrete $\bF_p$-algebras
	\[
	\CAlg_{\bF_p}^{\Bool, \heartsuit}\hookrightarrow \CAlg_{\bF_p}^\heartsuit.
	\]
	This inclusion is monadic and the associated monad $\Bool^{\heart}$ is a flat unipotent plethory.
\end{proposition}

\begin{proof} 
By Remark \ref{rem:fixed points for 1-categories}, there is a natural identification
\[
\CAlg_{\bF_p}^{\Bool, \heartsuit}\simeq (\CAlg_{\bF_p}^\heartsuit)^{\h\B\bN}
\]
with the invariants for the $hB\bN$-action on $\CAlg_{\bF_p}^\heartsuit$ given by the Frobenius. Proposition \ref{prop:categorical invariants are monadic} and its opposite then imply that the inclusion is both monadic and comonadic. The corresponding monad $\Bool^{\heart}$ is therefore a plethory. Explicitly, it sends a commutative $\bF_p$-algebra $A$ to the coequalizer $\mathrm{coeq}_{\varphi_A,\id}(A)$ of the Frobenius and the identity (see Example \ref{ex: invariants of actions}). The group scheme associated to the plethory $\Bool^{\heart}$ is thus given by the discrete group
$$
R_{\Bool^{\heartsuit}} = \Spec(\bF_p[x]/(x^p-x)) \simeq \underline{\bZ/p}.
$$
This is evidently flat, and unipotent because $\underline{\bZ/p}\cong \fib(\varphi-\id\colon \Ga\to \Ga)$ is a limit of unipotent group schemes.
\end{proof}

\begin{definition}
	Let us write $\Bool$ for the flat derived plethory on $\DAlg_{\bF_p}$ which corresponds to $\Bool^{\heart}$ under Proposition~\ref{prop:image restriction plethories}.
	We define the category $\DAlg_{\bF_p}^{\Bool}$ of \textit{derived $p$-Boolean rings} as $\Mod_{\Bool}(\DAlg_{\bF_p})$.
\end{definition} 

\begin{remark}
	The free $p$-Boolean ring monad on $\Mod_{\bF_p}^\heart$ is compatible with sifted colimits and
takes an $n$-dimensional $\bF_p$-vector space to $$(\bF_p[x]/(x^p-x))^{\otimes n}\we\bF_p^{\oplus p^n}.$$
Note that a polynomial endofunctor of $\Mod_{\bF_p}^{\heart,\omega}$ must have polynomial growth,
whereas this one is exponential. It follows that the free $p$-Boolean ring monad is not directly
derivable using the polynomial approach. There is also no natural filtration on the monad. 
Nonetheless, Proposition \ref{prop:boolean derivable} shows that it is right extendable; one can also verify this directly using the Eilenberg--Moore spectral sequence.
\end{remark}

%
%

\begin{theorem}[Boolean comparison]\label{thm:boolean comparison}
    There are equivalences
    $\s\c\CAlg_{\bF_p}^{\Bool}[W^{-1}]\we\DAlg_{\bF_p}^\Bool\we\DAlg_{\bF_p}^{\varphi=1}$.
\end{theorem}

\begin{proof}
    The first equivalence follows directly from Theorem~\ref{thm:cosimplicial}. 
    For the second equivalence, let us first note that the forgetful functor $\DAlg_{\bF_p}^{\varphi=1}\to \DAlg_{\bF_p}$ is monadic and comonadic by Proposition \ref{prop:categorical invariants are monadic}.
    By Theorem~\ref{thm:deriving-plethories}, it suffices to show that the corresponding monad $T^{\varphi=1}$ is a derived plethory whose underlying classical plethory is equivalent to $\Bool^{\heartsuit}$.
    Condition (1) of Definition~\ref{def:derived_plethory} follows from $\DAlg_{\bF_p}^{\varphi=1}\to \DAlg_{\bF_p}$ preserving colimits.
    For condition (2), we need to check that $T^{\varphi=1}(\bF_p[x])$ is discrete. 
    This follows from the pushout squares in $\DAlg_{\bF_p}$
    $$\begin{tikzcd}
    {\bF_p[x]}\arrow[r, "x-x'"]\arrow[d, "x\mapsto 0"{swap}] & {\bF_p[x, x']}\arrow[r, "{(x, x^p)}"]\arrow[d, "{(y, y)}"] & {\bF_p[x]}\arrow[d]\\
    \bF_p\arrow[r] & {\bF_p[y]}\arrow[r] & {T^{\varphi=1}(\bF_p[x])}
    \end{tikzcd}$$
    where the right square is a pushout by Example~\ref{ex: invariants of actions}  and the top horizontal composite is flat.  
    It remains to show (3). By Lemma~\ref{lem:ccff condition} it is enough to show that the natural map $\LSym^{\varphi =1}_{\bF_p}(\bF_p[-n])\to \bF_p$ is ccff for all $n\ge 1$. 
    Passing to corresponding affine stacks, the map in question is induced by the map $\Spec \bF_p \to (\B^n\Ga)^{\varphi=1}$, where the target is the fiber of the map $\varphi-\id\colon \B^n\Ga\to\B^n\Ga$. Since $\varphi-\id\colon \Ga \to\Ga$ is a surjective flat map we get by induction that $(\B^n\Ga)^{\varphi=1}\simeq B^n\underline{\bZ/p}$. By \cite[Prop.~2.10, Rem.~3.17]{mathew-mondal} the map $R\Gamma(B^n\underline{\bZ/p},\mathcal O) \to \bF_p$ is ccff for any $n\ge 1$.
\end{proof}

\bibliographystyle{amsplain}
\bibliography{dca}
\addcontentsline{toc}{section}{References}

\medskip
\noindent
\textsc{Department of Mathematics, Northwestern University}\\
{\ttfamily antieau@northwestern.edu}

\medskip
\noindent
\textsc{FB Mathematik und Informatik, Universit\"at M\"unster}\\
{\ttfamily ehuebner@uni-muenster.de}

\medskip
\noindent
\textsc{CNRS, LMO, Universit\'e Paris-Saclay}\\
{\ttfamily dmkubrak@gmail.com}

\medskip
\noindent
\textsc{Institut de Mathématiques de Toulouse, Université de Toulouse}\\
{\ttfamily joost.nuiten@math.univ-toulouse.fr}

\end{document}